\documentclass[a4paper,reqno,11pt]{amsart}
\usepackage{amssymb,amsfonts,amsmath,amsthm,graphicx}
\usepackage{amsrefs}
\usepackage{tikz}
\usepackage{pgfplots}
\usepackage{txfonts}
\usepackage{subfigure}
\usepackage{mathrsfs}
\usepackage{mathtools}
\usepackage[normalem]{ulem}
\usepackage[top=3cm, bottom=2.5cm, left=3cm, right=4cm, headsep=14pt]{geometry}
\numberwithin{equation}{section}
\numberwithin{figure}{section}
\newtheorem{thm}{Theorem}[section]

\newtheorem{corollary}[thm]{Corollary}
\newtheorem{definition}[thm]{Definition}

\newtheorem{lemma}[thm]{Lemma}
\newtheorem{notation}[thm]{Notation}
\newtheorem{prop}[thm]{Proposition}
\newtheorem{proposition}[thm]{Proposition}
\newtheorem{rmrk}[thm]{Remark}
\newtheorem{remark}[thm]{Remark}
\newtheorem{theorem}[thm]{Theorem}

\newcommand{\baru}{\bar{u}}
\newcommand{\calc}{\mathcal{B}}
\newcommand{\cala}{\mathcal{A}_{-}}
\newcommand{\calb}{\mathcal{A}}
\newcommand{\cald}{\mathcal{C}}
\newcommand{\cale}{\mathcal{C}_{+}}

\newcommand{\de}{\mathit{\Delta}E_2^{\phi,\,\xi}}
\newcommand{\deone}{\mathit{\Delta}E_1^{\phi,\,\xi}}

\newcommand{\ehat}{  \hat{\mathcal{E}}_\phi}
\newcommand{\einf}{\mathrm{ess}\,\mathrm{inf}}
\newcommand{\ephi}{\mathcal{E}_{\phi}^\xi}

\newcommand{\eps}{\varepsilon}
\newcommand{\esup}{\mathrm{ess}\,\mathrm{sup}}
\newcommand{\Ew}{\mathit{\Delta}E_{\omega,\phi}^\xi}
\newcommand{\ez}{\mathcal{E}_0^\xi}
\newcommand{\neps}{ \mathcal{N}_\eps(\um)}

\newcommand{\per}{{\rm Per}}

\newcommand{\LM}{\mathcal{L}}
\newcommand{\HM}{\mathcal{H}}
\newcommand{\tphild}{\mathbb{T}_{\phi L}^{d}}
\newcommand{\R}{\mathbb{R}}

\newcommand{\sgn}{{\rm sgn}}
\newcommand{\sigmad}{\sigma _d}

\newcommand{\tphil}{\mathbb{T}_{\phi L}}

\newcommand{\uo}{u_{\omega,\phi}}
\newcommand{\uws}{u_{{\omega_*},\phi}}
\newcommand{\umo}{u_\omega^m}
\newcommand{\um}{u_{m,\phi}}
\newcommand{\us}{u_{s,\phi}}

\newcommand{\ust}{\tilde{u}_{s,\phi}}

\newcommand{\xid}{\tilde{\xi}_d}
\newcommand{\intfil}{\int_{\tphil}}
\renewcommand{\leq}{\leqslant}
\renewcommand{\geq}{\geqslant}

\def\Xint#1{\mathchoice
{\XXint\displaystyle\textstyle{#1}}%
{\XXint\textstyle\scriptstyle{#1}}%
{\XXint\scriptstyle\scriptscriptstyle{#1}}%
{\XXint\scriptscriptstyle\scriptscriptstyle{#1}}%
\!\int}
\def\XXint#1#2#3{{\setbox0=\hbox{$#1{#2#3}{\int}$}
\vcenter{\hbox{$#2#3$}}\kern-.5\wd0}}

\def\dashint{\Xint-}

\begin{document}

\title[Critical points of the Cahn-Hilliard energy]
{Existence and properties of certain critical points of the Cahn-Hilliard energy}
\vspace{1in}
\author{Michael Gelantalis}
\address{Michael Gelantalis, RWTH Aachen University}
\email{gelantalis@math1.rwth-aachen.de}
\author{Alfred Wagner}
\address{Alfred Wagner, RWTH Aachen University}
\email{wagner@instmath.rwth-aachen.de}
\author{Maria G. Westdickenberg}
\address{Maria G. Westdickenberg, RWTH Aachen University}
\email{maria@math1.rwth-aachen.de}
\begin{abstract}
The Cahn-Hilliard energy landscape on the torus is explored in the critical regime of large system size and mean value close to $-1$. Existence and properties of a ``droplet-shaped'' local energy minimizer are established. A standard mountain pass argument leads to the existence of a saddle point whose energy is equal to the energy barrier, for which a quantitative bound is deduced. In addition, finer properties of the local minimizer and appropriately defined constrained minimizers are deduced. The proofs employ the $\Gamma$-limit (identified in a previous work), quantitative isoperimetric inequalities, variational arguments, and Steiner symmetrization.
\end{abstract}
\maketitle
\section{Introduction}\label{S:intro}
In this paper we explore the infinite dimensional energy landscape associated to the Cahn-Hilliard \cite{CH} energy
\begin{align}\label{Energy}
E(u):=\int _{\Omega}\frac{1}{2}|\nabla u|^2+G(u)\,dx,
\end{align}
where  $G$ is a double-well potential, $\Omega\subset\R^d$ for $d\geq 2$, and the functions $u$ belong to
\begin{align}\notag
\Big\{u\in H^1\cap L^4(\Omega): \dashint_\Omega u \, dx=m\Big\},
\end{align}
for mean value $m$ strictly between the minima of $G$.
For simplicity of presentation, suppose that the minimizers of $G$ are normalized to be $\pm 1$. The energy landscape, which is a fundamental model for phase separation, reflects a competition between the energy and the mean constraint. Indeed, the mean constraint rules out the absolute energy minimizers $u\equiv \pm 1$ and raises the question of the lowest achievable energy given the constraint. One may also ask about the existence and ``shape'' of additional local minimizers and the height of the energy barriers surrounding them.

The study of energy barriers and the related critical points is driven by the issue of \emph{nucleation and growth} phenomena in physics and other applications.
For instance, when nearly homogeneous mixtures of alloys, glasses, or polymers are quenched, they tend to separate into distinct preferred phases.
When the initial homogeneous state is a local energy minimizer, the associated parameter regime is called the nucleation and growth regime (to be distinguished from the spinodal regime). In the Cahn-Hilliard model, which has been widely studied in experiments, numerical simulations, and analysis, the nucleation phenomenon consists of the formation and growth of small droplet-like regions of one phase inside a nearly homogeneous bulk phase. The initial formation and growth of droplets is often modelled by the stochastic Cahn-Hilliard-Cook equation \cite{C}.

Nucleation behavior was described already by Cahn and Hilliard in \cite{CH3}, where they discuss the formation of a so-called ``critical nucleus,'' a droplet-like state whose radius is exactly such that an infinitesimal increase in size leads under deterministic forces to growth and relaxation to a similarly droplet-like local minimizer. Moreover they point out the importance of the height of the energy barrier, which they define as the energy difference between the homogeneous state and the saddle point. In terms of mathematical analysis, the fact that nucleation events take place by way of the saddle point of least energy was put on rigorous ground by the theory of large deviations \cite{FW}.
Deriving accurate information about the critical nucleus experimentally is extremely challenging and there has been a considerable effort to study the nucleation problem numerically \cites{Sc,W,ZCD,CLEZS,ZCD2,DEPSW,LZZ,ZZD}.


In terms of analysis, most previous work has studied the Cahn-Hilliard energy
\begin{align}
  E_\phi(u):=\int_\Omega \frac{\phi}{2}|\nabla u|^2+\frac{1}{\phi}G(u)\,dx,\label{ep}
\end{align}
for $\Omega$ and mean value $m$ fixed and $\phi$ small.
In the so-called critical parameter regime studied in \cites{BGLN,CCELM,GW} and the present paper, the analysis is subtle because the energy of the homogeneous state $u\equiv m$, the energy of a droplet-like local minimizer, and the energy barrier in between these two states are all of the same order.

Our results include the existence and symmetry properties of a nonuniform local minimizer, existence of a saddle point, and quantitative bounds on the droplet shape of critical points in the form of their Fraenkel asymmetry and the $L^2$ distance to a sharp-interface droplet profile. Our work uses variational arguments and $\Gamma$-convergence in a fundamental way; indeed, an
objective of the paper is to explore the use of $\Gamma$-limits and error bounds to glean information about the shape of the energy landscape. In addition we make use of quantitative isoperimetric inequalities and Steiner symmetrization. We will give the results as we go along, once we have introduced the necessary notation and tools. For the reader who is eager to turn to the main results, we refer to theorems \ref{t:exist}, \ref{t:saddle}, and \ref{t:spheres} along with propositions \ref{t:exus}, \ref{prop:steiner}, and \ref{Asym.stronger}.

\subsection{The critical parameter regime}
To be concrete, let $\Omega=\mathbb{T}_L$ be the $d$-dimensional flat torus with side length $L$ and consider mean value $m=-1+\phi$. For simplicity, we set
\begin{align}
G(u)=\frac{(1-u^2)^2}{4}, \label{gbl}
\end{align}
but more general nondegenerate double-well potentials are possible (see subsection \ref{ss:org}). For $m$ close to $-1$ and periodic boundary conditions, it is easy to see that the uniform state $\baru\equiv m$ is a local energy minimizer. Determining whether it is the global minimizer is more subtle, as we now explain. Define the constant
\begin{align}\label{xi.sub.d}
\xi_d:=c_0^{d/(d+1)}\sigmad^{1/(d+1)}\frac{d+1}{4^{d/(d+1)}d^{1/(d+1)}},
\end{align}
where here and throughout, $c_0$ denotes
\begin{align*}
  c_0=\int_{-1}^{1}\sqrt{2G(s)}\,ds\overset{\eqref{gbl}}=\frac{2\sqrt{2}}{3},
\end{align*}
and $\sigmad$ stands for the perimeter (surface area) of the $(d-1)-$ unit sphere in $\mathbb{R}^d$.
It was shown in \cites{BGLN,CCELM} that the scaling $\phi\sim L^{-d/(d+1)}$ is critical in the following sense. For
\begin{align}
  \phi=\xi L^{-d/(d+1)}\label{crit}
\end{align}
with $\xi<\xi_d$, $\baru$ is the global minimizer of $E$ in $\mathbb{T}_L$ for  $\phi$ sufficiently small. For \eqref{crit} with $\xi>\xi_d$, on the other hand, $\baru$ is \emph{not} the global minimizer, and moreover the global minimizer is close in $L^p(\mathbb{T}_L)$ to a droplet-shaped function (i.e., a function that is close to $1$ on a ball $B$ and close to $-1$ on $\mathbb{T}_L\setminus B$).

In order to look more closely at the energy difference to the uniform state, one would like to analyze the rescaled energy gap
\begin{align}
\ephi(u)=\frac{E(u)-E(\bar{u})}{\phi^{-d+1}},\notag
\end{align}
which, recalling the mean constraint and rescaling  space by a factor of $\phi$, can also be written as
\begin{align}
\ephi(u)&=\int_{\tphil}\frac{\phi}{2}|\nabla u|^2+\frac{1}{\phi}\Big(G(u)-G(-1+\phi)\Big)\,dx\notag\\
&= \int_{\tphil} \underbrace{\frac{\phi}{2}|\nabla u|^2 +\frac{1}{\phi}\Big(G(u)-G(-1+\phi)-G'(-1+\phi)(u-(-1+\phi)) \Big)}_{=:e_\phi(u)}\,dx.\label{ephi}
\end{align}
We will restrict the space of functions in $H^1\cap L^4(\tphil)$ with the norm $\|\nabla u\|_{L^2(\tphil)}+\|u\|_{L^4(\tphil)}$
 to the affine subspace
\begin{align}\label{H-phi}
X_\phi:=\Big\{u\in H^1\cap L^4(\tphil): \dashint_{\tphil} u \, dx=-1+\phi\Big\}.
\end{align}

In \cite{GW}, the first and third authors establish $\Gamma$-convergence as $\phi\to 0$ to
\begin{align}\label{F_0}
\lefteqn{\ez(u):=}\notag\\
&\begin{dcases}
c_0{\rm Per}(\{u=1\})-4|\{u=1\}|+4\frac{|\{u=1\}|^2}{\xi^{d+1}}& \text{if}\;u=\pm 1 \,\text{a.e.}\,\,\text{and}\,\,{\rm Per}(\{u=1\})<\infty\\
+\infty&\text{otherwise}
\end{dcases}
\end{align}
in the $-1+L^p(\R^d)$ topology for any $p\in(1,\infty)$ (see \cite{GW} for details). Since the torus $\tphil$ converges to $\R^d$ in this limit, one does not have the compactness that is available in the classical problem of Modica and Mortola, cf. \cites{MM,Mo,S}.

The $\Gamma$-convergence result of \cite{GW} is based on matching (at leading order) upper and lower bounds for the rescaled energy gap; see subsection \ref{ss:lowup} for a summary.
These bounds in
the regime
\begin{align*}
  \xi\in(\xi_d,\infty),\quad \phi=\xi\,L^{-d/(d+1)},\quad \phi\ll 1
\end{align*}
imply the existence of the minimizer already pointed out in \cites{BGLN,CCELM} and imply in addition the existence of a minimum energy saddle point ``in between'' this minimizer and $\baru$.

Because the classical isoperimetric inequality (cf.\ theorem \ref{Regular.Isop}) says that the perimeter functional is minimized on balls, minimization of \eqref{F_0} as a function of the measure $|\{u=1\}|$ leads to the function $f_\xi :[0,\infty)\to \mathbb{R}$ defined via
\begin{align}
f_\xi(\nu):=\bar{C}_1 \nu^{(d-1)/d}-4\nu+4\xi^{-(d+1)}\nu^2,\notag
\end{align}
where
\begin{align}
\bar{C}_1:=c_0\sigmad^{1/d}d^{(d-1)/d}.\label{barc1}
\end{align}

In terms of $f_\xi$ it is easy to understand the constant $\xi_d$ in \eqref{xi.sub.d}: It is exactly the value of $\xi$ at which the strictly positive minimizer $\nu_m$ of $f_\xi$ changes from being a local minimizer to the global minimizer---which explains heuristically why it is at this point that the global minimizer of the energy changes from being $\baru$ (analogous to $\nu=0$) to being a nonuniform ``droplet-like'' state (analogous to $\nu=\nu_m$).
A second important value of $\xi$ that will play a role in our paper is the (saddle-node) bifurcation point
\begin{align}
\tilde{\xi_d}:&=c_0^{d/(d+1)}\sigmad^{1/(d+1)}\left(1-\frac{1}{d}\right)^{d/(d+1)}\frac{2^{1/(d+1)}(d+1)}{4^{d/(d+1)}d^{1/(d+1)}},\label{xid}
\end{align}
which is such that for $\xi<\tilde{\xi_d}$ the function $f_\xi$ has no positive local extrema, while for $\xi>\xid$, $f_\xi$ has a strictly positive local maximum and a strictly positive local minimum. See figure \ref{figurezz} for an illustration.

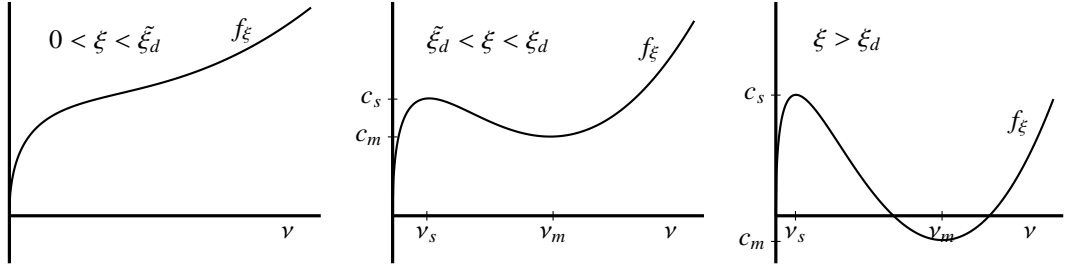
\begin{figure}
\begin{tikzpicture}[scale=0.63]

\begin{scope}[shift={(-8.5,0)}]
\draw [black, very thick] (5.5,0) -- (-1,0);
\draw [black, very thick] (-1,-1) -- (-1,4.5);
\draw[black, thick, domain=-1:5.3, samples=500] plot [smooth] (\x, {5*sqrt(0.5*(\x+1))-3*0.5*(\x+1)+0.5*0.25*(\x+1)*(\x+1)});
\node [below] at (4.8,0) {\small{$\nu$}};
\node [left] at (4.3,3.9) {\small{$f_\xi$}};
\node at (1,3.7) {\small{$0<\xi<\tilde{\xi_d}$}};
\end{scope}

\draw [black, very thick] (5,0) -- (-1.5,0);
\draw [black, very thick] (-1.5,-1) -- (-1.5,4.5);
\draw[black, thick, domain=-1.5:4.8, samples=500] plot [smooth] (\x, {9.7*sqrt(0.4*(\x+1.5))-10.15*0.4*(\x+1.5)+2*0.18*(\x+1.5)*(\x+1.5)});
\node [below] at (1.85,0) {\small{$\nu_m$}};
\node [below] at (4.3,0) {\small{$\nu$}};
\node at (0.5,3.7) {\small{$\tilde{\xi_d}<\xi<\xi_d$}};
\node [left] at (-1.5,1.66) {\small{$c_m$}};
\node [left] at (-1.5,2.455) {\small{$c_s$}};
\draw (-1.6,2.455) -- (-1.4,2.455);
\draw (-1.6,1.66) -- (-1.4,1.66);
\node [left] at (4.3,3.5) {\small{$f_\xi$}};
\draw (2.35-0.5,0.1) -- (2.35-0.5,-0.1);
\draw (-0.285-0.5,0.1) -- (-0.285-0.5,-0.1);
\node [below] at (-0.285-0.5,0) {\small{$\nu_s$}};
\draw [black, very thick] (12+0.5,0) -- (6+0.5,0);
\draw [black, very thick] (6+0.5,-1) -- (6+0.5,4.5);
\draw[black, thick, domain=6+0.5:11.8+0.5, samples=500] plot [smooth] (\x, {15*sqrt(0.3*(\x-6-0.5))-23*0.3*(\x-6-0.5)+7.5*0.09*(\x-6-0.5)*(\x-6-0.5)});
\node at (8,3.7) {\small{$\xi>\xi_d$}};
\node [left] at (11.5+0.5,1.95) {\small{$f_\xi$}};
\draw (9.47+0.5,0.1) -- (9.47+0.5,-0.1);
\node [below] at (9.47+0.5,0) {\small{$\nu_m$}};
\draw (5.9+0.5,2.54) -- (6.1+0.5,2.54);
\draw (6.41+0.5,-0.1) -- (6.41+0.5,0.1);
\draw (5.9+0.5,-0.54) -- (6.1+0.5,-0.54);
\node [left] at (6+0.5,-0.54) {\small{$c_m$}};
\node [left] at (6+0.5,2.54) {\small{$c_s$}};
\node [below] at (11.3+0.5,0) {\small{$\nu$}};
\node [below] at (6.41+0.5,0) {\small{$\nu_s$}};
\end{tikzpicture}
\caption{For $0<\xi<\tilde{\xi_d}$ the function $f_\xi$ has no positive local extrema. The global minimizer of $f_\xi$ is zero for $\xi<\xi_d$ and $\nu_m>0$ for  $\xi>\xi_d$.}\label{figurezz}
\end{figure}

\begin{notation}\label{not.1}
For $\xi>\xid$, we denote the strictly positive maximum and minimum of $f_\xi$ by $\nu_s,\,\nu_m$, respectively, and the corresponding function values by $$c_s:=f_\xi(\nu_s),\qquad c_m:=f_\xi(\nu_m).$$
It is easy to check that
\begin{align}
1\lesssim \nu_s< \nu_m<\frac{\xi^{d+1}}{2}\qquad\text{for $\xi\in (\xid,\xi_d]$,} \label{11.1}
\end{align}
where $\lesssim$ and related notation are explained in notation \ref{not:oh}.
Introducing
\begin{align}
  \gamma_0^2:=4(\nu_m-\nu_s),\label{gm0}
\end{align}
we note for future reference that
\begin{equation}\label{gamma_0.1}
\nu_m\,\, \text{is the strict minimizer of}\,\, f_\xi \,\, \text{on the interval}\,\, \Big[\nu_m-\frac{\gamma_0^2}{4},\nu_m+\frac{\gamma_0^2}{4}\Big].
\end{equation}
In addition, we define
 \begin{align*}
    \Psi(x;\omega)=\begin{cases}
     +1& x\in B_{\omega}(0)\\
     -1& x\in \R^d\setminus B_{\omega}(0),
   \end{cases}
 \end{align*}
where $B_{\omega}(0)$ is a ball with volume $\omega$ and center $0$. We will abbreviate
\begin{align}\label{defn.Psi_m}
\Psi_s:=\Psi(\cdot\,;\nu_s),\qquad \Psi_m:=\Psi(\cdot\,;\nu_m).
\end{align}
Later (in  lemmas~\ref{la} and~\ref{la2}) we verify that $\Psi_s$ and $\Psi_m$ are a saddle point and local minimum of the limit functional. For this reason we will sometimes refer to $\Psi_s$, $\Psi_m$ as the \uline{limit saddle point} and \uline{limit local minimizer}.
\end{notation}

\begin{notation} We compensate for the translation invariance of the problem by using
 \begin{align*}
 |u-\Psi_m|_{\R^d}&:=\inf_{x_0\in\R^d}||u-\Psi_m(\cdot-x_0)||_{L^2(\R^d)},\\
 |u-\Psi_m|_{\tphil}&:=\inf_{x_0\in\tphil}||u-\Psi_m(\cdot-x_0)||_{L^2(\tphil)},
 \end{align*}
 where in the second equation, $\Psi_m$ is understood in the periodic sense, i.e., we restrict to $[-\phi L/2,\phi L/2)^d$ and then take the periodic continuation.
 \end{notation}

\subsection{First result: a droplet-shaped local minimum}
In order to state our results, we introduce the following ``volume'' functional $\nu(\cdot)$, also used in \cite{GW}, which will play an important role in our analysis.
\begin{notation}\label{notvol}
Let $\kappa:=\phi^{1/3}$. As in \cite{GW}, we define a smooth partition of unity $\chi_1, \chi_2$, and $\chi_3 :\mathbb{R}\to [0,1]$ such that $\chi_1+\chi_2+\chi_3=1$ and
\begin{align*}
\chi_1(t)&=
\begin{cases}
1\quad&\text{for}\quad t\leq -1+\kappa\\
0\quad&\text{for}\quad t\geq -1+2\kappa,
\end{cases}\\
\chi_2(t)&=
\begin{cases}
1\quad&\text{for}\quad -1+2\kappa\leq t\leq 1-2\kappa\\
0\quad&\text{for}\quad t\leq -1+\kappa \quad \text{and} \quad t\geq 1-\kappa,
\end{cases}\\
\chi_3(t)&=
\begin{cases}
1\quad&\text{for}\quad t\geq 1-\kappa\\
0\quad&\text{for}\quad t\leq 1-2\kappa.
\end{cases}
\end{align*}
We use $\chi_3$ to define the continuous ``volume-type'' functional $\nu:X_\phi\to \mathbb{R}$ via
\begin{align}
\nu(u):=\int_\Omega \chi_3(u)\,dx.\label{voltp}
\end{align}
This functional roughly measures the volume of $u\approx +1$. We will occasionally refer to the sharp-interface analogue
 \begin{align*}
   \nu_0(u)=|\{u=+1\}|.
 \end{align*}
\end{notation}

Our first theorem exploits  information about the energy and its connection to the $\Gamma$-limit in order to prove the existence of a droplet-shaped local minimizer. For simplicity (and to avoid a $\xi$-dependence in our constants), we restrict our attention to $\xi\in(\tilde{\xi_d},\xi_d]$. This in any event is the more interesting regime, since it has not yet been explored in \cites{BGLN,CCELM} or elsewhere, to the best of our knowledge, and since we identify \emph{a local but nonglobal minimizer}. However one could use the same approach to establish existence and properties of the global minimizer for $\xi>\xi_d$.

\begin{thm}\label{t:exist}
Consider $\xi\in(\tilde{\xi_d},\xi_d]$ and the critical scaling \eqref{crit}. For $\phi$ sufficiently small, there exists a nonconstant local minimizer $\um$ of $\ephi$. This function minimizes $\ephi$ over
\begin{align}
  |u-\Psi_m|_{\tphil}\leq\gamma_0,\label{uni}
\end{align}
where $\Psi_m$ is the limit local minimizer defined in \eqref{defn.Psi_m} and $\gamma_0$ is the constant from \eqref{gm0}.

The local minimizer $\um$ is well-approximated by $\Psi_m$ in the sense that, for every  $\gamma\in(0,\gamma_0)$, there exists $\phi_0>0$ such that
\begin{align}
  |\um-\Psi_m|_{\tphil}\leq \gamma\qquad \text{for all }0<\phi\leq \phi_0.\label{ump}
\end{align}
In addition the closeness of the local minimizer $\um$ and the limit local minimizer $\Psi_m$ in volume and energy are estimated by
\begin{align}
  \left|\ephi(\um)-c_m\right|&\lesssim \phi^{1/3},\label{en.cl}\\
   |\nu(\um)-\nu_m|&\leq C\phi^{1/6},\label{num}
\end{align}
where $C$, in addition to depending as usual on $d$, depends on $\xi$  (and is large for $\xi$ near $\tilde{\xi}_d$).

\end{thm}
\begin{remark}[Case of equality] We remark that the case $\xi=\xi_d$ is not excluded. In \cites{BGLN,CCELM}, where the \emph{global minimizer} is studied, only $\xi<\xi_d$ and $\xi>\xi_d$ are considered. In our setting, because we are interested in \emph{local minimizers}, there is no reason to exclude $\xi=\xi_d$, the crossover point at which global minimality is traded from one minimizer to the other. Although our results do not identify the global minimizer, they do imply that any global minimizer is $L^2$ close to $\bar{u}$ or $\Psi_m$.
\end{remark}
\begin{rmrk}[Approximate strictness of $\um$]\label{rem:strictmin}
One can deduce from the theorem together with lemma \ref{l:alt} below that $\um$ is ``approximately a strict minimizer'' in the sense that
for every $\gamma\in (0,\gamma_0]$, there exists $\delta>0$ and $\phi_0>0$ such that for all $\phi \leq\phi_0$ there holds
\[
|u-\um|_{\tphil}=\gamma \qquad\Rightarrow\qquad \ephi(u)>\ephi(\um)+\delta.
\]
\end{rmrk}

\begin{remark}[Quantified closeness to the sharp-interface minimum] We improve from \eqref{ump} to the quantified estimate  in \eqref{visa} below.
\end{remark}

One may expect---via a symmetrization argument---to show that $\um$ is spherically symmetric. On $\R^d$ symmetrization leads indeed to spherical symmetry. The periodic setting ``frustrates'' the system, however, preventing  spherical symmetry of sets whose volume is too large. In subsection \ref{ss:steinuo}, we use Steiner symmetrization on the torus to deduce additional information about local energy minimizers. Meanwhile, the variational arguments that lead to \eqref{ump} and the quantification in theorem \ref{t:spheres} provide a measure of the deviation from sphericity that is forced by confinement to the torus $\tphil$.

\subsection{Second result: Towards the critical nucleus}
In proposition \ref{t:exus} below, we use theorem \ref{t:exist}, a lower bound on the energy barrier around $\um$, and a mountain pass argument to deduce the existence of a saddle point with energy equal to this barrier.
As mentioned above, energy barriers are a fundamental object in the study of large deviations, where they give the exponential factor in the expected time for a stochastic perturbation to drive the system out of the basin of attraction of a local minimizer in the small noise limit; cf.\ \cite{FW}.

Given theorem \ref{t:exist}, it is natural to define the energy barrier around $\um$ as
\begin{align}\notag
\deone:=\inf_{\psi }\max_{t\in [0,1]} \ephi(\psi(t)).
\end{align}
over  paths $\psi \in C([0,1];X_\phi)$ such that
\begin{align*}
  \ephi(\psi(0))<\ephi(\um),\qquad \psi(1)=\um.
\end{align*}
Our lower bound (cf.\ proposition \ref{lemma.lower.bound}) bounds this quantity from below and allows for a mountain pass argument. Unfortunately our constructions do not take us all the way to $\um$, so that we do not obtain a matching upper bound. Suppose that we are satisfied with reaching the following neighborhood of $\um$:
\begin{align*}
  \mathcal{N}_\eps(\um):=\left\{u\in X_\phi\colon |u-\um|_{\tphil}+\ephi(u)-\ephi(\um)< \eps\right\}.
\end{align*}
Then we can define the modified energy barrier
\begin{align}
\de:=\inf_{\psi }\max_{t\in [0,1]} \ephi(\psi(t)),\label{barkk}
\end{align}
over  paths $\psi \in C([0,1];X_\phi)$ such that
\begin{align*}
  \ephi(\psi(0))<\ephi(\um),\qquad \psi(1)\in\mathcal{N}_\eps(\um).
\end{align*}
The constructions from proposition \ref{prop.upper.bound} together with theorem \ref{t:exist} verify that, for fixed, small $\eps$, and $\phi$ sufficiently small, there exists at least one such path. Our existence result for saddle points takes the following form.
\begin{proposition}\label{t:exus}
  For $\xi\in(\tilde{\xi_d},\xi_d)$, the critical scaling \eqref{crit}, and $\phi>0$ small enough, there exists a saddle point $\us$ of $\ephi$ such that
    \begin{align}
    \ephi(\us)=\deone\geq c_s+O(\phi^{1/3}).\label{ebfind}
  \end{align}
  In addition, there exists a (possibly different) saddle point $\ust$ such that
  \begin{align}
    \ephi(\ust)=\de=c_s+O(\phi^{1/3}).\label{ebfind2}
  \end{align}
\end{proposition}
Although $\de$---and hence also $\ust$---depend on $\eps$, our estimate on the right-hand side of \eqref{ebfind2} is independent of $\eps$ as $\phi\to 0$. For this reason we do not explicitly denote the $\eps$-dependence.


We would like to say more about the ``droplet-like'' shape of the saddle points and the connection to the saddle point of the $\Gamma$-limit. Indeed, it is natural to think of $\us$ as the so-called critical nucleus \cite{CH3}, which is close in volume and $L^2$ to $\Psi_s$. Unfortunately we just miss being able to establish these facts.
As a partial substitute, we establish closeness in volume and $L^2$ of appropriately defined \emph{constrained minimizers} of the energy.

We define the constrained minimizers in the following way.
For $\omega\in [0,\xi^{d+1}/2)$, we define the functions $\uo\in X_\phi$ such that
\begin{align*}
  \nu(\uo)=\omega,\qquad \ephi(\uo)=\ehat(\omega),
\end{align*}
where
\begin{align}
\ehat(\omega):=  \min\Big\{\ephi(u)\colon u\in X_\phi,\,\nu(u)=\omega\Big\}.\label{ehat}
\end{align}
Because $\nu(u)$ is a stand-in for the volume of the set $\{u\approx 1\}$, we refer to such points $\uo$ as \uline{volume-constrained minimizers} or simply constrained minimizers, when there is no risk of confusion. Existence of the volume-constrained minimizers follows from the direct method of the calculus of variations.

One would like to use these constrained minimizers to define a continuous path over the mountain pass that keeps the energy as small as possible along the way. Indeed, continuity of $\omega\mapsto \uo$ would allow one to deduce information about the ``volume'' of $\us$ and hence also the $L^2$ closeness to $\Psi_s$, using for instance the work of Ghoussoub and Preiss \cite{GP}, in which they extract additional information about the location (in phase space) and type of critical points based on a mountain pass argument involving separating sets. Unfortunately uniqueness of the constrained minimizers is an open question (cf. remark \ref{rem:nonuq}), and continuity of $\omega\mapsto \uo$ is not immediately clear.

Although we cannot yet deduce fine properties of $\us$ or $\ust$, we can deduce the corresponding properties of an appropriately defined constrained minimizer.
We define the ``weak'' energy barrier surrounding $\um$ as
\begin{align}\label{eb2}
\Ew:=\sup_{\omega\in [0,\nu_m]} \ehat(\omega).
\end{align}
From the lower and upper bounds in propositions \ref{lemma.lower.bound} and \ref{prop.upper.bound}, one can immediately deduce the following. (We omit the proof.)
\begin{lemma}
  The weak energy barrier satisfies
  \begin{align}
    \Ew=\de+O(\phi^{1/3})=c_s+O(\phi^{1/3}).\label{ebjan2}
  \end{align}
\end{lemma}
For our ``approximate saddle point,'' we choose (any) $\omega_*\in [0,\nu_m]$ such that
\begin{align}
  \ehat(\omega_*)=\Ew+O(\phi^{1/3})\label{omgst}
\end{align}
and denote by $\uws$ a corresponding volume-constrained minimizer. In the following theorem, we establish the desired properties of $\uws$.
\begin{thm}\label{t:saddle}
Consider $\xi\in(\tilde{\xi_d},\xi_d]$ and the critical scaling \eqref{crit}. For $\phi$ sufficiently small, the volume-constrained minimizer with volume $\omega_*$ satisfying \eqref{omgst} is well-approximated by the limit saddle point $\Psi_s$ in the following sense. For every  $\gamma>0$ sufficiently small, there exists $\phi_0>0$ such that
\begin{align}
  |\uws-\Psi_s|_{\tphil}\leq \gamma\qquad \text{for all }0<\phi\leq \phi_0.\label{ump2}
\end{align}
The constrained minimizer $\uws$ and the limit saddle point $\Psi_s$ are close in volume and energy in the sense that
\begin{align}
|\ephi(\uws)-c_s|&\lesssim \phi^{1/3},\label{enps}\\
  | \nu(\uws)-\nu_s|&\leq C \phi^{1/6},\label{wmm2}
\end{align}
where $C$, in addition to depending as usual on $d$, depends on $\xi$  (and is large for $\xi$ near $\tilde{\xi}_d$).
\end{thm}

\begin{rmrk}[Approximate mountain pass property]\label{rem:appxmtn}
We do not establish that $\uws$ is a saddle point, much less a saddle point of mountain pass type. However $\uws$, is an approximate mountain pass point in the following sense.
For any $\gamma>0$, there exist $\delta>0$, $\phi_0>0$ such that for $\phi\leq\phi_0$, there holds
\begin{enumerate}
\item[(i)] for any $u\in X_\phi$,
\[
\nu(u)=\nu(\uws),\,\, |u-\uws|_{\tphil}>\gamma \qquad\Rightarrow\qquad \ephi(u)>\ephi (\uws)+\delta,
\]
\item[(ii)] there is a point in the $\gamma$-neighborhood of $\uws$ with smaller volume and lower energy and a point in the $\gamma$-neighborhood of $\uws$ with larger volume and lower energy.
\end{enumerate}
The first point follows from theorem \ref{t:saddle} together with lemmas \ref{la2} and \ref{lc2} below.
For the second point, it is convenient to use closeness of $\uws$ to $\Psi_s$, so that it suffices to find a point $\hat{u}_\phi\in X_\phi$ such that
\begin{align*}
  |\hat{u}_\phi-\Psi_s|<\gamma/2\qquad\text{and}\qquad \ephi(\hat{u}_\phi)<\ephi(\uws).
\end{align*}
The constructions from proposition \ref{prop.upper.bound} (for volume slightly less or slightly greater than $\nu_s$) do the job.
\end{rmrk}

\begin{remark}[Quantified closeness to the sharp-interface saddle]
As for the minimizer, we quantify \eqref{ump2} in \eqref{visa} below.
\end{remark}

Although we do not manage to show theorem \ref{t:saddle} with $\uws$ replaced by $\us$ or $\ust$, we do obtain information about any approximately optimal path for $\de$. We make this connection precise in remark \ref{rem:all} below after first introducing theorem \ref{t:spheres}.

\subsection{Refinement via isoperimetric inequalities and Steiner symmetrization}
Because the perimeter functional plays the only geometric role in the $\Gamma$-limit, the (classical) isoperimetric inequality suggests that approximately radial functions should be optimal in terms of energy.
This idea is a key ingredient in the lower bound  of the $\Gamma$-convergence argument. Now we would like to measure the defect. As mentioned above, although the critical points would be exactly radial on $\R^d$, they are frustrated in our setting because they are confined to the torus.
We are interested in estimating their deviation from sphericity for $\phi$ small but nonzero, which corresponds to $\tphil$ large but bounded.  We achieve this goal by obtaining quantitative bounds on the $L^2$ distance and so-called Fraenkel asymmetry as a function of $\phi$.
The Fraenkel asymmetry and sharp isoperimetric inequality of Fusco, Maggi, and Pratelli \cite{FMP} also play an important role in several of our proofs. We recall the definitions and theorems.

It will be useful to define the ``isoperimetric function'' ${\rm P}_E$ that associates to a set the perimeter of the ball with the same volume.
\begin{definition}
The Euclidean isoperimetric function in $\mathbb{R}^d$ is defined by
\begin{equation}\label{Eucl.Isop.Fn}
{\rm P}_E(A):=\sigmad ^{1/d}d^{(d-1)/d}|A|^{(d-1)/d},
\end{equation}
for any Borel set $A\subset \mathbb{R}^d$.
\end{definition}
Using this notation, we can express the classical isoperimetric inequality on $\R^d$ in the following way.
\begin{thm}\label{Regular.Isop}
For any Borel set $A\subset \R^d$ there holds
\[
\per(A)\geq {\rm P}_E(A),
\]
where $\per(A)$ is the perimeter of $A$ in $\R^d$.
\end{thm}
While theorem \ref{Regular.Isop} does not apply to the torus, its conclusion still holds true for sets of small enough measure. This is the content of the following isoperimetric inequality, which we state as in \cite{CCELM}, and which is a special case of \cite[theorem 4.4]{MJ}.
\begin{thm}[\cite{CCELM}, theorem 6.1]\label{Regular.Isop.Torus}
Let $d\geq 2$ and consider the unit $d$-dimensional flat torus $\mathbb{T}_1:=[-1/2,1/2]^d$. There exists an $\epsilon=\epsilon(d)>0$ such that for any Borel set $A\subset \mathbb{T}_1$ with $|A|\leq \epsilon$, the perimeter $\per_{\mathbb{T}_1}(A)$ of $A$ in $\mathbb{T}_1$ satisfies
\[
\per_{\mathbb{T}_1}(A)\geq {\rm P}_E(A).
\]
\end{thm}
We will use the isoperimetric inequality in the setting of the torus $\tphil$, in which case theorem \ref{Regular.Isop.Torus} takes the form
\begin{equation}\label{Regular.Isop.Torus2}
\per_{\tphil}(A)\geq {\rm P}_E(A),
\end{equation}
for any Borel set $A\subset \tphil$ with $|A|\leq \epsilon \,|\tphil|$. The positive constant $\epsilon$ is the same as in theorem \ref{Regular.Isop.Torus}.

The next order correction to the perimeter is probed via so-called quantitative isoperimetric inequalities, which quantify how much the perimeter of a set is increased from that of a ball when the set deviates from spherical. The deviation from sphericity is measured in terms of the Fraenkel asymmetry.
\begin{definition}[Fraenkel asymmetry]\label{def:frk}
The \uline{Fraenkel asymmetry in $\R^d$} of a set $E\subset\R^d$ is defined as
  \begin{align*}
    \lambda(A):=\min_{x\in\R^d} \frac{|A\triangle B(x)|}{|A|},
  \end{align*}
  where $B(x)\subset \R^d$ is a ball with center $x$ and volume $|A|$ and $A\triangle B$ denotes the symmetric difference of $A$ and $B$.

   Similarly, the \uline{Fraenkel asymmetry in the torus} of a set $A\subset\tphil$ whose measure does not exceed that of a ball of radius $\phi L/2$ is defined as
  \begin{align*}
    \lambda(A):=\min_{x\in\tphil} \frac{|A\triangle B(x)|}{|A|},
  \end{align*}
  where $B(x)\subset \tphil$ is a ball with center $x$ and volume $|A|$.
\end{definition}

We will use both the quantitative isoperimetric inequality on the full space and on the torus. We begin by recalling the sharp result of Fusco, Maggi, and Pratelli \cite{FMP}.
\begin{thm}[\cite{FMP}, theorem 1.1]
There exists a constant $C=C(d)$ such that for any Borel set $A\subset\R^d$ with $0<|A|<\infty$, the perimeter $\per(A)$ of $A$ in $\R^d$ satisfies
\begin{equation}\label{Regular.Sharp.Isop}
\per(A)\geq {\rm P}_E(A)+C(d)\lambda(A)^2 {\rm P}_E(A),
\end{equation}
where $\lambda(A)$ is the Fraenkel asymmetry of $A$ in $\R^d$ and ${\rm P}_E(A)$ is given by \eqref{Eucl.Isop.Fn}.
\end{thm}
Following the method of \cite[theorem 6.2]{CCELM} to verify the classical isoperimetric inequality on the torus, we verify the following quantitative isoperimetric inequality on the torus for sets of small measure.
\begin{corollary}\label{SharpIsoper}
Consider the unit $d$-dimensional flat torus $\mathbb{T}_1:=[-1/2,1/2]^d$. There exist constants $C=C(d)>0$ and $\epsilon=\epsilon(d)>0$ such that, for any Borel set $A\subset \mathbb{T}_1$ with $|A|<\epsilon$, the perimeter $\per_{\mathbb{T}_1}(A)$ of $A$ in $\mathbb{T}_1$ satisfies the inequality
\begin{align*}
\per_{\mathbb{T}_1}(A)\geq {\rm P}_E(A)+C(d)\lambda(A)^2 {\rm P}_E(A)-4d|A|.
\end{align*}
Here $\lambda(A)$ is the Fraenkel asymmetry of $A$ in the torus and ${\rm P}_E(A)$ is given by \eqref{Eucl.Isop.Fn}.
\end{corollary}
The proof of corollary \ref{SharpIsoper} is included in the appendix. We will use the sharp quantitative isoperimetric inequality in the setting of the torus $\tphil$,
in which case corollary \ref{SharpIsoper} takes the form
\begin{equation}\label{SharpIsoper2}
\per_{\tphil}(A)\geq {\rm P}_E(A)+C(d)\lambda(A)^2 {\rm P}_E(A)-\frac{4d|A|}{\phi L},
\end{equation}
for any Borel set $|A|\subset \tphil$ with $|A|<\epsilon\, |\tphil|$. Here $\per_{\tphil}(A)$ denotes the perimeter of $A$ in $\tphil$, and $C(d)$ and $\epsilon=\epsilon (d)$ are the same as in corollary \ref{SharpIsoper}.

In addition to using the quantitative isoperimetric inequality in our existence proofs, we use the inequality in order to prove the following theorem, which quantifies the degree to which the critical points are ``droplet-like.''
\begin{theorem}\label{t:spheres} Consider $\xi\in(\xid,\xi_d]$ and the critical scaling \eqref{crit}. The minimizer $\um$ and any volume-constrained minimizer $\uws$ for $\omega_*$ satisfying \eqref{omgst} are approximately spherical in the sense that, for every $s\in[-1+2\phi^{1/3},1-2\phi^{1/3}]$,  both $\um$ and $\uws$ satisfy
\begin{align}
\lambda(\{u>s\})\lesssim \phi^{\alpha}\qquad\text{with}\quad \alpha=\min\{1/6,1/(2d)\}\label{visam}
\end{align}
and consequently
\begin{align}
  |\um-\Psi_m|_{\tphil}^2+|\uws-\Psi_s|_{\tphil}^2
  \begin{cases}\label{visa}
  \leq C(\xi,d) \phi^{1/6}&\text{for}\quad d=2,3\\
  \lesssim \phi^{1/(2d)}&\text{for}\quad d\geq 4.
  \end{cases}
\end{align}
In fact, for any $\omega \in (0,\xi^{d+1}/2)$ and $\phi$ sufficiently small, any associated volume-constrained minimizer $\uo$ satisfies
\begin{align}
\lambda(\{\uo>s\})+ |\uo-\Psi(\cdot,\omega)|_{\tphil}^2\lesssim \frac{\phi^{\alpha}}{\omega}\qquad\text{with}\quad \alpha=\min\{1/6,1/(2d)\}.\label{uwtoo}
\end{align}
\end{theorem}
\begin{remark}[Near local minimizers and nearly optimal paths]\label{rem:all}
To be concrete, we state theorems \ref{t:exist}, \ref{t:saddle}, and \ref{t:spheres} in terms of the local minimizer $\um$ and the constrained minimizer $\uws$. However neither result uses the Euler-Lagrange equation, and a corollary to the proofs may be stated in the following form:

Any function $\tilde{u}\in X_\phi$ that is a nearly local minimizer in the sense that
\begin{align*}
  \nu(\tilde{u})=\nu_m+O(\phi^{1/6}),\qquad \ephi(\tilde{u})\leq c_m+O(\phi^{1/3})
\end{align*}
is droplet-shaped and close to the sharp-interface minimizer in the sense that
\begin{align}
  \lambda(\{\tilde{u}>s\})&\lesssim \phi^{\alpha}\qquad  \text{ for every $s\in[-1+2\phi^{1/3},1-2\phi^{1/3}]$},\label{paht3}\\
\text{and}\qquad  |\tilde{u}-\Psi_m|_{\tphil}^2&\lesssim \phi^{\alpha},\notag
\end{align}
where, as in theorem \ref{t:spheres}, $ \alpha=\min\{1/6,1/(2d)\}$.
Similarly, any function $\tilde{u}\in X_\phi$ such that
\begin{align}
  \nu(\tilde{u})=\nu_s,\qquad \ephi(\tilde{u})\leq c_s+O(\phi^{1/3})\label{paht}
\end{align}
is droplet-shaped and close to the sharp-interface saddle in the sense of \eqref{paht3} and
\begin{align}
  |\tilde{u}-\Psi_s|_{\tphil}^2\lesssim \phi^{\alpha}.\label{paht2}
\end{align}
In particular, we use the second fact to make the following observation. Although we do not determine the volume or shape of the saddle points $\us$ and $\ust$,
we do know that \emph{any path that is nearly optimal for $\de$ contains a point $\tilde{u}$ satisfying \eqref{paht}}, and hence contains an approximate droplet-state that is well-approximated by the limit saddle, in the sense made precise by \eqref{paht3} and \eqref{paht2}.
\end{remark}

We use Steiner symmetrization to obtain additional information about the qualitative properties of $\um$ and $\uws$ in section \ref{S:stf}. To improve from the existence of a Steiner symmetric minimizer to the fact that any constrained minimizer is Steiner symmetric, the main issue is the behavior of the gradient energy under symmetrization. To this end, we apply the fairly recent work \cite{CF}.
Our main result takes the following form.
\begin{proposition}\label{prop:steiner}
Consider $\xi\in (0,\xi_d]$, the critical scaling \eqref{crit}, and $\omega_1>0$.  For $\phi>0$ sufficiently small, the following holds true. For any $\omega\in [\omega_1,\xi^{d+1}/2]$, any volume-constrained minimizer $\uo$ is (up to a translation) equal to its Steiner symmetrization about the origin. In particular, its superlevel sets are simply connected and $\uo$ is strictly decreasing in all directions away from the unique point of maximum.
\end{proposition}
In $d=2$, we use the connectedness of the superlevel sets from proposition \ref{prop:steiner} together with the Bonnesen inequality to obtain even stronger information about the droplet-like shape of the critical points; see proposition \ref{Asym.stronger}.



\subsection{Additional related results in the literature}

Previous analysis has focused on \eqref{ep} with fixed mean and $\phi$ small or order one. In dimension one, the  global minimizer and saddle point were analyzed in \cites{CGS,BFi} and stochastic nucleation was analyzed in \cite{BGW}.
In $d\geq 2$, some of the first results on the qualitative properties of critical points appear in \cite{GM,NT}.
More recently, dynamical systems techniques lie at the heart of a series of papers by Wei and Winter \cite{WW}-\cite{WW4} and Bates and Fusco \cite{BFu},
in which so-called spike and bubble solutions of the Cahn-Hilliard energy are analyzed. In particular, \cite{WW2} establishes existence of critical points that possess an ``interior spike,'' \cites{BFu,WW4} establish the existence and properties of critical points with multiple interior spikes, and \cite{WW3} establishes existence of critical points with a spherical interface.
While our parameter regime leads (for both the minimum and saddle point) to the scale separation
\begin{align*}
 \underbrace{ \begin{pmatrix}\text{lengthscale of $\Omega$} \end{pmatrix}}_{``macroscale''}
 \gg
 \underbrace{\begin{pmatrix}\text{lengthscale of droplet}\end{pmatrix}}_{``mesoscale''}
 \gg
 \underbrace{\begin{pmatrix}
    \text{lengthscale of the interface}
  \end{pmatrix}}_{``microscale''},
\end{align*}
the interior spikes of \cites{WW2,WW4,BFu}  has no mesoscopic scale
and the bubbles of \cite{WW3} satisfy
\begin{align*}
  \begin{pmatrix}\text{lengthscale of $\Omega$} \end{pmatrix}\sim \begin{pmatrix}\text{lengthscale of droplet}\end{pmatrix}\gg \begin{pmatrix}
    \text{lengthscale of the interface}
  \end{pmatrix}.
\end{align*}

The idea of using  $\Gamma$-convergence to establish existence of local (and not just global) minimizers goes back to Kohn and Sternberg \cite{KS}. They study the zeroth order Allen-Cahn energy (i.e., the energy $E_\phi$ defined in \eqref{ep} with no mean constraint) on nonconvex domains and establish the existence of $L^1$ local minimizers for $\phi$ small. They also comment on the constrained problem for $\phi\ll 1$, although they do not study droplet-type functions. In a similar spirit, Choksi and Sternberg \cite{CS} study $E_\phi$ on the unit flat torus for $d=2$ and with a fixed mean constraint $m$. After establishing that disks and strips always locally minimize perimeter, they deduce the existence of nearby droplet and strip local minimizers of $E_\phi$ for small $\phi$.
Another implementation of this idea can be found in Chen and Kowalczyk \cite{CK} who, utilizing local maxima of the curvature of $\partial \Omega$, show the existence of local minimizers of the energy \eqref{ep} on a smooth bounded domain $\Omega\subset \R^2$ for small $\phi$ and a fixed mean $m$.
In related work, Sternberg and Zumbrun \cite{SZ} study the Cahn-Hilliard energy landscape for strictly convex domains $\Omega\subset\R^d$ and show that stable critical points have a thin, connected transition layer connecting the pure phases $\approx \pm 1$.
To contrast with our paper, we point out that in \cites{KS,CK,SZ} the \emph{geometry of the domain} plays a central role, whereas in our work the central role is played by the \emph{nonconvexity of the $\Gamma$-limit}.

\subsection{Generalizations and organization}\label{ss:org}
Working on the torus is not important for the energetic bounds and $\Gamma$-limit, however periodicity allows us to apply the quantitative isoperimetric inequality of \cite{FMP}. In addition working on the torus allows us to exploit Steiner symmetrization.
Rather than taking the standard double-well potential \eqref{gbl}, it is straightforward to consider more general double-well potentials. Normalizing as usual so that the global minima  are $\pm1$, we may consider any potential $G\in C^2(\R)$ such that
\begin{itemize}
\item $G(\pm1)=0$ and $G(u)>0$ for all $u\in \R\setminus \{\pm1\}$,
\item $G''(\pm 1)>0$,
\item $G$ is convex on $(-\infty, -1)$ and $(1,\infty)$,
\item $G(u)\gtrsim u^2$ for $|u|\gg 1$.
\end{itemize}
If the last assumption is replaced by $G(u)\gtrsim |u|^p$ for $p\in(1,\infty)$, then then our results hold for $L^p(\tphil)$.


We begin in section \ref{S:prelim} by recalling and establishing some preliminary estimates. Then in section \ref{S:min}, we establish and exploit connections between the $\Gamma$-limit and the original energy. In particular, in subsection
\ref{ss:locmin} we prove theorem \ref{t:exist}, establishing existence and initial properties of the local minimizer $\um$. In subsection \ref{ss:sadtoo}, we prove the corresponding results for the saddle point, deducing in particular proposition \ref{t:exus} and theorem \ref{t:saddle}.
Deviation from sphericity is quantified in subsection \ref{ss:dev}, proving theorem \ref{t:spheres} and establishing  that $\um$ is a volume-constrained minimizer. Two auxiliary lemmas are proved in subsection \ref{ss:pflemmas}.

Section \ref{S:stf} derives additional properties of the constrained minimizers using Steiner symmetrization and the Euler-Lagrange equation. After introducing Steiner symmetrization in subsection \ref{ss:steiner}, we apply it in subsection \ref{ss:steinuo} to deduce that constrained minimizers are Steiner symmetric. In the final subsection \ref{S:finer}, the resulting connectedness of superlevel sets is used together with the Bonnesen inequality in $d=2$ to obtain a sharper characterization of the droplet-like shape of constrained minimizers.


\section{Notation and preliminary estimates}\label{S:prelim}
In this section we recall some facts from \cite{GW} and establish some preliminary estimates.
To begin we introduce some additional notation that will be used throughout the paper.
\begin{notation}\label{not:oh}
  For nonnegative quantities $X$ and $Y$, we write $X \lesssim Y$ to indicate that there exists a constant $C>0$ that depends at most on $d$ such that $X \leq C Y$ for small enough $\phi>0$. Writing $X \sim Y$ means that $X \lesssim Y$ and $Y \lesssim X$. In addition we use the standard $O(\cdot)$ and $o(\cdot)$ notation (where again dependency on $d$ is permitted) with respect to $\phi \to 0$ (unless otherwise indicated).
\end{notation}
We recall for reference below that in the critical scaling \eqref{crit}, there holds
\begin{align}
               \frac{1}{\phi}\int_{\tphil} G(-1+\phi)\,dx &\to \xi^{d+1},\label{f2}\\
               \phi|\tphil|&= \xi^{d+1}.\label{f1}
              \end{align}
As above, we abbreviate $\kappa:=\phi^{1/3}$. Given a function $u:\tphil\to \R$, we define the partition of $\tphil$ via
\begin{align}
\cala(u):&=\{u\leq -1-\kappa\},\notag\\
\calb(u):&=\{-1-\kappa<u\leq-1+\kappa\},\notag\\
\calc(u):&=\{-1+\kappa<u\leq 1-\kappa\},\notag\\
\cald(u):&=\{1-\kappa< u\leq 1+\kappa\},\label{calc}\\
\cale(u):&=\{u>1+\kappa\}.\notag
\end{align}
For simplicity we write $\cala,\dots,\cale$ instead of $\cala(u),\dots,\cale(u)$ when there is no danger of confusion.

Recall the partition of unity and approximate volume functional $\nu(\cdot)$ from notation \ref{notvol}. We will use the partition of unity to decompose the energy \eqref{ephi} as
\begin{equation}\label{sl.10.5}
\ephi(u)=\int_{\tphil}e_\phi(u)\Big(\chi_1(u)+\chi_2(u)+\chi_3(u)\Big)\,dx.
\end{equation}
In order to maintain universal constants in our estimates, we often restrict to a given range of volumes $\nu(u)$.
Also, because we are interested in functions of relatively low energy, we will often restrict to
\begin{align}
\ephi(u)\leq \max_{\xi\in [\tilde{\xi}_d,\xi_d]}\max\left\{2c_s, f_\xi\left(\frac{\xi^{d+1}}{2}\right) \right\}+1=:E_{M}.\label{emm}
\end{align}
One can instead consider functions with
\begin{align*}
\ephi(u)\leq C\qquad \text{for }C<\infty,
\end{align*}
but then some bounds will depend on $C$.

\subsection{Lower and upper bounds}\label{ss:lowup}
In this subsection we summarize the  lower and upper bounds from \cite{GW} that we will need in the sequel. In addition we slightly refine the upper bound.

We begin by pointing out that the proof of \cite[proposition 2.4]{GW} rules out functions with order one energy and large volumes of $u\approx 1$, as we summarize in the following lemma.
\begin{lemma}\label{nobigsets}
  Let $\xi\in(0,\xi_d]$ and consider the critical scaling \eqref{crit}. For every $\epsilon_0>0$, there exists $\phi_0>0$ such that for all $\phi\leq \phi_0$ and $u\in X_\phi$, there holds
  \begin{align*}
    \nu(u)\geq \epsilon_0\,(\phi \,L)^d\quad\Rightarrow \quad \ephi(u)\gg 1 \quad\text{and hence}\quad\ephi(u)>E_M \;(\text{cf.\ \eqref{emm}}).
  \end{align*}
\end{lemma}
\begin{proof}
First we point out that, according to the scaling \eqref{crit}, there holds
\begin{align}
  \epsilon_0\,(\phi \,L)^d=\frac{\epsilon_0 \xi^{d+1}}{\phi}.\label{epphi}
\end{align}
Also, we may without loss of generality assume that $\ephi(u)\lesssim 1$. Invoking  \cite[lemma 2.3]{GW} and estimating as in  \cite[proposition 2.4]{GW}, we observe
  \begin{align*}
    \int_{\tphil}e_\phi(u)\,\chi_2(u)\,dx\geq 0,\quad  \int_{\tphil}e_\phi(u)\,\chi_3(u)\,dx\gtrsim -\nu(u),
  \end{align*}
  while convexity, Jensen's inequality and the assumption $\nu(u)\geq \epsilon_0\,(\phi \,L)^d$ lead to
  \begin{align*}
     \int_{\tphil}e_\phi(u)\,\chi_1(u)\,dx\gtrsim \nu(u)^2.
  \end{align*}
  Choosing $\phi_0$ small enough so that $\nu(u)$ (according to \eqref{epphi}) is sufficiently large and adding the energy estimates leads to
  \begin{align*}
    \ephi(u)\gtrsim \nu(u)^2\qquad\text{and hence to }\qquad \ephi(u)\gg 1.
  \end{align*}
\end{proof}

The lower bound follows directly from \cite[proposition 2.4]{GW} after rescaling, applying \eqref{Regular.Isop.Torus2}, and noting that, according to the previous lemma, we may deduce from \eqref{em.sept} that $\nu(u)\leq \epsilon_0\,(\phi \,L)^d$ for $\epsilon_0=\epsilon_0(d)$ as in \cite[proposition 2.4]{GW}.
\begin{proposition}[Lower bound \cite{GW}]\label{lemma.lower.bound}
In the critical regime \eqref{crit}, for any $\omega>0$ and for $\phi>0$ sufficiently small, the following holds. If $u\in X_\phi$  satisfies
\begin{align}
\quad\ephi(u)\leq E_M, \label{em.sept}
\end{align}
then the energy is bounded below by
\begin{align*}
  \ephi(u)\geq C_1(\phi) \nu(u)^{(d-1)/d}-C_2(\phi)\nu(u),
\end{align*}
where
\begin{align*}
  C_1(\phi)=(1-8\phi^{1/3})^{1/2}(c_0-8\sqrt{2}\phi^{2/3})\sigma_d^{1/d}d^{(d-1)/d},\qquad C_2(\phi)=(2+\phi^{1/3})(2-3\phi).
\end{align*}
Moreover, if $u$ in addition satisfies
\begin{align*}
\nu(u)\geq \omega,
\end{align*}
then
\begin{equation}\label{lower.bound}
\ephi(u)\geq f_\xi (\nu(u))+I(u)+O(\phi^{1/3}),
\end{equation}
where
\begin{equation}\label{DefnIsopTerm}
I(u):=\int_{-1+2\phi^{1/3}}^{1-2\phi^{1/3}}\sqrt{2\tilde{G}(t)}\,\Big({\rm Per}_{\tphil}(\{u>t\})-{\rm P}_E(\{u>t\})\Big)\,dt.
\end{equation}
Here we have abbreviated
\[
\tilde{G}(t):=(1-8\phi^{1/3})G(t).
\]
\end{proposition}
The positive functional $I(\cdot)$ defined in \eqref{DefnIsopTerm} can be thought of as the extra term in the surface tension owing to the deviation from sphericity of the superlevel sets $\{u>t\}$.

The next proposition provides the upper bound on the energy that we need. It is based on an idea from \cite{BCK} and the upper bound construction is used explicitly in \cite[lemma 2.2]{CCELM} and later in  \cite[lemma 3.5]{GW}.
\begin{proposition}[Upper bound]\label{prop.upper.bound}
Consider $\xi\in(\tilde{\xi_d},\xi_d]$ and the critical regime \eqref{crit}, and fix $\omega_1,\omega_2$ such $0<\omega_1<\nu_s<\omega_2<\xi^{d+1}/2$. Then for small enough $\phi>0$ there exists a path $\hat{\psi}\in C([0,1], X_\phi)$ and a time $t_1\in(0,1)$ satisfying
\[\hat\psi(0)\equiv \bar{u},\quad \nu(\hat \psi(t_1))=\omega_1,\quad \nu(\hat{\psi}(1))=\omega_2,
\]
such that the energy on the first part of the time interval is bounded as
\[
\quad  \sup_{t\in[0,t_1]}\ephi(\hat{\psi}(t))= f_\xi(\omega_1)+ o(1),
\]
and such that the construction on the second part of the time interval is close in $L^2$ and energy to a sharp interface profile in the sense that, for all $t\in[t_1,1]$, there holds
\begin{equation}\label{upper.bound}
|\hat \psi (t)-\Psi(\cdot;\nu(\hat \psi (t)))|_{\tphil}^2\lesssim\phi|\ln\phi|,\quad \ephi(\hat \psi(t))=f_\xi(\nu(\hat \psi(t)))+O(\phi |\ln \phi|).
\end{equation}
\end{proposition}
\begin{proof}
As in \cite[proposition 3.1]{GW} the path is obtained by a construction which---in our current scaling---consists of a linear interpolation between $\bar{u}$ and a ``droplet-like state'' of small volume, cf.  \cite[lemma 3.2]{GW}, followed by a family of droplet-like states such that the volume grows from $\omega_1$ at $t=t_1$ to $\omega_2$ for the final state $\hat \psi(1)$, cf. \cite[lemma 3.5]{GW}.

We begin by recalling the construction of \cite[lemma 3.5]{GW}. For $R$ large and positive, let the function $v_R:\R \rightarrow [-1,1]$ be a smooth, odd function such that
\begin{align}\label{GW.v_R}
v_R(x) := \begin{cases} -\tanh\left(x/\sqrt{2}\right) & \textrm{for}\quad |x| <R \\ -\sgn (x) & \textrm{for}\quad |x| > 2R, \end{cases} \qquad
\end{align}
with a  monotone interpolation on $R\leq |x|\leq 2R$. Now for $\eta\in[0,1],$ let the radius $r_\eta$ be defined by
\begin{align}\label{GW.reta}
r_\eta:=\eta^{\frac{1}{d}} \left(\frac{\phi d}{2 \sigma _d}\right)^{\frac{1}{d}}L.
\end{align}
Then the construction $u_\eta:\mathbb{T}_L \mapsto \R$ is defined by
\begin{align}\label{GW.trialfunction}
u_\eta(x):=v_R(|x|-r_\eta)+\alpha(\eta),
\end{align}
where $\alpha(\eta)$ is a constant chosen such that $u_\eta\in \Big\{u\in H^1\cap L^4(\mathbb{T}_L): \dashint_{\mathbb{T}_L} u \, dx=-1+\phi\Big\},$ and we recall for reference below
that $\alpha(\eta)\lesssim\phi$; cf.\ \cite[equation (3.8)]{GW}. In order to match the notation of the current paper, we rescale space to define $\hat{u}_{\eta,\phi}:\tphil\to\R$ via
\begin{align*}
 \hat{u}_{\eta,\phi}(x):=u_\eta\left(\frac{x}{\phi}\right)\qquad\text{and}\qquad \hat{r}_\eta:=\phi r_\eta=\eta^{1/d}\xi^{(d+1)/d}\left(\frac{d}{2\sigma_d}\right)^{1/d}.
\end{align*}
It is with  the construction $\hat{u}_{\eta,\phi}$ that we will establish the proof.

Because the hyperbolic tangent has logarithmic tails and $\alpha(\eta)\lesssim\phi$, we deduce that
\begin{align}
\left|\hat{u}_{\eta,\phi}(x)+1\right|&\lesssim \phi\quad\text{ for $|x|-\hat{r}_\eta\gtrsim \phi|\ln\phi|$},\notag\\
\left|\hat{u}_{\eta,\phi}(x)-1\right|&\lesssim \phi\quad\text{ for $\hat{r}_\eta-|x|\gtrsim \phi|\ln\phi|$}.\label{11.b}
\end{align}
As a consequence of this bound, the volume $\hat{\nu}_\eta=\sigma_d(\hat{r}_\eta)^d/d$ is related to the ``volume-type'' functional defined in \eqref{voltp} via
\begin{align}
  \nu(\hat{u}_{\eta,\phi})\leq \hat{\nu}_\eta\leq \nu(\hat{u}_{\eta,\phi})+O(\phi|\ln \phi|)).\label{scp}
\end{align}
We use this fact in two ways. First, we observe that for every $\omega_1\leq\omega\leq\omega_2$, there is an $\eta\in(0,1)$ (bounded away from zero and one) such that
\begin{align}
  \nu(\hat{u}_{\eta,\phi})=\omega.\label{deli1}
\end{align}
We define $\eta:=\eta(\omega)$ and $\hat{u}_{\omega,\phi}:=\hat{u}_{\eta(\omega),\phi}$ using this value.

Using the energy bound from \cite[Remark 3.6]{GW} for all such $\eta$ values, we deduce---in the notation and scaling of the current paper---that
\begin{align*}
  \ephi(\hat{u}_{\omega,\phi})\leq f_\xi\big(\hat{\nu}_\eta\big)+O(\phi).
\end{align*}
A second application of \eqref{scp} delivers
\begin{align}
  \ephi(\hat{u}_{\omega,\phi})\leq f_\xi\left(\nu(\hat{u}_{\omega,\phi})\right)+O(\phi |\ln \phi|).\label{deli2}
\end{align}
Finally, we observe that \eqref{11.b}, \eqref{scp}, and \eqref{deli1} for $\omega_1\leq\omega\leq\omega_2$ imply
\begin{align}
  |\hat{u}_{\omega,\phi}-\Psi(\cdot;\omega)|_{\tphil}^2\lesssim\phi|\ln\phi|.\label{11.a}
\end{align}
Together \eqref{deli2} and \eqref{11.a} yield \eqref{upper.bound}.

The estimate
\[
\sup_{t\in[0,t_1]}\ephi(\hat{\psi}(t))= f_\xi(\omega_1)+ o(1),
\]
follows from the fact that, for that part of the path $t:[0,t_1]\to \hat \psi (t)$ that consists of a convex combination of $\bar{u}$ and a suitably small droplet-like state the energy stays well below $f_\xi(\omega_1)$, cf.  \cite[lemma 3.2]{GW}, while for the rest of this path the energy is given by an estimate similar to \eqref{deli2}.
\end{proof}


\subsection{Elementary bounds}
In this subsection we collect several basic but important estimates. The first lemma summarizes $L^2$ and measure bounds for functions of bounded energy. As above, we will abbreviate $\kappa=\phi^{1/3}$ when we bracket the values of $u$.
\begin{lemma}\label{l:bulk}
Consider $\xi\in (0,\xi_d]$ and the critical scaling \eqref{crit}. Suppose that  $u\in X_\phi$ satisfies
\begin{equation}\label{ebd}
\ephi(u)\leq E_M,
\end{equation}
where $E_M$ is defined in \eqref{emm}.
Then the following holds true:
\begin{align}
\int_{\cala\cup\calb}(u+1)^2\,dx+\int_{\cald\cup\cale}\big(u-1\big)^2\,dx&\lesssim \phi,\label{i}\\
\int_{\calc}(u+1)^2\,dx+\int_{\calc}\big(u-1\big)^2\,dx\lesssim |\calc|&\lesssim \phi^{1/3},\label{ii}\\
|\cala|\lesssim\int_{\cala}\big(u-1\big)^2\,dx&\lesssim \phi^{1/3},\label{iii''}\\
|\cale|\lesssim\int_{\cale}(u+1)^2\,dx&\lesssim \phi^{1/3}.\label{iii}
\end{align}
\end{lemma}
One consequence of the lemma is that the volume of suitable superlevel sets of bounded energy functions are close to the volume $\nu(u)$.
\begin{corollary}\label{cor.1}
 Consider $\xi\in (0,\xi_d]$ and the critical scaling \eqref{crit}. Consider $u\in X_\phi$ such that $\ephi(u)\leq E_M$.
The superlevel sets for $s\in[-1+\kappa,1-\kappa]$ satisfy
\begin{equation}\label{saddle.superset}
|\{u>s\}|=\nu(u)+O(\phi^{1/3}).
\end{equation}
\end{corollary}
Next we bound the $L^2$ distance to a sharp-interface function in terms of the Fraenkel asymmetry of $\{u\geq1-2\kappa\}$.
\begin{lemma}\label{l:l2fr}
Consider $\xi\in (0,\xi_d]$ and the critical scaling \eqref{crit}.
For any  $u\in X_\phi$ such that
\begin{align}
 \nu(u)\leq\frac{\xi^{d+1}}{2} \qquad\text{and}\qquad \ephi(u)\leq E_M, \label{nottoobig}
\end{align}
there holds
\begin{align}
  |u-\Psi(\cdot;\nu(u))|_{\tphil}^2\lesssim \lambda\left(\{u\geq 1-2\kappa\}\right)+O(\phi^{1/3}).\notag
\end{align}
\end{lemma}

Our next lemma establishes that the Fraenkel asymmetry of the superlevel sets are comparable in the following sense.
\begin{lemma}\label{l:frel}
Consider the critical scaling \eqref{crit} with $\xi\in (0,\xi_d]$ and suppose that $u\in X_\phi$ satisfies \eqref{nottoobig}.
The Fraenkel asymmetry of the intermediate superlevel sets of $u$ are comparable in the sense that
\begin{equation}\label{lambda.ineq.1}
\inf_{s\in[-1+2\kappa,1-2\kappa]}\lambda(\{u>s\})\geq \sup_{s\in[-1+2\kappa,1-2\kappa]}\lambda(\{u>s\})+\frac{O(\phi^{1/3})}{\nu(u)}.
\end{equation}
\end{lemma}
Finally, we remark that one can show, via a mild adaptation of \cite[lemma 2.3]{GW}, that the volume-constrained minimizers are bounded.
\begin{lemma}\label{l:linf}
Consider $\xi\in (0,\xi_d]$, the critical scaling \eqref{crit}, and $\omega_1>0$.  For $\phi>0$ sufficiently small, the volume-constrained minimizer $\uo$ for any $\omega\in [\omega_1,\xi^{d+1}/2]$ satisfies
 \begin{align}
  -1-\phi^{1/3}\leq \einf\; \uo\leq \esup\; \uo \leq 1+\phi^{1/3}.\label{einfsup}
 \end{align}
\end{lemma}
We now present the proofs of these elementary facts, with the exception of lemma \ref{l:linf}, whose proof is longer and is included in subsection \ref{ss:linf}.
\begin{proof}[Proof of lemma \ref{l:bulk}]
From
\[
\ephi (u)\geq  \frac{1}{\phi}\int_{\tphil} G(u) \, dx-\frac{G(-1+\phi)}{\phi}|\tphil|,
\]
together with \eqref{f2} and \eqref{ebd}, we deduce that
\begin{equation}\label{int.G.bound}
\int_{\tphil} G(u)\,dx\lesssim \phi,
\end{equation}
To obtain \eqref{i}, we observe that
\[
G(u)\gtrsim (u+1)^2\,\,\text{on}\,\,\cala\cup\calb\quad\text{and}\quad G(u)\gtrsim (u-1)^2\,\,\text{on}\,\,\cald\cup\cale,
 \]
and combine this with \eqref{int.G.bound}.

Estimates \eqref{ii}, \eqref{iii''}, and \eqref{iii} follow from \eqref{int.G.bound} and
\[
G(u)\gtrsim \phi^{2/3}\,\,\text{on}\,\, \calc,\quad G(u)\gtrsim \phi^{2/3}(u-1)^2\,\,\text{on}\,\,\cala,\,\,\, \text{and}\,\,\,G(u)\gtrsim \phi^{2/3}(u+1)^2\,\,\text{on}\,\,\cale,
\]
respectively.
\end{proof}

\begin{proof}[Proof of corollary \ref{cor.1}]
 By the definition of the volume functional $\nu$ it follows that
\[
|\{u>1-\kappa\}|\leq \nu(u)\leq |\{u>-1+\kappa\}|.
\]
Thus for any $s\in[-1+\kappa,1-\kappa]$ there holds
\[
-|\{-1+\kappa<u\leq s\}|\leq |\{u>s\}|-\nu(u)\leq |\{s<u\leq 1-\kappa\}|,
\]
which as a consequence of \eqref{ii} implies \eqref{saddle.superset}.
\end{proof}


\begin{proof}[Proof of lemma \ref{l:l2fr}]
We abbreviate $\Psi_u:=\Psi(x;\nu(u))$, and we denote by $B(x)$ and $\tilde{B}(x)$ the balls with center $x$ and volume $\nu(u)$ and $|\{u\geq 1-2\kappa\}|$, respectively. In addition we recall the ``triangle inequality'' for the symmetric difference
\begin{align}
|A\triangle B|\leq |A\triangle C|+|C\triangle B|.\label{sdtri}
\end{align}

According to lemma \ref{l:bulk}, one has
\begin{eqnarray*}
 \lefteqn{ ||u-\Psi_{u}(\cdot-x)||_{L^2(\tphil)}^2}\\
 &=&\int_{\tphil}(u-(-1))^2\textbf{1}_{\Psi_u(\cdot-x)=-1}\,dx+\int_{\tphil}(u-1)^2\textbf{1}_{\Psi_u(\cdot-x)=1}\,dx\\
  &\leq &\int_{\tphil}(u-(-1))^2\textbf{1}_{\Psi_u(\cdot-x)=-1,\, 1-\kappa\leq u\leq 1+\kappa}\,dx\\
  &&\quad+\int_{\tphil}(u-1)^2\textbf{1}_{\Psi_u(\cdot-x)=1,\, -1-\kappa\leq u\leq -1+\kappa}\,dx +O(\phi^{1/3})\\
  &\lesssim &\left|B(x)\triangle \{u\geq 1-2\kappa\}\right|+O(\phi^{1/3})\\
  &\overset{\eqref{saddle.superset},\eqref{sdtri}}\lesssim &\left|\tilde{B}(x)\triangle \{u\geq 1-2\kappa\}\right|+O(\phi^{1/3})\\
  &\overset{\eqref{nottoobig}}\lesssim &\frac{\left|\tilde{B}(x)\triangle \{u\geq 1-2\kappa\}\right|}{|\{u\geq 1-2\kappa\}|}+O(\phi^{1/3}),
  \end{eqnarray*}
  where in the last step we also used the fact that
\[
|\{u\geq 1-2\kappa\}|\leq \nu(u)+|\{1-2\kappa\leq u\leq 1-\kappa\}|\overset{\eqref{ii}}\leq \frac{\xi_d^{d+1}}{2}+O(\phi^{1/3})\lesssim 1.
\]
Infimizing over $x$ leads to
\begin{align}
 |u-\Psi_u|_{\tphil}^2  &\lesssim \lambda\left(\{u\geq 1-2\kappa\}\right)+O(\phi^{1/3}),\notag
\end{align}
as desired.
\end{proof}
\begin{proof}[Proof of lemma \ref{l:frel}]
The proof uses \eqref{saddle.superset} and \eqref{sdtri}.
Let  $x \in \tphil$ and for $t\in[-1+2\kappa,1-2\kappa]$ denote by $B_t(x)\subset \tphil$ the ball with center $x$ and volume $|\{u>t\}|$.
For any $s\in[-1+2\kappa,1-2\kappa]$ we have
\begin{eqnarray*}
\lambda(\{u>t\})&\leq&\frac{|B_t(x)\triangle \{u>t\}|}{|\{u>t\}|}\\
 &\overset{\eqref{sdtri}}\leq& \frac{|B_s(x)\triangle \{u>s\}|+|B_s(x)\triangle B_t(x)|+|\{u>s\}\triangle \{u>t\}|}{|\{u>t\}|} \\
& =&\frac{|B_s(x)\triangle \{u>s\}|}{|\{u>s\}|}\, \frac{|\{u>s\}|}{|\{u>t\}|}+\frac{|B_s(x)\triangle B_t(x)|+|\{u>s\}\triangle \{u>t\}|}{|\{u>t\}|} \\
&\overset{\eqref{saddle.superset}}\leq& \frac{|B_s(x)\triangle \{u>s\}|}{|\{u>s\}|}+\frac{O(\phi^{1/3})}{\nu(u)},
\end{eqnarray*}
which implies \eqref{lambda.ineq.1}.
\end{proof}



\subsection{Structure of the limit energy}
Our existence proofs rely on the structure of the limit energy $\ez$. In the following two lemmas we analyze $\ez$  near $\Psi_m$ and $\Psi_s$. The proofs are straightforward but we include them for completeness and because the proof of lemma \ref{la} serves as the backbone for the proof of lemma \ref{l:alt}. Similarly, the proof of lemma \ref{la2} could be used to prove a finite-$\phi$ analogue.
\begin{lemma}[Sharp interface energy near the local minimizer]\label{la}
 Fix $\xi \in (\tilde{\xi_d},\xi_d]$ and define $\gamma_0$ as in \eqref{gm0}. The function $\Psi_m$ is a strict local minimizer of $\ez$ in the sense that for all $\gamma\in (0,\gamma_0]$ the following hold.
  \begin{itemize}
    \item (Local minimizer) The function $\Psi_m$ locally minimizes $\ez$:
    \begin{align}
      \inf_{|u-\Psi_m|_{\R^d}\leq\gamma}\ez(u)=\ez(\Psi_m).\label{A1}
    \end{align}
    \item (Strictness) There exists $\delta>0$ such that
    \begin{align}
     \gamma\leq |u-\Psi_m|_{\R^d}\leq\gamma_0\qquad\Rightarrow\qquad \ez(u)-\ez(\Psi_m)\geq 2\delta.\label{A2}
    \end{align}
  \end{itemize}
\end{lemma}
\begin{lemma}[Sharp interface energy near the saddle point]\label{la2}
Fix $\xi \in (\tilde{\xi_d},\xi_d]$ and define $\gamma_0$ as in \eqref{gm0}. The function $\Psi_{s}$ is a saddle point of $\ez$ in the sense that for all $\gamma\in(0,\gamma_0]$, the following hold.
   \begin{itemize}
    \item (Minimizer at volume $\nu_s$) $\Psi_{s}$ minimizes $\ez$ subject to a volume constraint:
    \begin{align}
      \inf_{\substack{|u-\Psi_{s}|_{\R^d}\leq\gamma\\ \nu_0(u)=\nu_s}}\ez(u)=\ez(\Psi_{s}).\label{A1b}
    \end{align}

    \item (Strictness at volume $\nu_s$) There exists $\delta>0$ such that
    \begin{align}
      |u-\Psi_{s}|_{\R^d}\geq\gamma,\, \nu_0(u)=\nu_s\qquad\Rightarrow\qquad \ez(u)\geq \ez(\Psi_{s})+2\delta.\label{A2b}
    \end{align}

    \item (Descent direction) There exists $\delta>0$ such that
    \begin{align}
     \inf_{ |u-\Psi_{s}|_{\R^d}\leq\gamma}\ez(u)\leq \ez(\Psi_{s})-2\delta.\label{A3b}
    \end{align}
  \end{itemize}
\end{lemma}
The proofs of the lemmas rely on the quantitative isoperimetric inequality in $\R^d$.
\begin{proof}[Proof of lemma \ref{la}]
Since \eqref{A1} follows from the fact that \eqref{A2} holds for every $\gamma\in (0,\gamma_0]$, it suffices to show \eqref{A2}.
To this end, consider $u$ such that $|u-\Psi_m|_{\R^d}=\tilde{\gamma}\in[\gamma,\gamma_0]$.

We may without loss of generality assume that $u=\pm 1$ a.e.\ and ${\rm Per}(\{u=+1\})<\infty$, since otherwise $\mathcal{E}_0^\xi(u)=\infty$ and \eqref{A2} is trivially satisfied.
Hence we may write $u=-1+2\chi_\mathcal{S}$ with ${\rm Per}(\mathcal{S})<\infty$, so that
\begin{equation}\label{E_0^xi}
\mathcal{E}_0^\xi(u)=c_0{\rm Per}(\mathcal{S})-4|\mathcal{S}|+4\xi^{-(d+1)}|\mathcal{S}|^2.
\end{equation}
Moreover, from $u=\pm 1$ and  $|u-\Psi_m|_{\R^d}=\tilde{\gamma}$, it follows that
\begin{equation}\label{SymDif1}
\frac{\tilde{\gamma}^2}{4}=\min_{x\in \mathbb{R}^d}|B_{\nu_m}(x)\triangle \mathcal{S}|=|B_{\nu_m}(x_0)\triangle \mathcal{S}|
\end{equation}
for some $x_0\in \mathbb{R}^d$.
From \eqref{SymDif1} and
\begin{equation}\label{set.ineq.1}
|B_{\nu_m}(x_0)|-|B_{\nu_m}(x_0)\triangle \mathcal{S}|\leq |\mathcal{S}|\leq |B_{\nu_m}(x_0)|+|B_{\nu_m}(x_0)\triangle \mathcal{S}|,
\end{equation}
one observes that
\begin{equation}\label{C_interval}
|\mathcal{S}|\in \Big[\nu_m-\frac{\tilde\gamma^2}{4},\nu_m+\frac{\tilde\gamma^2}{4}\Big].
\end{equation}
On the one hand, if
\begin{align*}
  \big||\mathcal{S}|-\nu_m\big|\geq\frac{\tilde\gamma^2}{8}\geq\frac{\gamma^2}{8},
\end{align*}
then
we deduce from \eqref{E_0^xi} and the isoperimetric inequality (cf.\ theorem \ref{Regular.Isop})  that
    \begin{equation}\label{delta_1}
    f_\xi(|\mathcal{S}|)\geq f_\xi(\nu_m)+\delta_1\qquad \text{for some}\,\, \delta_1>0.
    \end{equation}
It thus suffices to consider
\begin{align*}
 \big||\mathcal{S}|-\nu_m\big|\leq \frac{\gamma^2}{8}.
\end{align*}
In this case, we claim that the Fraenkel asymmetry $\lambda_\mathcal{S}$ of $\mathcal{S}$ satisfies \begin{equation}\label{lambda_C}
    \lambda_\mathcal{S}:=\min_{x\in \mathbb{R}^d}\frac{\big |B_{|\mathcal{S}|}(x)\triangle \mathcal{S}\big|}{|\mathcal{S}|}\geq \lambda_0\qquad\text{for some}\;\lambda_0>0.
    \end{equation}
Indeed, recalling \eqref{SymDif1}, we have for every $x\in \mathbb{R}^d$ that
    \[
    \big|B_{|\mathcal{S}|}(x)\triangle \mathcal{S}\big|\geq \big|B_{\nu_m}(x_0) \triangle \mathcal{S}\big|-\big||\mathcal{S}|-|B_{\nu_m}(x_0)|\big|=\frac{\tilde\gamma^2}{4}-\big||\mathcal{S}|-\nu_m\big| \geq \frac{\tilde\gamma^2}{8}\geq \frac{\gamma^2}{8},
    \]
    which in particular implies
    \[
    \frac{ \big|B_{|\mathcal{S}|}(x)\triangle \mathcal{S}\big|}{|\mathcal{S}|}\geq \frac{\gamma^2}{8|\mathcal{S}|}\overset{\eqref{C_interval}}\geq \frac{\gamma^2}{8(\nu_m+\gamma^2/8)}=:\lambda_0.
    \]
  Minimizing over $x$ yields \eqref{lambda_C}. Now we use the sharp  isoperimetric inequality \eqref{Regular.Sharp.Isop} and
\begin{equation}\label{gamma_0.2}
{\rm Per}(B_\nu)\geq {\rm Per}(B_{\nu_s})\qquad\text{for all }\nu \in [\nu_m-\gamma_0^2/4,\nu_m+\gamma_0^2/4]
\end{equation}
to deduce
    \begin{eqnarray}
    \ez(u)&\overset{\eqref{Regular.Sharp.Isop},\eqref{lambda_C}} \geq& c_0 {\rm P}_E(\mathcal{S})-4|\mathcal{S}|+4\xi^{-(d+1)}|\mathcal{S}|^2+c_0C(d){\rm P}_E(\mathcal{S})\lambda_0^2 \notag \\
    &=& f_\xi(|\mathcal{S}|)+c_0C(d){\rm P}_E(\mathcal{S})\lambda_0^2 \notag \\
    &\overset{\eqref{gamma_0.1},\eqref{gamma_0.2}}\geq& f_\xi(\nu_m)+c_0C(d){\rm P}_E(B_{\nu_s})\lambda_0^2 \notag \\
    &=& \ez(\Psi_m)+  c_0C(d){\rm P}_E(B_{\nu_s})\lambda_0^2,
    \end{eqnarray}
which establishes \eqref{A2} with $\delta= c_0C(d){\rm P}_E(B_{\nu_s})\lambda_0^2$.
\end{proof}



\begin{proof}[Proof of lemma \ref{la2}]
The proof is analogous to that of lemma \ref{la}. We remark that \eqref{A1b} follows from the fact that \eqref{A2b} holds for every $\gamma\in(0,\gamma_0]$.
Also note that it is a simple matter to establish \eqref{A3b} since it suffices to observe that, for example, for  $u_\nu:=-1+2\chi_{B_\nu}$ with
\[
\nu=\nu_s+\frac{\gamma^2}{4},
\]
we have $ |u_\nu-\Psi_{s}|_{\R^d}=\gamma$ and
\[
\ez(u_\nu)=f_\xi(\nu)=f_\xi(\nu_s)-\delta=\ez(\Psi_{s})-\delta,
\]
for some $\delta>0$.

Hence it suffices to show \eqref{A2b}. As above, we may assume without loss of generality that $u=-1+2\chi_\mathcal{S}$ with ${\rm Per}(\mathcal{S})<\infty$ and $|\mathcal{S}|=\nu_s$, in which case there holds
\begin{equation}\label{E_0^xi2}
\mathcal{E}_0^\xi(u)=c_0{\rm Per}(\mathcal{S})-4\nu_s+\frac{4\nu_s^2}{\xi^{d+1}}.
\end{equation}
Because of the constraint $|\mathcal{S}|=\nu_s$, we observe as for \eqref{lambda_C}  that $|u-\Psi_s|_{\R^d}\geq \gamma$ implies
\begin{equation}\label{lambda_s}
\lambda(\mathcal{S})\geq \frac{\gamma^2}{8\nu_s}=:\lambda_0.
\end{equation}
Hence \eqref{E_0^xi2}, \eqref{lambda_s}, and the quantitative isoperimetric inequality \eqref{Regular.Sharp.Isop} yield
\eqref{A2b} with
\[
\delta:=c_0C(d){\rm P}_E(B_{\nu_s})\lambda_0^2.
\]
\end{proof}
\subsection{Constrained minimizers are bounded}\label{ss:linf}
Here we prove lemma \ref{l:linf}. The proof builds on \cite[lemma 2.3]{GW}.
\begin{proof}[Proof of lemma \ref{l:linf}]
As usual, we set $\kappa=\phi^{1/3}$ and assume $\phi\leq \phi_0$ for an appropriately chosen value $\phi_0$. We proceed in two steps.

\underline{Step 1: Upper bound.}
The first step is a reformulation of a specific case of \cite[lemma 2.3]{GW}. We show, for any $u\in X_\phi$, if $|\{x\colon u(x)>1+\kappa\}|>0$, then there exists a function $\tilde{u}\in X_\phi$ such that
\begin{itemize}
\item   [(i)] $\nu(\tilde{u})=\nu(u)$,
\item   [(ii)] $\tilde{u}=u$ on $\{x:-1+\phi\leq u(x)\leq 1+\kappa\}$,
\item   [(iii)] $ \esup\; \tilde{u}\leq 1+\kappa$,
\item   [(iv)] $\ephi(\tilde{u})< \ephi(u)$.
\end{itemize}
For completeness, we recount the proof.  Recall the definition $\cale:=\{x\colon u(x)>1+{\kappa}\}$ and let $\hat{\calb}:=\{x\colon u(x)< -1+\phi\}.$ Notice that $|\cale|>0$ and the mean condition imply $|\hat{\calb}|>0$.
For $\lambda \in [0,1]$ we define the function
\begin{align*}
\tilde{u}_\lambda(x):=
\begin{cases}
\min\{u(x),1+\kappa \}\quad&\text{for}\quad x\in \tphil\setminus\hat{\calb}\\
(1-\lambda)u(x)+\lambda (-1+\phi)\quad&\text{for}\quad x\in \hat{\calb}.
\end{cases}
\end{align*}
It is easy to see that
$$\dashint_{\tphil} \tilde{u}_0\,dx<-1+\phi, \qquad \dashint_{\tphil} \tilde{u}_1\,dx>-1+\phi.$$
Hence there exists $\lambda_*\in(0,1)$ such that $$\dashint_{\tphil} \tilde{u}_{\lambda_*}dx=-1+\phi.$$
The function $\tilde{u}:=\tilde{u}_{\lambda_*}$ belongs to $X_\phi$ and satisfies properties (i)-(iii).
It remains to check  (iv). The energy difference can be written as
\begin{align}
\ephi(\tilde{u})-\ephi(u)&=\int_{\hat{\calb}\cup\cale}\frac{\phi}{2}|\nabla \tilde{u}|^2-\frac{\phi}{2}|\nabla u|^2+\frac{1}{\phi}\Big(G(\tilde{u})-G(u)\Big)\,dx\notag\\
&\leq \frac{1}{\phi}\int_{\hat{\calb}\cup\cale}G(\tilde{u})-G(u)\,dx,\label{energy ineq}
\end{align}
where we have used (ii) and the fact that the gradient energy of $\tilde{u}$ is smaller than that of $u$.
Since $u
\leq \tilde{u}\leq -1+\phi$ on $\tilde{\calb}$, the convexity of $G$ on $(-\infty,-1+\phi]$ implies
\begin{align*}
G(\tilde{u})-G(u)\leq G'(\tilde{u})(\tilde{u}-u)\leq G'(-1+\phi)(\tilde{u}-u).
\end{align*}
On $\cale$ on the other hand, $\tilde{u}=1+\kappa$ and the convexity of $G$ on $[1+\kappa, \infty)$ imply that
$$G(1+\kappa )-G(u)\leq -G'(1+\kappa)\big(u-\tilde{u}\big).$$
Inserting these two inequalities into \eqref{energy ineq} yields
\begin{align}
\ephi(\tilde{u})-\ephi(u)
&\leq \frac{G'(-1+\phi)}{\phi}\int_{\hat{\calb}}\tilde{u}-u\,dx+\frac{G'(1+\kappa)}{\phi}\int _{\cale}\tilde{u}-u\,dx\notag\\
&< \frac{G'(1+\kappa )}{\phi}\int_{\hat{\calb}\cup\cale}\tilde{u}-u\,dx=0,
\end{align}
where the final equality is a consequence of (ii) and $\int _{\tphil}\tilde{u}\,dx=\int _{\tphil}u\,dx$. The last inequality is strict since $G'(-1+\phi)<G'(1+\kappa)$ and because the assumption $|\cale|>0$ implies $\int_{\cale}(\tilde{u}-u)\,dx<0$ and therefore $\int_{\hat{\calb}}(\tilde{u}-u)\,dx>0$.

\underline{Step 2: Lower bound.}
We show, for any $u\in X_\phi$, if
\begin{align*}
|\{x\colon u(x)<-1-\kappa\}|>0\qquad \text{and}\qquad \ephi(u)\leq E_M,
\end{align*}
then there exists a function $\tilde{u}\in X_\phi$ such that
\begin{itemize}
\item   [(i)] $\nu(\tilde{u})=\nu(u)$,
\item   [(ii')] $\tilde{u}=u$ on $\{x:-1-\kappa\leq u(x)\leq 1-\kappa\}$,
\item   [(iii')] $ \einf\; \tilde{u}\geq -1-\kappa$,
\item   [(iv)] $\ephi(\tilde{u})< \ephi(u)$.
\end{itemize}
Recall the definition $\cala:=\{x\colon u(x)<-1-\kappa\}$ and let $\hat{\cald}:=\{x\colon u(x)>1-\kappa\}$.
For $\lambda \in [0,1]$ we define the function
\begin{align*}
\tilde{u}_\lambda(x):=
\begin{cases}
\max\{u(x),-1-\kappa \}\quad&\text{for}\quad x\in \tphil\setminus\hat{\cald}\\
(1-\lambda)u(x)+\lambda (1-\kappa)\quad&\text{for}\quad x\in \hat{\cald}.
\end{cases}
\end{align*}
It is easy to see that
$$\dashint_{\tphil} \tilde{u}_0\,dx>-1+\phi.$$
In order to check the mean of $\tilde{u}_1$, we deduce from \eqref{iii''}, \eqref{int.G.bound}, and the Cauchy-Schwarz inequality that
\begin{align}
 \int_{\cala}\left|u-(-1-\kappa)\right|\,dx&\leq \left(|\cala|\int_{\cala}\left|u-(-1-\kappa)\right|^2\,dx\right)^{1/2}\notag\\
&\lesssim \left(|\cala|\int_{\cala}G(u)\,dx\right)^{1/2}\lesssim \phi^{2/3}.\label{csl}
\end{align}
We use this fact to estimate
\begin{eqnarray}
  \int_{\tphil}(u-\tilde{u}_1)\,dx&=&\int_{\cala}u-(-1-\kappa)\,dx+\int_{\hat{\cald}}u-(1-\kappa)\,dx\notag\\
  &\overset{\eqref{csl}}\geq& -O(\phi^{2/3})+\int_{\{u>1-\kappa/2\}}u-(1-\kappa)\,dx\notag\\
  &\geq& -O(\phi^{2/3})+\frac{\phi^{1/3}}{2}\left(\omega_1+O(\phi^{1/3})\right)>0,\label{fomu}
\end{eqnarray}
where we have as in lemma \ref{l:bulk} deduced
\begin{align*}
  \left|\left\{x\colon 1-\kappa\leq u(x)\leq 1-\frac{\kappa}{2}\right\}\right|\lesssim \phi^{1/3}.
\end{align*}
From \eqref{fomu} we obtain
$$\dashint_{\tphil} \tilde{u}_1\,dx<-1+\phi.$$
Hence, as in step 1, there exists $\lambda_*\in(0,1)$ such that $$\dashint_{\tphil} \tilde{u}_{\lambda_*}dx=-1+\phi.$$
As in step 1, we define $\tilde{u}:=\tilde{u}_{\lambda_*}\in X_\phi$, for which (i), (ii'), and (iii') follow immediately, and
it remains only to check (iv).
Analogously to above, we observe that
\begin{align}
\ephi(\tilde{u})-\ephi(u)&\leq \frac{1}{\phi}\int_{\cala\cup\hat{\cald}}G(\tilde{u})-G(u)\,dx,\label{energg}
\end{align}
and deduce from convexity of $G$ on $(-\infty,-1-\kappa]$ and on $[1-\kappa,\infty)$ that
\begin{align*}
  G(-1-\kappa)-G(u)&\leq -G'(-1-\kappa)\big(u-\tilde{u}\big)   \qquad\text{on }u<-1-\kappa,\\
  G(\tilde{u})-G(u)&\leq -G'(\tilde{u})\big(u-\tilde{u}\big)\qquad \text{on }u>1-\kappa,
\end{align*}
respectively. Note that, since in after Step 1 we can assume $u\leq 1+\kappa$ on $\hat{\cald}$, there holds $1-\kappa\leq \tilde{u}\leq 1+\kappa-2\lambda_* \kappa$ on $\hat{\cald}$. Since $\lambda_*\in(0,1)$ one can check that $\sup_{\hat{\cald}}|G'(\tilde{u})|<-G'(-1-\kappa)$. Substituting into \eqref{energg}, we conclude as in step 1 that
\begin{align*}
  \ephi(\tilde{u})-\ephi(u)&<\frac{-G'(-1-\kappa)}{\phi}\int_{\cala\cup\hat{\cald}}u-\tilde{u}\,dx=0.
\end{align*}
\medskip

The combination of steps 1 and 2 and the characterization of $\uo$ as a minimizer of $\ephi$ subject to $\nu(u)=\omega$ imply \eqref{einfsup}.
\end{proof}


\section{Existence and properties via the $\Gamma$-limit and isoperimetric inequalities}\label{S:min}
In this section we use the $\Gamma$-limit in order to deduce existence and properties of  critical points.
We begin with the proof of theorem \ref{t:exist}, using the upper bound constructions from  proposition \ref{prop.upper.bound}, the structure of the limit energy from  lemma \ref{la}, and  lemma \ref{lc} below, which establishes a uniform lower bound on the energy of functions that are $\gamma$ away from $\Psi_m$. We will refer to lemma \ref{lc} as the finite $\phi$ estimate for the local minimizer. In subsection \ref{ss:sadtoo}, we prove proposition \ref{t:exus} and theorem \ref{t:saddle}, also introducing the finite $\phi$ estimate for the saddle point. Finally, we prove theorem \ref{t:spheres} in subsection \ref{ss:dev}, quantifying the sphericity of the constrained minimizers.

\subsection{Local minimizer}\label{ss:locmin}
The finite $\phi$ estimate for the local minimum is the following.
\begin{lemma}[Finite $\phi$ estimate]\label{lc}
Consider $\xi\in(\tilde{\xi_d},\xi_d]$ and the critical scaling \eqref{crit}. Define $\gamma_0$ as in \eqref{gm0}. For every $\gamma\in (0,\gamma_0]$ and $\delta>0$, there exists $\phi_0>0$ such that for $\phi\leq \phi_0$, one has for all $u\in X_\phi$ that
\begin{align}
\gamma\leq |u-\Psi_m|_{\tphil}\leq\gamma_0\qquad\Rightarrow\qquad \ephi(u)\geq \inf_{\gamma\leq|u-\Psi_m|_{\R^d} \leq\gamma_0}\ez(u)-\delta.\label{C}
\end{align}
\end{lemma}
As remarked in the introduction, we cannot get compactness in the usual way since $\tphil$ grows to $\R^d$ in the limit. The proof of lemma \ref{lc}, which is fairly involved, is given in subsection \ref{ss:pflemmas}. In fact, we only use this information in  the proof of theorem \ref{t:exist} in the form of the simpler  estimate \eqref{thasall}. In lemma \ref{l:alt} we point out that \eqref{thasall} can be established by a conceptually simpler and much shorter proof.
We include lemma \ref{lc} because we believe that it is interesting in its own right---indeed, it tells us that being bounded away from the limit minimizer creates the \emph{same energetic penalization}, up to a small error, as it creates in the limit energy. The same remark holds true for the corresponding saddle point estimates. See also remarks \ref{rem:all} and \ref{rem:pass}.

\begin{proof}[Proof of theorem \ref{t:exist}]
For any $\phi>0$, the direct method of the calculus of variations yields a function $\um\in X_\phi$ that minimizes $\ephi$ subject to the constraint $|u-\Psi_m|_{\tphil}\leq \gamma_0$.
For $\gamma\in (0,\gamma_0]$, let $\delta>0$ be the constant given in \eqref{A2}. Lemma \ref{lc} says, for $\phi$ sufficiently small, that functions with $\gamma\leq |u-\Psi_m|_{\tphil}\leq\gamma_0$ satisfy
\begin{align}
  \ephi(u)\geq \inf_{\gamma\leq|u-\Psi_m|_{\R^d} \leq\gamma_0}\ez(u)-\delta\overset{\eqref{A2}}\geq c_m+\delta.\label{thasall}
\end{align}
At the same time,
according to \eqref{upper.bound} with $\omega_2=\nu_m$ and $t=1$, there exists $\hat{u}_{\omega,\phi}$ such that $|\hat{u}_{\omega,\phi}-\Psi_m|_{\tphil}\leq \gamma_0$ and
\begin{equation}\label{t:exist.constr.}
  \ephi(\hat{u}_{\omega,\phi})= c_m+ O(\phi |\ln \phi|)<c_m+\frac{\delta}{2}
\end{equation}
for $\phi$ sufficiently small. We deduce on the one hand that $\um$ is an unconstrained local minimizer of $\ephi$ (since it belongs to the interior of the $L^2$-$\gamma_0$ ball around $\Psi_m$) and on the other hand that \eqref{ump} holds.

To address the estimates \eqref{en.cl}-\eqref{num}, we begin by establishing rough closeness to $\nu_m$ in volume. For the lower bound, we use lemma \ref{l:bulk}, and assume without loss of generality that $|\um-\Psi_m|_{\tphil}=||\um-\Psi_m||_{L^2(\tphil)}$, to argue
\begin{eqnarray*}
  \nu(\um)&\geq& \left|\{1-\phi^{1/3}\leq \um\leq 1+\phi^{1/3}\}\right|\\
 &\geq& \nu_m-\left|\{x\in B_{\nu_m}\colon \um<1-\phi^{1/3}\;\text{or}\;\um>1+\phi^{1/3} \} \right|\\
 &\overset{\eqref{ii},\eqref{iii}}=& \nu_m-\left|\{x\in \tphil\colon\Psi_m=1\;\text{and}\; \um<-1+\phi^{1/3} \} \right|+o(1)\\
 &\overset{\eqref{ump}}\geq& \nu_m-\gamma^2+o(1).
\end{eqnarray*}
For the corresponding upper bound, note that
\[
\gamma^2\overset{\eqref{ump}}\geq\left|\{\um\geq 1-\phi^{1/3}\;\text{and}\;\Psi_m=-1\}\right|\geq \left|\{\um \geq1-\phi^{1/3}\}\right|-\nu_m\overset{\eqref{saddle.superset}}=\nu(\um)-\nu_m+o(1).
\]
One thus obtains for $\phi_0>0$ sufficiently small that
 \begin{align}
   |\nu(\um)-\nu_m|&\lesssim \gamma^2 \qquad\text{for all $0<\phi\leq \phi_0$}.\label{vol.77}
\end{align}
In particular, for $\gamma$ small enough, $\nu(\um)$ falls within a neighborhood of $\nu_m$ on which $f_\xi$ is convex and has $c_m$ as a minimum value.
We use these facts to improve to the quantitative estimates. Combining \eqref{t:exist.constr.} and \eqref{lower.bound} leads on the one hand to
\begin{align}
O(\phi^{1/3})\geq   \ephi(\hat{u}_{\omega,\phi})-c_m\geq \ephi(\um)-c_m &\geq f_\xi(\nu(\um))-c_m+O(\phi^{1/3})\label{fife}\\
&\geq O(\phi^{1/3}),\notag
\end{align}
which is \eqref{en.cl}. On the other hand, \eqref{vol.77} justifies a Taylor approximation of $f_\xi$ on the right-hand side of \eqref{fife}, which
implies
\begin{align}
O(\phi^{1/3})\geq  \ephi(\um)-c_m\geq  |\nu(\um)-\nu_m|^2/C+O(\phi^{1/3})\notag
\end{align}
for a constant $C=C(\xi,d)$. From here one deduces \eqref{num}.

\end{proof}


\begin{lemma} \label{l:alt}
Let $\gamma_0$ be as in \eqref{gm0}. For any $\gamma \in (0,\gamma_0]$, there exists $\delta>0$ and $\phi_0>0$ such that for all $\phi<\phi_0$ there holds
\begin{align}
 \gamma\leq |u-\Psi_m|_{\tphil}\leq \gamma_0\qquad\Rightarrow\qquad \ephi(u)\geq \ez(\Psi_m)+\delta.\label{eqr}
\end{align}
\end{lemma}

\begin{proof}[Proof of lemma \ref{l:alt}]
We prove lemma \ref{l:alt} via an adaptation to $\phi>0$, $\tphil$ of the proof of lemma \ref{la}.
Fix any $\gamma \in (0,\gamma_0]$. As usual, by translating on the torus, we will assume that
\begin{align}
|u-\Psi_m|_{\tphil}^2:=  \inf_{x_0\in\tphil} ||u-\Psi_m(\cdot-x_0)||_{L^2(\tphil)}^2=||u-\Psi_m(\cdot-0)||_{L^2(\tphil)}^2.\label{gFL}
\end{align}
Note that we may assume without loss of generality that
\[
\ephi(u)\leq E_M \quad\text{(cf.\ \eqref{emm})},
\]
so that by lemma \ref{l:bulk}, the definition of $\nu(u)$, and $||u-\Psi_m||_{L^2(\tphil)}\leq \gamma_0$ we obtain
\begin{align}
  \nu_m-\frac{\gamma_0^2}{4}+O(\phi^{1/3})\leq \nu(u)\leq \nu_m+\frac{\gamma_0^2}{4}+O(\phi^{1/3}).\label{basin2}
\end{align}

Analogously to in the proof of lemma \ref{la}, we observe that $|\nu(u)-\nu_m|>\gamma^2/8$ (and the lower bound from proposition \ref{lemma.lower.bound}) would imply for $\phi$ sufficiently small that
\begin{align*}
  \ephi(u)\geq f_\xi (\nu(u))+O(\phi^{1/3})>f_\xi(\nu_m)+\delta+O(\phi^{1/3})\geq f_\xi(\nu_m)+\frac{\delta}{2}=\ez(\Psi_m)+\frac{\delta}{2}.
\end{align*}
Hence we may assume that
\begin{align}
  |\nu(u)-\nu_m|\leq \frac{\gamma^2}{8}.\label{voldiff}
\end{align}
As in the proof of lemma \ref{la}, this will lead to a positive lower bound for the Fraenkel asymmetry of suitable sets. It remains to establish this fact.

To begin, we remark that, since \eqref{basin2} implies $f_\xi(\nu(u))\geq f_\xi(\nu_m)$, it suffices in light of proposition \ref{lemma.lower.bound}  to establish a positive ($\phi$-independent) lower bound on
\begin{equation}
I(u)=\int_{-1+2\kappa}^{1-2\kappa}\sqrt{2\tilde{G}(t)}\,\Big({\rm Per}_{\tphil}(\{u>t\})-{\rm P}_E(\{u>t\})\Big)\,dt.\label{IF}
\end{equation}

  Given the scaling regime \eqref{crit}, the quantitative isoperimetric inequality \eqref{SharpIsoper2}, and the lower bound on ${\rm P}_E(\{u>s\})$ implied by \eqref{basin2}, it suffices to establish a positive ($\phi$-independent) lower bound on $\lambda(\{u>s\})$ for all $s\in [-1+2\phi^{1/3},1-2\phi^{1/3}]$. According to lemma \ref{l:frel}, it in fact suffices to establish a  positive ($\phi$-independent) lower bound on $\lambda(\{u>0\})$. Since corollary \ref{cor.1} implies that
\begin{equation}\label{zero.vol}
\nu_u:=  |\{u>0\}|=\nu(u)+O(\phi^{1/3}),
\end{equation}
it suffices to produce a lower bound on
\begin{align*}
   \min_{x\in\tphil} |B_{\nu_u}(x)\triangle \{u>0\}|.
\end{align*}

To this end, we observe that \eqref{gFL} and lemma \ref{l:bulk} imply
\begin{align*}
  |B_{\nu_m}(x)\triangle \{u>0\}|\geq\frac{\gamma^2}{4}+o(1).
\end{align*}
Using the triangle inequality \eqref{sdtri}, we deduce from this fact together with \eqref{voldiff} and \eqref{zero.vol}  that
\begin{align*}
  |B_{\nu_u}(x)\triangle \{u>0\}|\geq\frac{\gamma^2}{8}+o(1).
\end{align*}
Since $x\in\tphil$  is arbitrary, we have established
\begin{align}
 \min_{x\in\tphil} |B_{\nu_u}(x)\triangle \{u>0\}|\geq\frac{\gamma^2}{8}+o(1).\label{FIF}
\end{align}
\end{proof}


\subsection{Analogous estimates for the saddle point}\label{ss:sadtoo}
In this subsection we collect our information about the saddle points of $\ephi$ that were defined in the introduction. We begin by establishing proposition \ref{t:exus}, which uses the upper and lower bounds on the energy barrier to invoke a mountain pass theorem for the existence of $\us$ and $\ust$.

\begin{proof}[Proof of proposition \ref{t:exus}]
We recall from \cite[lemma A.2]{GW} that $\ephi$ is in $C^1(X_\phi;\R)$ and satisfies the Palais-Smale condition.
Using these conditions and a standard mountain pass argument (see for instance \cite[Corollary 1.8]{GW}), we can establish the existence of the saddle point $\us$ once we show that

 (i) There exists a function $u\in X_\phi$ with $\ephi(u)<\ephi(\um)$,

(ii) There exists $c>0$ such that, for any such $u$ and any continuous path $\psi$ with $\psi(0)=u$, $\psi(1)=\um$, there holds
\begin{align*}
  \sup_{t\in[0,1]}\ephi(\psi(t))\geq \ephi(\um)+c.
\end{align*}

The first condition is established via the upper bound construction from  proposition \ref{prop.upper.bound}, using the fact that $c_m$ is a local but not global minimum value of $f_\xi$ for $\xi\in (\tilde{\xi}_d,\xi_d)$. To confirm condition (ii), we define the unique point $\nu_-\neq \nu_m$ via
\begin{align*}
  f_\xi(\nu_-)=f_\xi(\nu_m),
\end{align*}
and observe that according to the lower bound \eqref{lower.bound}, there holds
\begin{align}
  \nu(u)\leq \nu_-+O(\phi^{1/3}).\label{newp}
\end{align}
On the other hand recall from \eqref{num} that $\nu(\um)\geq\nu_m+C\,\phi^{1/6}$. Hence a second application of the lower bound (and continuity of $\nu(\cdot)$) implies
\begin{align}
  \sup_{t\in[0,1]}\ephi(\psi(t))&\geq c_s+O(\phi^{1/3})\overset{\eqref{en.cl}}\geq\ephi(\um)+(c_s-c_m)+O(\phi^{1/3})\label{cmep2}\\
  &\geq\ephi(\um)+\frac{c_s-c_m}{2}\label{cmep}
\end{align}
for $\phi$  small.

Analogously, the existence of $\ust$ follows from a standard mountain pass argument once we show that

 (i) There exists a function $u\in X_\phi$ with $\ephi(u)<\ephi(\um)$,

(ii') There exists $c>0$ such that, for any such $u$, any $v\in \neps$, and any continuous path $\psi$ with $\psi(0)=u$, $\psi(1)=v$, there holds
\begin{align*}
  \sup_{t\in[0,1]}\ephi(\psi(t))\geq \ephi(v)+c.
\end{align*}

Hence it suffices to check (ii'). To this end, we observe that by the definition of $\neps$, \eqref{ump} and lemma \ref{l:bulk}, there exists $\tilde\eps>0$ such that for any $\eps \in (0,\tilde \eps)$ there holds $\nu(v)\geq (\nu_m+\nu_s)/2$ for small enough $\phi$. On the other hand as above, we observe that \eqref{newp} holds for any $u$ satisfying (i). It follows as above for any $\eps\leq (c_s-c_m)/4$ that
\begin{eqnarray}
    \sup_{t\in[0,1]}\ephi(\psi(t))&\overset{\eqref{cmep2}}\geq & c_s+O(\phi^{1/3})\label{cmep3}\\
    &\overset{\eqref{cmep}}\geq & \ephi(\um)+\frac{c_s-c_m}{2}
    \geq \ephi(v)-\eps+\frac{c_s-c_m}{2}\notag\\
    &\geq& \ephi(v)+\frac{c_s-c_m}{4}.\notag
\end{eqnarray}

It remains to bound the energy barriers. The lower bound on $\deone$ is implied by the first inequality in \eqref{cmep2}. Similarly, the lower bound  $\de\geq c_s+O(\phi^{1/3})$ is implied by \eqref{cmep3}. On the other hand, from  theorem \ref{t:exist} and the upper bound constructions in proposition \ref{prop.upper.bound}, we deduce the  upper bound  $\de\leq c_s+O(\phi^{1/3})$.
\end{proof}

As explained in the introduction, in order to find a function that satisfies the folkloric properties of the critical nucleus, we turn to the volume-constrained minimizer $\uws$. We begin by stating the analogy of lemma \ref{lc} for the sharp-interface saddle point $\Psi_s$. We apply the estimate in the proof of theorem \ref{t:saddle} and remark \ref{rem:appxmtn}. See also remark \ref{rem:pass} below.
\begin{lemma}[Finite $\phi$ estimate]\label{lc2}
Let $\gamma_0>0$ be as in \eqref{gm0}. For every $\gamma\in (0,\gamma_0]$, $\delta>0$, there exist $\phi_0>0$, $\beta_0>0$ such that for $\phi\leq \phi_0$, one has for all $u\in X_\phi$ that
\begin{align}
\left| \nu(u)-\nu_s\right|\leq \beta_0\; &\text{and}\; |u-\Psi_{s}|_{\tphil}\geq\gamma
 \quad\Rightarrow\quad \ephi(u)\geq \inf_{\substack{|u-\Psi_{s}|_{\R^d}\geq\gamma\\ \nu_0(u)=\nu_s}}\ez(u)-\delta.\label{C2}
\end{align}
\end{lemma}
The proof of the lemma is given in subsection \ref{ss:pflemmas}.
\begin{remark}\label{rem:pass}
Notice that  lemma \ref{lc2} and \eqref{A2b} imply that any approximately optimal path for $\de$ stays within a $\gamma$ neighborhood of $\Psi_s$ for all volumes close to $\nu_s$. The mountain pass around $\Psi_s$ is ``narrow'' in this sense; see also remark \ref{rem:all}.
\end{remark}

We now turn to the proof of theorem \ref{t:saddle}, which establishes the existence and properties of the constrained minimizer $\uws$.
\begin{proof}[Proof of theorem \ref{t:saddle}]
We first note that the lower and upper bounds in propositions \ref{lemma.lower.bound} and \ref{prop.upper.bound}, and the fact that $\nu_s$ is the unique maximum of $f_\xi$ over $[0,\nu_m]$, imply that
\begin{align}
\omega_*=\nu_s+o(1).\label{ooo}
\end{align}
Next, for given $\gamma\in (0,\gamma_0]$ we use lemma \ref{la2} to identify $\delta>0$ such that
\begin{align}
      |u-\Psi_{s}|_{\R^d}\geq\gamma,\, \nu_0(u)=\nu_s\qquad\Rightarrow\qquad \ez(u)\geq c_s+2\delta.\label{ipp}
    \end{align}
With these $\gamma$, $\delta$, we extract from \eqref{C2} that for $\nu(u)=\omega_*$ (which by \eqref{ooo} is close to $\nu_s$) and $|u-\Psi_{s}|_{\tphil}\geq\gamma$, there holds
\begin{align*}
\ephi(u)\geq \inf_{\substack{|u-\Psi_{s}|_{\R^d}\geq\gamma\\ \nu_0(u)=\nu_s}}\ez(u)-\delta\overset{\eqref{ipp}}\geq c_s+\delta.
\end{align*}
On the other hand the constructions from proposition \ref{prop.upper.bound} and \eqref{ooo} yield a function $\hat{u}_\omega$ such that $\nu(\hat{u}_\omega)=\omega_*$ and
\begin{align*}
  \ephi(\hat{u}_\omega)\leq c_s+\frac{\delta}{2}.
\end{align*}
Minimality of $\uws$ yields \eqref{ump2}.

It remains to deduce  \eqref{enps} and \eqref{wmm2}. Combining the estimate of $\Ew$ and the definition of $\uws$ (see \eqref{ebjan2} and \eqref{omgst}) leads to \eqref{enps}.
To obtain \eqref{wmm2}, note that the definition of $\uws$, the bound \eqref{enps}, and proposition \ref{prop.upper.bound} imply
\begin{align}
f_\xi(\nu_s)+O(\phi^{1/3}) \leq \ephi(\uws) \leq f_\xi(\omega_*)+O(\phi |\ln \phi|),\label{star}
\end{align}
so that
\begin{equation}\label{t:saddle.1}
f_\xi(\omega_*)\geq f_\xi(\nu_s)+O(\phi^{1/3}).
\end{equation}
Given \eqref{ooo} we may apply  the Taylor formula  and $f_\xi'(\nu_s)=0$, $f_\xi''(\nu_s)<0$ to deduce from \eqref{t:saddle.1} that
\[
|\omega_*-\nu_s|\leq C \phi^{1/6},
\]
for $C=C(\xi,d)$, which is \eqref{wmm2}.

\end{proof}


\subsection{Deviation from sphericity}\label{ss:dev}

In this subsection we look more closely at the ``droplet-like-shape'' of the local minimizer $\um$ and the volume-constrained minimizers $\uo$ using
 the quantitative isoperimetric inequality.
The main ingredient for establishing the droplet-like shape of $\um$ and the volume-constrained minimizers is the following observation.
\begin{lemma}\label{frank1}
Consider $\xi\in (0,\xi_d]$ and the critical scaling \eqref{crit}. For any $\omega>0$ and for $\phi>0$ sufficiently small, the following holds. If $u\in X_\phi$ satisfies
\begin{align}
 \omega\leq \nu(u) \leq \frac{\xi^{d+1}}{2}\qquad\text{and}\qquad \ephi(u)\leq f_\xi(\nu(u))+O(\phi^{1/3}),\label{ein}
\end{align}
then for every $s\in[-1+2\phi^{1/3},1-2\phi^{1/3}]$, the Fraenkel asymmetry of the superlevel set $\{u>s\}$ satisfies
\begin{align}
\lambda(\{u>s\})\lesssim \frac{\phi^{\alpha}}{\nu(u)}\qquad\text{with}\quad \alpha=\min\{1/6,1/(2d)\}.\label{lamsize}
\end{align}
\end{lemma}
\begin{proof}
The energy bound in \eqref{ein} and the lower bound \eqref{lower.bound} imply
\begin{align}
I(u)\lesssim\phi^{1/3},\label{ibd2}
\end{align}
where $I(\cdot)$ is the asymmetry cost defined in \eqref{DefnIsopTerm}.
The quantitative isoperimetric inequality \eqref{SharpIsoper2} applied to $\{u>s\}$ gives
\begin{eqnarray}
\lefteqn{\per_{\tphil}(\{u>s\})-{\rm P}_E(\{u>s\})}\notag \\
&\geq& C(d){\rm P}_E(\{u>s\})\lambda(\{u>s\})^2-\frac{4d|\{u>s\}|}{\phi L} \notag\\
&\overset{\eqref{crit},\eqref{saddle.superset}}=& C(d){\rm P}_E(\{u>s\})\lambda(\{u>s\})^2+O(\phi^{1/d})\notag.
\end{eqnarray}
Substituting into  \eqref{DefnIsopTerm} and applying the scaling bound \eqref{ibd2}, we obtain
\begin{align}\label{isop3}
\int_{-1+2\phi^{1/3}}^{1-2\phi^{1/3}}\sqrt{2\tilde{G}(s)}\,\,{\rm P}_E(\{u>s\})\,\lambda(\{u>s\})^2\,ds\lesssim \phi^{1/3}+\phi^{1/d}.
\end{align}
Applying lemma \ref{l:frel}  yields
\begin{align*}
\left(\sup_{s\in[-1+2\phi^{1/3},1-2\phi^{1/3}]}\lambda(\{u>s\})+\frac{O(\phi^{1/3})}{\nu(u)}\right)^2  &\int_{-1+2\phi^{1/3}}^{1-2\phi^{1/3}}\sqrt{2\tilde{G}(s)}\,\,{\rm P}_E(\{u>s\})\,ds\\
&\lesssim \phi^{1/3}+\phi^{1/d},
\end{align*}
so that
\begin{align*}
\sup_{s\in[-1+2\phi^{1/3},1-2\phi^{1/3}]}\lambda(\{u>s\})+\frac{O(\phi^{1/3})}{\nu(u)}
\lesssim \frac{\phi^{1/6}+\phi^{1/(2d)}}{\nu(u)^{(d-1)/(2d)}},
\end{align*}
where we have used $|\{u>s\}|\geq \nu(u)$ for all $s\in[-1+2\phi^{1/3},1-2\phi^{1/3}]$. The estimate \eqref{lamsize} follows.
\end{proof}

\begin{proof}[Proof of theorem \ref{t:spheres}]
We begin by establishing that $\um$ and $\uws$ satisfy \eqref{visam} and \eqref{visa}.  Notice that \eqref{visa} follows from \eqref{visam}, lemma \ref{l:l2fr}, \eqref{num}, and \eqref{wmm2}. Hence it suffices to establish \eqref{visam}. This will follow from an application of the previous lemma.
Indeed, for the minimizer, we have
\begin{eqnarray*}
  \ephi(\um)&\overset{\eqref{en.cl}}\leq & c_m+O(\phi^{1/3})\\
  &\leq& f_\xi(\nu(\um))+O(\phi^{1/3}),
\end{eqnarray*}
since \eqref{num} implies that $\nu(\um)$ is within a neighborhood for which $c_m$ is the minimum value of $f_\xi$. Hence \eqref{ein} is verified for $\um$. On the other hand, \eqref{star} and \eqref{ooo} verify \eqref{ein}  for $\uws$.


Lastly, we verify \eqref{uwtoo} for any volume-constrained minimizer with the help of the upper bound \eqref{upper.bound} and lemmas \ref{l:l2fr} and \ref{frank1}.
\end{proof}

Finally, we point out that $\um$ can be characterized as a constrained minimizer of appropriate volume.
\begin{lemma}\label{l:struc}
The local minimizer $\um$ is a constrained minimizer of the energy subject to a volume constraint in the sense that it minimizes $\ephi$ over all functions $u$ such that
\begin{align*}
  \nu(u)=\nu(\um).
\end{align*}
\end{lemma}

\begin{proof}
We define the volume
\begin{align}
  \omega_\phi^m:=\nu(\um)\label{wmo}
\end{align}
and let $\umo$ be an associated constrained minimizer, i.e., a function that minimizes $\ephi$ subject to $\nu(u)=\omega_\phi^m$.
It suffices to show that for some $\gamma>0$ and for all $\phi$ sufficiently small, the constrained minimizer $\umo$ belongs to a $\gamma$ neighborhood of $\Psi_m$:
\begin{align}
  |\umo -\Psi_m|_{\tphil}< \gamma.\label{est.1}
\end{align}
Indeed, it then follows from the characterization of $\um$ as the minimizer over the $\gamma$-neighborhood that
\begin{align*}
  \ephi(\um)\leq \ephi(\umo)\leq \ephi(u)\quad\text{for all }u \text{ with }\nu(u)=\omega_\phi^m.
\end{align*}

We will now show that \eqref{est.1} holds. Let $\nu_{m,\phi}:=|\{\umo\geq 1-\phi^{1/3}\}|$ and let  $\Psi_{m,\phi}$ denote the sharp-interface profile with this droplet volume. By corollary \ref{cor.1} it follows that
\begin{equation}\label{est.2}
\nu_{m,\phi}=\omega_\phi^m+O(\phi^{1/3}).
\end{equation}

In light of \eqref{num} and \eqref{est.2}, there holds
\begin{align}
  ||\Psi_m-\Psi_{m,\phi}||_{L^2(\tphil)}< \frac{\gamma}{2},\label{psipsi}
\end{align}
so that, by the triangle inequality, it suffices to show that
\begin{align}
  |\umo -\Psi_{m,\phi}|_{\tphil}< \frac{\gamma}{2},\label{est.1b}
\end{align}
which in turn follows from \eqref{uwtoo} and \eqref{est.2}.
\end{proof}
\subsection{Proofs of lemmas}\label{ss:pflemmas}
We now present the proofs of lemmas \ref{lc} and \ref{lc2}.
Lemma \ref{lc} establishes a link between the finite $\phi$ energy of functions that are $\gamma$ away from $\Psi_m$ and the limit energy of the same set of functions.
In contrast to the lower bound in the $\Gamma$-convergence proof in \cite{GW}, (i) we want an estimate that is uniform over $|u-\Psi_m|_{\R^d}=\gamma$ (rather than just an estimate for any given point $u_0$ in this set), and (ii) we \emph{do not assume} $L^2$ convergence to some function $u_0$ in the sense that $$\int_{\tphil}(u_\phi- u_0)^2\,dx\to 0.$$

Roughly speaking, the issue that arises is that, although there is a function $u_0=\pm 1$ such that $u_\phi\to u_0$ in $L^2(K)$ for any compact set $K$ (cf. lemma \ref{l:mm} below), it may be that the $L^2$ distance to $\Psi_m$ drops in the limit, i.e., that $||u_0-\Psi_m||_{L^2(R^d)}<\gamma$ even though $||u_\phi-\Psi_m||_{L^2(\tphil)}=\gamma$ for every $\phi>0$. The volume costs are straightforward, but we need a good bound on the perimeter cost.
What we establish in the proof below is that there is ``no free lunch'' in the sense that, on the one hand, $\ephi(u_\phi)$ includes the full perimeter cost of $u_0$ on $K$ and, on the other hand, if some of $\{u_\phi\sim 1\}$---roughly speaking the volume $\beta/4$ in the proof below---has drifted off to infinity in the limit, then $\ephi(u_\phi)$ also includes the associated perimeter cost of this mass, at least in the sense of
$$\left(\frac{\beta}{4}\right)^{(d-1)/d}.$$
This is enough to conclude.

We begin by establishing $L^2$ convergence on compact sets. The argument is standard but we include it for completeness.
\begin{lemma}\label{l:mm}
Fix $\xi\in (\tilde{\xi}_d,\xi_d]$ and the critical scaling \eqref{crit}. Let $u_\phi\in X_\phi$ be a sequence of functions such that
\begin{align}
  \ephi(u_\phi) \lesssim 1.\label{hypo}
\end{align}
Then there exist $u_0=\pm 1$ a.e.\ on $\R^d$ and a subsequence of $\{u_\phi\}_{\phi >0}$ such that for any compact set $K\subset\R^d$
\begin{align*}
  u_\phi \mathbf{1}_K\to u_0 \mathbf{1}_K\quad\text{in }L^2(K).
\end{align*}
\end{lemma}
\begin{proof}
Consider a compact set $K\subset \mathbb{R}^d$ and note that, for $\phi$ small enough so that $K\subset \tphil$, we have
\begin{align}
\ephi (u_\phi)&=\int_{\tphil}\dfrac{\phi}{2}|\nabla u_\phi|^2+\dfrac{1}{\phi}G(u_\phi)\,dx-\frac{G(-1+\phi)}{\phi}|\tphil| \notag \\
&\geq \int_K\dfrac{\phi}{2}|\nabla u_\phi|^2+\dfrac{1}{\phi}G(u_\phi)\,dx-\frac{G(-1+\phi)}{\phi}|\tphil| \notag \\
&\geq \int_K \sqrt{2G(u_\phi)}|\nabla u_\phi|-\frac{G(-1+\phi)}{\phi}|\tphil| \notag \\
&=\int_K |\nabla F(u_\phi)|\,dx-\frac{G(-1+\phi)}{\phi}|\tphil|,\label{var1}
\end{align}
where $F(t):=\int_{-1}^t\sqrt{2G(s)}\,ds$. From
\eqref{var1} together with \eqref{hypo} and \eqref{f2},
we deduce
\begin{equation}\label{variation2}
\sup _{\phi>0}\int_K |\nabla F(u_\phi)|\,dx <\infty.
\end{equation}
On the other hand, since $G(s)\sim s^4$ for large values of $s$, we have
\begin{equation}\label{variation3}
\sup_{\phi >0}\int_K |F(u_\phi)|\,dx\leq C_K \left(1+ \sup _{\phi >0}\int_K G(u_\phi)\,dx\right)<\infty.
\end{equation}
By \eqref{variation2} and \eqref{variation3} it follows that $\{F(u_\phi)\}_{\phi >0}$ is bounded in $BV(K)$. Consequently,
there exists $w_0\in L^1(K)$ and a subsequence of $\{u_\phi\}_{\phi >0}$ such that $F(u_\phi)\to w_0$ in $L^1(K)$. By the uniform continuity of $F^{-1}$ on $\mathbb{R}$ it follows that $u_\phi$ converges in measure on $K$ to $u_0:=F^{-1}(w_0)$. Moreover, by the second inequality in \eqref{variation3} we have $\sup_{\phi >0}\|u_\phi\|_{L^4(K)}<\infty$; hence the family $\{u_\phi^2\}_{\phi >0}$ is uniformly integrable on $K$. We thus deduce that $\|u_\phi-u_0\|_{L^2(K)}\to 0$.

Using an expanding sequence of compact sets $\{K_n\}_{n\geq 1}$ with $\cup_{n\geq1} K_n=\R^d$ and a diagonal argument, one can define $u_0$ on all of $\R^d$ so that
\begin{align*}
  u_\phi \mathbf{1}_K\to u_0 \mathbf{1}_K\quad\text{in }L^2(K),
\end{align*}
for any compact set $K\subset\R^d$.

Finally, we check that $u_0=\pm 1$ a.e. on $\R^d$. Indeed, for any compact set $K\subset \R^d$ we have
\begin{align}
\ephi(u_\phi)\geq  \frac{1}{\phi}\int_K G(u_\phi)\,dx-\frac{G(-1+\phi)}{\phi}|\tphil|,
\end{align}
so that by Fatou's lemma, \eqref{f2}, the energy bound \eqref{hypo}, and the fact that (up to a subsequence) $u_\phi$ converges a.e. to $u_0$ in $K$, we have
\begin{equation}
\int_KG(u_0)\,dx \leq \liminf_{\phi \to 0}\int_KG(u_\phi)\,dx
\leq \liminf_{\phi \to 0}\left(\phi\, \ephi (u_\phi)\right)=0.
\end{equation}
\end{proof}

\begin{proof}[Proof of lemma \ref{lc}: Finite $\phi$ estimate for the local minimum]
The proof is by contradiction. Hence, assume for a contradiction that there exists $\tilde\gamma\in(0,\gamma_0]$ and $\delta>0$ such that there exists a sequence $\phi\downarrow 0$ and a corresponding sequence of functions $u_\phi\in X_\phi$ such that
\begin{align}
&|u_\phi-\Psi_m|_{\tphil}= \gamma_\phi\in[\tilde\gamma,\gamma_0],\notag\\
\qquad\text{and}\qquad &\ephi(u_\phi)\leq \inf_{\gamma\leq|u-\Psi_m|_{\R^d} \leq\gamma_0}\ez(u)-2\delta.\notag
\end{align}
Without loss of generality, we may assume that
\begin{align}
\gamma_\phi\to \gamma\in[\tilde\gamma,\gamma_0].\label{gp2}
\end{align}
By translating on the torus, we will always assume that
\begin{align*}
  \inf_{x_0\in\tphil} ||u_\phi-\Psi_m(\cdot-x_0)||_{L^2(\tphil)}=||u_\phi-\Psi_m(\cdot-0)||_{L^2(\tphil)}.
\end{align*}
In addition, by density of smooth functions in $X_\phi$, there exists a sequence of $C^\infty$ functions $\tilde{u}_\phi\in X_\phi$ such that
\begin{align}
&|\tilde{u}_\phi-\Psi_m|_{\tphil}= \gamma+o(1)_{\phi\downarrow 0},\label{g}\\
\qquad\text{and}\qquad &\ephi(\tilde{u}_\phi)\leq \inf_{\gamma\leq|u-\Psi_m|_{\R^d} \leq\gamma_0}\ez(u)-\delta.\label{eless}
\end{align}
In the remainder of the proof we will work with this smooth sequence (and for notational simplicity, we will write $u_\phi$ instead of $\tilde{u}_\phi$).

\medskip

\underline{Step 1: Preliminary bounds.}
By comparison with radial constructions, it is easy to check that
\begin{align}
  \inf_{\gamma\leq|u-\Psi_m|_{\R^d} \leq\gamma_0}\ez(u)\leq c_s \quad\text{for }\xi\in (\tilde{\xi_d},\xi_d],\label{esmo}
\end{align}
so that in particular our sequence satisfies
\begin{align}
  \ephi(u_\phi)\leq E_M,\label{edo}
\end{align}
where we recall the definition of $E_M$ in \eqref{emm}. Proposition \ref{lemma.lower.bound}, lemma \ref{nobigsets} and corollary \ref{cor.1} imply
\begin{equation}\label{cone}
|\cald(u_\phi)|\lesssim 1.
\end{equation}
By applying lemma \ref{l:mm} we deduce that there exists $u_0=\pm 1$ a.e.\ such that (up to subsequences) $u_\phi\to u_0$ in $L^2(K)$ for any compact set $K$.
Arguing as in \cite[theorem 1.9]{GW}, we obtain for $u_0$ the bounds
\begin{align}
  |\{u_0=1\}|\lesssim 1,\qquad \per( \{u_0=1\})\lesssim 1.\label{cimit}
\end{align}
Moreover, since a set with bounded volume and perimeter can be well approximated by a smooth, open, bounded set (cf. \cite[Remark 13.12]{Ma}), an approximation argument similar to the one used in the proof of \cite[theorem 1.9]{GW} allows us to assume that $\{u_0=1\}$ is bounded.

\underline{Step 2: Estimates on a compact set.}
Let
\begin{align}
K=[-k,k]^d\label{K}
\end{align}
be a compact $d$-dimensional cube that compactly contains $\{u_0=1\}$ and $B_{\nu_m}(0)$. (Note for future reference that $\Psi_m=-1$ on $K^c$.) Because of
\begin{align*}
  u_\phi\to u_0 \quad\text{in }L^2(K),
\end{align*}
we have, according to \eqref{g}, that
\begin{align}
  ||u_\phi-\Psi_m||_{L^2(K)}^2\to ||u_0-\Psi_m||_{L^2(K)}^2=\gamma^2-\beta\quad\text{for some }\beta\in [0,\gamma^2].\label{l2s}
\end{align}
In addition, obtaining from lemma \ref{l:bulk} that
\begin{align}
|\calc(u_\phi)|,\,|\cale(u_\phi)|\to 0,\label{bdz}
\end{align}
we observe that
\begin{align}
\intfil \chi_3(u_\phi)\mathbf{1}_K\,dx\to |\{u_0=1\}|.\label{nuk}
\end{align}

\underline{Step 3: The deficit.}
In view of \eqref{g} and \eqref{l2s}, there holds
\begin{align}
  ||u_\phi-\Psi_m||_{L^2(\tphil\setminus K)}^2=||u_\phi-(-1)||_{L^2(\tphil\setminus K)}^2=\beta+o(1)_{\phi\downarrow 0}.\label{betabig}
\end{align}
On the other hand, in view of lemma \ref{l:bulk}, we have
\begin{align*}
  \intfil (u_\phi-(-1))^2\mathbf{1}_{\tphil\setminus K}\mathbf{1}_{\cala\cup\calb\cup\calc\cup\cale}\,dx=o(1)_{\phi\downarrow 0},
\end{align*}
so that \eqref{betabig} improves to
\begin{align}
  \intfil (u_\phi-(-1))^2\mathbf{1}_{\tphil\setminus K}\mathbf{1}_{\cald}\,dx=\beta+o(1)_{\phi\downarrow 0}.\label{bet2}
\end{align}
From \eqref{bet2} we read off
\begin{align}
 \left |\cald(u_\phi)\cap \left(\tphil\setminus K\right)\right|\to \frac{\beta}{4},\label{deficit}
\end{align}
which because of \eqref{bdz} we can also express as
\begin{align}
  \intfil \chi_3(u_\phi)\mathbf{1}_{\tphil\setminus K}\,dx=\frac{\beta}{4}+o(1)_{\phi\downarrow 0}.\label{deficit2}
\end{align}

\underline{Step 4: Total energy.}
We now calculate the cost associated to the sequence $\{u_\phi\}$. Combining \eqref{nuk} and \eqref{deficit2} gives
\begin{align}
  \nu(u_\phi)= |\{u_0=1\}|+\frac{\beta}{4}+o(1)_{\phi\downarrow 0}.\label{vol}
\end{align}
Estimating the energy as in the proof of \cite[proposition 2.4]{GW}, we obtain
\begin{align}
  \intfil e_\phi(u_\phi)\chi_3(u_\phi)\,dx&\geq -(4+o(1)_{\phi\downarrow 0})\nu(u_\phi),\notag\\
  \intfil e_\phi(u_\phi)\chi_1(u_\phi)\,dx&\geq (4+o(1)_{\phi\downarrow 0})\frac{\nu(u_\phi)^2}{\xi^{d+1}}.\label{fir}
\end{align}

The perimeter cost is more involved. For the contribution corresponding to $\chi_2$, we split the integral over the compact set $K$ (cf.\ \eqref{K}) and $\tphil\setminus K$.
The convexity of $G$ near $-1$, expressed in the form
\begin{align}
0=G(-1)&\geq G(-1+\phi)+G'(-1+\phi)(-1-(-1+\phi))\notag\\
&=G(-1+\phi)-\phi G'(-1+\phi),\label{mconv1}
\end{align}
implies that
\begin{align}
e_\phi(u)\geq \frac{1}{\phi} \Big(G(u)-G'(-1+\phi)(u+1) \Big)\geq 0\quad\text{on }\calc.\label{eb}
\end{align}
Hence we can replace $\chi_2$ by $\chi_2^\eta$ (where the support is on $(-1+\eta,1-\eta)$) for fixed $\eta>0$. On $K$, we use the $L^2(K)$ convergence and argue  as in the proof of \cite[theorem 1.9]{GW} to deduce
\begin{align}
  \intfil e_\phi(u_\phi)\chi_2(u_\phi)\mathbf{1}_K\,dx& \geq \intfil e_\phi(u_\phi)\chi_2^\eta(u_\phi)\mathbf{1}_K\,dx\notag\\
    &\geq \Big(c_0+o(1)_{\eta\downarrow 0}\Big)\big({\rm Per_{K^\circ}}(\{u_0=1\})+o(1)_{\phi\downarrow 0}\big),\notag\\
  &\geq \Big(c_0+o(1)_{\eta\downarrow 0}\Big)\big({\rm Per}(\{u_0=1\})+o(1)_{\phi\downarrow 0}\big),\label{pert1}
\end{align}
where we have recalled that $\{u_0=1\}$ is compactly contained in $K$.

On $\tphil\setminus K$, on the other hand, we use the coarea formula as in the proof of \cite[proposition 2.1]{GW} to argue that
\begin{align}
\lefteqn{\intfil e_\phi(u_\phi)\chi_2(u_\phi)\mathbf{1}_{\tphil\setminus K}\,dx}\notag\\
&\geq \int_{-1+2\eta}^{1-2\eta}\sqrt{2\tilde{G}(s)}\mathcal{H}^{d-1}(\{x\in\tphil\setminus K\colon u_\phi(x)=s\})
\,ds\notag \\
&\geq \underset{s\in(-1+2\eta,1-2\eta)}\einf\mathcal{H}^{d-1}(\{x\in\tphil\setminus K\colon u_\phi(x)=s\})   \int_{-1+2\eta}^{1-2\eta}\sqrt{2\tilde{G}(s)}
\,ds\notag\\
&\geq  \left(\mathcal{H}^{d-1}\left(\{x\in\tphil\setminus K\colon u_\phi(x)=s_\phi\}\right)-\phi\right) \int_{-1+2\eta}^{1-2\eta}\sqrt{2\tilde{G}(s)}
\,ds
\label{firp}
\end{align}
where we have chosen $s_\phi \in [-1+2\eta,1-2\eta]$ to approximate the essential infimum.
Notice that we may in addition without loss of generality assume that
\begin{align*}
 \mathcal{H}^{d-1}(\{x\in\tphil\setminus K\colon u_\phi(x)=s_\phi\})<\infty,
\end{align*}
since otherwise \eqref{firp}, \eqref{fir}, and \eqref{edo} lead to a contradiction.
We would like to pass from the level surface on the right-hand side of \eqref{firp} to a closed level surface on the torus. To do so, we allow an extra degree of freedom. Namely, we introduce the hypercubes and corresponding surfaces
\begin{align}
 K_\ell=[-\ell,\ell]^d,\qquad S_\ell:=\partial K_\ell,\qquad \ell\in [k,2k].\label{cubes}
\end{align}
Trivially, we estimate
\begin{align}
\lefteqn{ \mathcal{H}^{d-1}(\{x\in\tphil\setminus K\colon u_\phi(x)=s_\phi\})}\notag\\
&\geq \mathcal{H}^{d-1}(\{x\in\tphil\setminus K_\ell\colon u_\phi(x)=s_\phi\}),\quad \ell\in [k,2k].\label{triv1}
\end{align}
On the other hand, for $\ell\in[k,2k]$
we can relate $\{u_\phi>s_\phi\}\cap \left(\tphil\setminus K_\ell\right)$ and the surface measure on the right-hand side of \eqref{triv1} via
\begin{align}
 \lefteqn{\mathcal{H}^{d-1}\left(\partial\left(\{u_\phi>s_\phi\}\cap \big(\tphil\setminus K_\ell\big)\right)\right)}\notag\\
 &=\mathcal{H}^{d-1}(\{x\in\tphil\setminus K_\ell\colon u_\phi(x)=s_\phi\})+\mathcal{H}^{d-1}\left(\{u_\phi>s_\phi\}\cap S_\ell\right).\label{volin}
\end{align}
It remains to argue that the second term on the right-hand side is small. To do so, we will exploit the degree of freedom allowed by $\ell$ in the form of the following lemma.

\begin{lemma}\label{l:sa}
Consider the cubes and surfaces defined in \eqref{cubes}.
Consider a set $E\subset \R^d$ such that $|E\cap (K_{2k}\setminus K)|\leq \eps/2.$
Then there exists $\ell_*\in [k,2k]$ such that
\begin{align}
0\leq \mathcal{H}^{d-1}(E\cap S_{\ell_*})\leq \frac{\eps}{k}.
\end{align}
\end{lemma}
The proof of lemma \ref{l:sa} follows from writing the volume as an integral of surface area and considering the infimum.
We apply lemma \ref{l:sa} to the set $E=\{u_\phi>s_\phi\}$, noting that the strong $L^2$ convergence of $u_\phi$ to $-1$ on $K_{2k}\setminus K$ yields
\begin{align*}
 |\{u_\phi\geq s_\phi\}\cap (K_{2k}\setminus K)|=o(1)_{\phi\downarrow 0}.
\end{align*}

Combining \eqref{triv1}, \eqref{volin}, and lemma \eqref{l:sa} with $E=\{u_\phi> s_\phi\}$, we obtain
\begin{align}
\lefteqn{ \mathcal{H}^{d-1}(\{x\in\tphil\setminus K\colon u_\phi(x)=s_\phi\})}\notag\\
&\geq \mathcal{H}^{d-1}\left(\partial\left(\{u_\phi>s_\phi\}\cap \big(\tphil\setminus K_{\ell_*}\big)\right)\right)-\mathcal{H}^{d-1}\left(\{u_\phi> s_\phi\}\cap S_{\ell_*}\right)\notag\\
&\geq \mathcal{H}^{d-1}\left(\partial\left(\{u_\phi>s_\phi\}\cap \big(\tphil\setminus K_{\ell_*}\big)\right)\right)+o(1)_{\phi\downarrow 0}.
  \label{subii}
\end{align}

The final ingredient that we need is the isoperimetric inequality on the torus \eqref{Regular.Isop.Torus2}, which we apply to $\{u_\phi> s_\phi\}\cap (\tphil\setminus K_{\ell_*})$,
recalling the bound on this set implied by \eqref{bdz} and \eqref{deficit}. Using \eqref{Regular.Isop.Torus2} in
 \eqref{subii} and substituting the result in \eqref{firp} leads to
\begin{align}
\lefteqn{\intfil e_\phi(u_\phi)\chi_2(u_\phi)\mathbf{1}_{\tphil\setminus K}\,dx}\notag\\
  &\geq \Big(\bar{C}_1+o(1)_{\eta\downarrow 0}\Big)\left(|\cald(u_\phi)\cap (\tphil\setminus K_{\ell_*})|\right)^{(d-1)/d}+o(1)_{\phi\downarrow 0}.\label{pert2}
\end{align}
Moreover we can improve from \eqref{deficit} to
\begin{align}
\left |\cald(u_\phi)\cap \left(\tphil\setminus K_{\ell_*}\right)\right|\to \frac{\beta}{4},\label{deficitb}
\end{align}
using the $L^2$ convergence of $u_\phi$ to $-1$ on $K_{2k}\setminus K$.

Adding \eqref{fir}, \eqref{pert1}, and \eqref{pert2} and using \eqref{deficitb} and \eqref{vol} to pass to the limit (first in $\phi$ and then in $\eta$) leads to
\begin{align}
 \lefteqn{ \liminf_{\phi\to 0} \ephi(u_\phi)}\notag\\
 &\geq c_0{\rm Per}(\{u_0=1\})+\bar{C}_1\left(\frac{\beta}{4}\right)^{(d-1)/d}-4\left( |\{u_0=1\}|+\frac{\beta}{4}\right)+4
 \frac{\left( |\{u_0=1\}|+\frac{\beta}{4}\right)^2}{\xi^{d+1}}.\label{uh}
\end{align}

\underline{Step 5: Derivation of a contradiction.}
We now observe that the right-hand side of \eqref{uh} is exactly the energy of the function $\tilde{u}$ defined as follows.
Let $\tilde{u}=u_0$ on $K$. Setting $K_{R}:=\{x\in\R^d\colon {\rm dist}(x,K)\leq R\}$, let $\tilde{u}=+1$ on a disk of volume $\beta/4$ in $\R^d\setminus K_{2R}$, and $\tilde{u}=-1$ otherwise. Here we choose $R$ big enough so that $K_{2R}\setminus K$ contains all balls of volume $\nu_m$ whose centers lie on $\partial K_R$.
The function $\tilde{u}$ so defined satisfies
\begin{align}
  ||\tilde{u}-\Psi_m||_{L^2(\R^d)}^2=||u_0-\Psi_m||_{L^2(K)}^2+4\left(\frac{\beta}{4}\right)\overset{\eqref{l2s}}=\gamma^2.\label{gamfar}
\end{align}
Moreover, we claim that $\Psi_m$ is optimal for $\tilde{u}$ in the sense that
\begin{align}
|\tilde{u}-\Psi_m|_{\R^d}=  ||\tilde{u}-\Psi_m||_{L^2(\R^d)},\label{gamfar2}
\end{align}
so that \eqref{gamfar} improves to
\begin{align}
|\tilde{u}-\Psi_m|_{\R^d}= \gamma.\label{gamfar3}
\end{align}
Indeed, we have on the one hand that for any $x_0\in\R^d$ such that $\Psi_m(\cdot-x_0)$ has $\{\Psi_m(\cdot-x_0)=1\}\cap K=\emptyset$, there holds
\begin{eqnarray*}
 ||\tilde{u}-\Psi_m(\cdot-x_0)||_{L^2(\R^d)}^2&\geq& 4\left( |\{u_0=1\}|+\nu_m-\frac{\beta}{4}\right)\\
 &\geq & 4\left(\nu_m-|\{u_0=-1,\,\Psi_m=1\}|+\nu_m-\frac{\beta}{4}\right)\\
 &\overset{\eqref{l2s}}\geq &4\left(\nu_m-\frac{\gamma^2}{4}+\frac{\beta}{4}+\nu_m-\frac{\beta}{4}\right)\\
 &=&4\left(2\nu_m-\frac{\gamma^2}{4}\right)\overset{\eqref{gm0}}\geq 4\Big(2\nu_m-(\nu_m-\nu_s)\Big)>\gamma_0^2\geq\gamma^2,
\end{eqnarray*}
so that \eqref{gamfar} implies that the centered minimizer $\Psi_m$ beats any such shift.
On the other hand, for any $x_0\in K_{R}$, optimality of $\Psi_m(\cdot-0)$ for $u_0$ is inherited from $u_\phi$ because of the strong $L^2$ convergence on $K_{2R}$:
\begin{align*}
||u_0-\Psi_m(\cdot-0)||_{L^2(K_{2R})}\leq||u_0-\Psi_m(\cdot-x_0)||_{L^2(K_{2R})},
\end{align*}
so that in particular
\begin{align*}
||\tilde{u}-\Psi_m(\cdot-0)||_{L^2(\R^d)}\leq||\tilde{u}-\Psi_m(\cdot-x_0)||_{L^2(\R^d)}.
\end{align*}

The combination of \eqref{gamfar3}, \eqref{uh}, and \eqref{eless} leads to
\begin{align*}
 \inf_{\gamma\leq|u-\Psi_m|_{\R^d} \leq\gamma_0}\ez(u)\overset{\eqref{gamfar3}}\leq  \ez(\tilde{u})
 \overset{\eqref{uh}}\leq \liminf_{\phi\to 0} \ephi(u_\phi) \overset{\eqref{eless}}\leq \inf_{\gamma\leq|u-\Psi_m|_{\R^d} \leq\gamma_0}\ez(u)-\delta.
\end{align*}
This contradiction completes the proof.
\end{proof}

It remains to prove lemma \ref{lc2}. The proof mirrors (almost exactly) the proof of lemma \ref{lc}. Hence we will be brief and highlight only the differences.
\begin{proof}[Proof of lemma \ref{lc2}: Finite $\phi$ estimate for the saddle point]
We assume for a contradiction that there  exists $\gamma\in(0,\gamma_0]$ and $\delta>0$ such that there exists a sequence $\phi\downarrow 0$ and a corresponding sequence of functions $u_\phi\in X_\phi$ such that
\begin{align}
\lim_{\phi\to 0}\nu(u_\phi)&=\nu_s,\quad |u_\phi-\Psi_s|_{\tphil}\geq \gamma,\label{g2}\\
\qquad\text{and}\qquad \ephi(u_\phi)&\leq \inf_{\substack{|u-\Psi_{s}|_{\R^d}\geq\gamma\\ \nu(u)=\nu_s}}\ez(u)-\delta.\label{eless2}
\end{align}
From $\ephi(u_\phi)\leq E_M$, proposition \ref{lemma.lower.bound} and lemma \ref{nobigsets}, we deduce $\cald(u_\phi)\lesssim 1$ and consequently that $|u_\phi-\Psi_s|_{\tphil}\lesssim 1$. Indeed, we estimate roughly
\begin{align*}
  |u_\phi-\Psi_s|_{\tphil}&\leq ||u_\phi-\Psi_s||_{L^2(\tphil)}^2
  =\int_{B_{\nu_s(x)}}(u_\phi-1)^2\,dx+\int_{\tphil\setminus B_{\nu_s}(x)}(u_\phi+1)^2\,dx\\
  &\lesssim \nu_s+\cald(u_\phi)\lesssim 1.
\end{align*}
Hence we may without loss of generality assume that
\begin{align*}
  |u_\phi-\Psi_s|_{\tphil}\to \tilde{\gamma}\geq \gamma.
\end{align*}
(For ease of notation, we drop the tilde.) In the following argument,  from the proof of lemma \ref{lc} is replaced by
\begin{align*}
|u_\phi-\Psi_s|_{\tphil}&= \gamma+o(1).
\end{align*}
Also, by translating on the torus, we assume as in the proof of lemma \ref{lc} that
\begin{align*}
  \inf_{x_0\in\tphil} ||u_\phi-\Psi_s(\cdot-x_0)||_{L^2(\tphil)}=||u_\phi-\Psi_s(\cdot-0)||_{L^2(\tphil)}.
\end{align*}

The analogues of steps 1-3 of  the proof of lemma \ref{lc} carry over to our setting. Because $\nu(u_\phi)=\nu_s+o(1)$, the estimates in \eqref{fir} simplify to
\begin{align}
  \intfil e_\phi(u_\phi)\chi_3(u_\phi)\,dx&\geq -C_2(\phi)(\nu_s+o(1)),\notag\\
  \intfil e_\phi(u_\phi)\chi_1(u_\phi)\,dx&\geq C_3(\phi)\frac{(\nu_s+o(1))^2}{\xi^{d+1}},\label{fir2}
\end{align}
while the perimeter estimate carries over unchanged. We are led to
\begin{align}
 \lefteqn{ \liminf_{\phi\to 0} \ephi(u_\phi)}\notag\\
 &\geq c_0{\rm Per}(\{u_0=1\})+\bar{C}_1 \Big(\frac{\beta}{4}\Big)^{(d-1)/d}-4\nu_s+4
 \frac{\nu_s^2}{\xi^{d+1}}.\label{uh2}
\end{align}
On the other hand, the analogues of \eqref{nuk} and \eqref{deficit2} together with $\nu(u_\phi)=\nu_s+o(1)$ imply the relation
\begin{align}
  \nu_s=|\{u_0=1\}|+\frac{\beta}{4},\label{als}
\end{align}
so that, in analogy to the proof of lemma \ref{lc}, we recognize the right-hand side of \eqref{uh2} as the energy of a sharp-interface function $\tilde{u}$ that agrees with $u_0$ on $K$ and takes value $+1$ on a ball of volume $\beta/4$ somewhere in $\R^d\setminus K_{2R}$, and is $-1$ otherwise. We observe that
\begin{align}
  \nu(\tilde{u})=|\{u_0=1\}|+\frac{\beta}{4}\overset{\eqref{als}}=\nu_s.\label{als2}
\end{align}
On the other hand, exactly as in the proof of lemma \ref{lc}, we observe
\begin{align}
  |\tilde{u}-\Psi_s|_{\R^d}=\gamma\label{gamfar4}
\end{align}
and obtain the contradiction
\begin{align*}
 \inf_{\substack{|u-\Psi_{s}|_{\R^d}\geq\gamma\\ \nu(u)=\nu_s}}\ez(u)\overset{\eqref{als2},\eqref{gamfar4}}\leq  \ez(\tilde{u})
 \overset{\eqref{uh2}}\leq \liminf_{\phi\to 0} \ephi(u_\phi) \overset{\eqref{eless2}}\leq \inf_{\substack{|u-\Psi_{s}|_{\R^d}\geq\gamma\\ \nu(u)=\nu_s}}\ez(u)-\delta.
\end{align*}
\end{proof}

\section{Steiner symmetrization and finer results in $d=2$}\label{S:stf}
In this section we derive more detailed results for the constrained minimizers using Steiner symmetrization, the Euler-Lagrange equation, and the Bonnesen inequality. Symmetrization techniques have been widely used to establish symmetry of global minimizers of various energies (see for instance \cites{BL,K,LN}).
We mention in addition the continuous symmetrization of Brock (cf. \cite{B} and the references therein), which he has used in some settings to establish symmetry of local minimizers.

When uniqueness of a minimizer is known a-priori, its symmetry often follows automatically. When uniqueness is not assured, it becomes important to discuss the equality of the energy of a given function and that of its symmetrization. For Dirichlet type functionals and Schwarz symmetrization, this has been done in \cite{BZ}; for Steiner symmetrization, the first sufficient conditions for equality go back to \cite{K}, and sharp conditions for Dirichlet boundary conditions were presented recently in \cite{CF}. We will check that the analysis of \cite{CF} carries over to the torus.

We begin by recalling the definition and properties of Steiner symmetrization on the torus in subsection \ref{ss:steiner}. Then in subsection \ref{ss:steinuo} we apply Steiner symmetrization to the constrained minimizers, deducing symmetry and connectedness of superlevel sets. Connectedness is used together with the Bonnesen inequality in subsection \ref{S:finer} to make more precise the droplet-like shape of the constrained minimizers in $d=2$.

\subsection{Steiner symmetrization}\label{ss:steiner}
In order to recall Steiner symmetrization on the torus, we will need some notation.
For $i\in\{1,\hdots,d\}$, we define
\begin{eqnarray*}
\hat{x}_{i}:=(x_1,\hdots,x_{i-1},x_{i+1},\hdots,x_d),\quad\hbox{and}\quad
(\hat{x}_{i},y):=(x_1,\hdots,x_{i-1},y,x_{i+1},\hdots,x_d).
\end{eqnarray*}
Let
\begin{eqnarray*}
\tphil^{i}:=\{\hat{x}_{i}:x\in\tphil\}.
\end{eqnarray*}
For $u\in H^{1}(\tphil)$ and $t\in\R$, we define the superlevel sets of $u$ by
\begin{eqnarray*}
\Omega_{t}:=\{u> t\}.
\end{eqnarray*}
We now fix $i\in\{1,\hdots,d\}$ and, for a given $\hat{x}_{i}\in \tphil^{i}$, set
\begin{eqnarray*}
\Omega_{t}(\hat{x}_{i}):=\{x_i: u(\hat{x}_{i},x_i)> t\}.
\end{eqnarray*}
Thus $\Omega_{t}(\hat{x}_{i})$ is one-dimensional, and we will denote
\begin{align}
\mu_{u}(\hat{x}_{i},t):=\LM^{1}(\Omega_{t}(\hat{x}_{i})),\label{mu}
 \end{align}
 where $\LM^1(E)$ stands for the one-dimensional Lebesgue measure of a measurable set $E\subset\R$. The symmetrization of this set is then defined
as
\begin{eqnarray*}
\Omega_{t}^{*}(\hat{x}_{i}):=\left\{(\hat{x}_{i},y)\in\tphil:0\leq\vert y\vert\leq
\frac{1}{2}\mu_{u}(\hat{x}_{i},t)\right\}.
\end{eqnarray*}
The Steiner symmetrization $S_{i}(\Omega_{t})$ of the set $\Omega_{t}$ with respect to the
hyperplane $\{x_i=0\}$ is defined by
\begin{eqnarray*}
S_{i}(\Omega_{t}):=\bigcup_{\hat{x}_{i}\in\tphil^{i}}\Omega_{t}^{*}(\hat{x}_{i}).
\end{eqnarray*}
We repeat this construction for each coordinate axis and define
\begin{eqnarray}\notag
\Omega_{t}^{*}:=S_{d}\circ \hdots\circ S_{1}(\Omega_{t})
\end{eqnarray}
as the Steiner symmetrization of $\Omega_{t}$. (The order matters, since there are sets $\Omega$ for which
\begin{eqnarray*}
S_{i}\circ S_{j}(\Omega)\ne S_{j}\circ S_{i}(\Omega)
\end{eqnarray*}
if $i\ne j$.) By construction $\vert\Omega_{t}^{*}\vert=\vert\Omega_{t}\vert$; we will refer to this property by saying that \emph{Steiner symmetrization is equimeasurable}.
We define the Steiner symmetrization $u^{*}$ of a function $u$ by
\begin{eqnarray*}
u^{*}(x):=\sup\{t\in\R : x\in \Omega_{t}^{*}\}\qquad\hbox{for}\quad x\in\:\tphil.
\end{eqnarray*}
The equimeasurability implies in particular that $\mu_{u}(\hat{x}_{i},t)=\mu_{u^{*}}(\hat{x}_{i},t)$ for $i=1,\ldots,d$.
In what follows we will frequently use the notation
\begin{eqnarray*}
u^{*}(\hat{x}_{i},y)=u^{*}(x_1,\hdots,x_{i-1},y,x_{i+1},\hdots,x_d).
\end{eqnarray*}
\begin{remark}\label{prop}
By construction, $u^{*}$ has the following properties:
\begin{itemize}
\item[(i)] $u^{*}(\hat{x}_{i},x_i)=u^{*}(\hat{x}_{i},-x_i)$ for $i=1,\hdots,d$;
\item[(ii)] the superlevel sets of $u^{*}$ are simply connected and starshaped with respect to the origin;
\item[(iii)] $\partial_{i}u^{*}(x)\leq 0$ on $\{x\in\tphil : 0\leq x_i\leq\frac{\phi L}{2}\}$.
\end{itemize}
\end{remark}
In  \cite[theorem 2.31]{K} it was proved that Steiner symmetrization on the torus satisfies
\begin{eqnarray}
\label{rearr1}&&\int\limits_{\tphil}\vert\nabla u\vert^2\:dx
\geq\int\limits_{\tphil}\vert\nabla u^{*}\vert^2\:dx\\
\label{rearr2}&&\int\limits_{\tphil}F(u)\:dx
=\int\limits_{\tphil}F(u^{*})\:dx\qquad \text{for measurable functions $F$}.
\end{eqnarray}

In the next subsection, we will apply Steiner symmetrization to volume-constrained minimizers and conclude from \eqref{rearr1}, \eqref{rearr2} that there exist Steiner symmetric volume-constrained minimizers. A natural next question is whether \emph{all such constrained minimizers} are Steiner symmetric, which requires studying
 the case of equality in \eqref{rearr1}. This has been done for Sobolev functions subject to a Dirichlet boundary condition on suitable measurable subsets of $\R^d$ in \cite[theorem 2.2 and section 1]{CF}.
Via a mild adaptation of the proof in \cite{CF}, we obtain the following result on the $d$-torus. We assume $u\in C^1(\tphil)$ since this is the case in our application and elements of the proof simplify.
\begin{theorem}\label{t:steineq}
Let $u\in C^{1}(\tphil)$ satisfy
\begin{eqnarray}\label{keyass}
\left|\left\{(\hat{x}_i,y)\in\tphil : \partial_{y}u(\hat{x}_{i},y)=0,\: u(\hat{x}_{i},y)<M(\hat{x}_{i})\right\}\right|=0
\end{eqnarray}
for all $i=1,\ldots,d$, where
\begin{eqnarray*}
M(\hat{x}_i):=\max\left\{u(\hat{x}_i,y), y\in \left[-\frac{\phi L}{2},\frac{\phi L}{2}\right]\right\}.
\end{eqnarray*}
If
\begin{eqnarray}\label{inteq}
\int\limits_{\tphil}\vert\nabla u\vert^2\:dx=\int\limits_{\tphil}\vert\nabla u^{*}\vert^2\:dx,
\end{eqnarray}
then there exists $a\in \tphil$ such that $u$ is Steiner symmetric about the point $a$.
\end{theorem}
\begin{remark}\label{rekeyass}
Notice that \eqref{keyass} implies that $u>\inf_{\tphil}u$ a.e. in $\tphil$. Indeed, $\partial_{y}u(\hat{x}_{i},y)=0$ on $\{u=\inf_{\tphil}u\}$, and if the latter set had positive measure, then we would contradict \eqref{keyass}.
Conversely, if $u>\inf_{\tphil}u$ a.e. then it is enough for \eqref{keyass} to check that the measure of $\partial_{y}u(\hat{x}_{i},y)=0$ on $\inf_{\tphil}u<u(\hat{x}_i,y)<M(\hat{x}_i)$ is zero. This will be the way that we check the condition in our application.
\end{remark}
For completeness, we give the proof (which simplifies considerably in our $C^1$ setting) in the appendix. The following lemma says that condition \eqref{keyass} can be equivalently formulated in terms of $u^{*}$ rather than $u$. It is the analogue of \cite[proposition 2.3]{CF}.
\begin{lemma}\label{steinereq}
Let $u\in C^1(\tphil)$. Then for all $\hat{x}_{i}\in\tphil^{i}$ we have
\begin{eqnarray*}
&&\LM^1\left(\left\{y:\partial_{y}u(\hat{x}_{i},y)=0,t<u(\hat{x}_{i},y)<M(\hat{x}_{i})\right\}\right)\\
&&\qquad=
\LM^1\left(\left\{y:\partial_{y}u^{*}(\hat{x}_{i},y)=0,t<u^{*}(\hat{x}_{i},y)<M(\hat{x}_{i})\right\}\right)
\end{eqnarray*}
for almost all $t\in(\inf_{\tphil}u,M(\hat{x}_{i}))$ and for all $i=1,\ldots,d$.
\end{lemma}
The proof of this lemma is also given in the appendix.
\subsection{Steiner symmetrization of the volume-constrained minimizers}\label{ss:steinuo}
In this subsection we deduce additional properties of the volume-constrained minimizers (and hence in particular of $\um$ and $\uws$) via the Steiner symmetrization. We recall that the volume-constrained minimizers $\uo$ minimize $\ephi$ over $X_\phi$ subject to $\nu(u)=\omega$ and
 that their existence is assured by the direct method of the calculus of variations.
According to the theory of constrained minimization, there exist two Lagrange multipliers $\lambda_{\phi}\in\R$
and $\lambda_{\omega}\in\R$ such that
\begin{eqnarray}\label{euler1}
D\ephi(u_{\omega,\phi})(w)+
\lambda_{\phi}\int\limits_{\tphil}w\:dx
+
\lambda_{\omega}\int\limits_{\tphil}\chi_3'(u_{\omega,\phi})\:w\:dx
=0
\end{eqnarray}
for all $w\in H^{1}\cap L^4(\tphil)$, where
\begin{eqnarray*}
D\ephi(u_{\omega,\phi})(w):=\phi\int\limits_{\tphil}\nabla u_{\omega,\phi}\cdot\nabla w\:dx
+
\frac{1}{\phi}\int\limits_{\tphil}G'(u_{\omega,\phi})w\:dx.
\end{eqnarray*}
For short we will write $u$ instead of $u_{\omega,\phi}$ for the rest of this subsection.
 Note that the Lagrange parameters $\lambda_{\phi}$ and $\lambda_{\omega}$ depend on $u$ in general. Thus we will write $\lambda_{\phi}=\lambda_{\phi}(u)$ and $\lambda_{\omega}=\lambda_{\omega}(u)$.

Recall from lemma \ref{l:linf} that any constrained minimizer is bounded. Therefore standard regularity theory applies, and we get $u\in W^{2,p}$ for all $1<p<\infty$ (see e.g. \cite[theorem 9.9]{GT}). By imbedding this implies $u\in C^{1,\alpha}$ for all $0\leq \alpha< 1$. The Schauder theory (cf.\ \cite[section 6]{GT}) and a bootstrap argument then imply that $u$ is smooth. Hence $u$ is a classical solution of
\begin{eqnarray}\label{euler2}
-\phi\Delta u+\frac{1}{\phi}G'(u)+\lambda_{\phi}(u)+\lambda_{\omega}(u)\chi_{3}'(u)=0\qquad\hbox{in}\:\tphil.
\end{eqnarray}
We are now ready to prove proposition \ref{prop:steiner}.

\begin{proof}[Proof of proposition \ref{prop:steiner}]
Let $u^*$ denote the successive Steiner symmetrization of $u$ about the origin with respect to the $d$-axes, say in the $x_1-,x_2-,\dots,x_d-$ order. (As explained in the previous subsection, the order needs to be specified.)
Because of \eqref{rearr2}, we have
\begin{eqnarray*}
\int_{\tphil} u^*\,dx=\int_{\tphil}u\,dx,\qquad \nu(u^*)=\nu(u),\qquad \int_{\tphil}G(u^*)\,dx=\int_{\tphil}G(u)\,dx.
\end{eqnarray*}
Together with \eqref{rearr1}, this implies the existence of a Steiner symmetric constrained minimizer.

Note that, by the definition of Steiner symmetrization (and up to a translation), $u^*$ is symmetric about and decreases monotonically in the direction away from the hyperplanes $x_1=0$,...,\,$x_d=0$. Moreover, the superlevel sets of $u^*$ are simply connected.

It remains to establish that $u$ is a translate of $u^*$. For this we turn to the result of \cite{CF}.
From remark \ref{prop} (iii), we know that
\begin{align}
\partial_i u^{*}\leq 0\qquad\text{ on $\{x_i\geq 0\}$}\label{ser1}
 \end{align}
for each $i=1,\ldots, d$. We will now strengthen this result.  Since $u^{*}$ is a smooth solution of the Euler Lagrange equation \eqref{euler2}, we may differentiate \eqref{euler2} with respect to $x_i$ to obtain the following linear equation for $\partial_{i}u^{*}$:
\begin{eqnarray}
-\phi\Delta\partial_{i}u^{*}+\left(\frac{1}{\phi}G''(u^{*})+\lambda_{\omega}(u^{*})\chi_{3}''(u^{*})\right)\partial_{i}u^{*}=0\label{ser2}
\end{eqnarray}
in $\{x\in\tphil : 0\leq x_i\leq\frac{\phi L}{2}\}$. Using \eqref{ser1}, \eqref{ser2}, and  the strong maximum principle of Serrin \cite[theorem 2.10]{HL}, we conclude that either $\partial_{i}u^{*}\equiv 0$ or $\partial_{i}u^{*}< 0$ on
$\left\{x\in\tphil : 0< x_i <\frac{\phi L}{2}\right\}$.
For $\phi$ small the estimate
\begin{align*}
  |u^*-\Psi(\cdot;\omega)|_{\tphil}\leq |u-\Psi(\cdot;\omega)|_{\tphil}\ll 1
\end{align*}
(cf.\ \eqref{uwtoo}) rules out the first possibility. Note that the first inequality is due to the fact that Steiner symmetrization is nonexpansive (see e.g. \cite{K} section II.2). Thus \eqref{ser1} improves to
\begin{eqnarray*}
\partial_{i}u^{*}< 0\qquad\hbox{on}
\qquad\left\{x\in\tphil : 0< x_i <\frac{\phi L}{2}\right\}
\end{eqnarray*}
for all $i=1,\ldots, d$. This implies that $u^*$ is strictly decreasing in all directions away from zero. Hence condition \eqref{keyass} is satisfied for $u^{*}$ and, by lemma \ref{steinereq} and theorem \ref{t:steineq}, $u$ is equal to a translate of $u^{*}$, as desired.
\end{proof}
\begin{remark}\label{rem:nonuq}
Proposition \ref{prop:steiner} does not establish uniqueness; there may be more than one Steiner symmetric constrained minimizer with prescribed volume $\omega$.
\end{remark}
\subsection{Refinement in $d=2$}\label{S:finer}
Using the connectedness of the superlevel sets of $\uo$ from proposition \ref{prop:steiner} together with the Bonnesen inequality, we can strengthen the quantitative estimate \eqref{visam} on the sphericity of the superlevel sets in dimension $d=2$. Loosely speaking, we can show that for any constrained minimizer $\uo$, the superlevel sets $\{\uo >\eta\}$ for $\eta \in (-1,1)$  cannot possess ``tentacles'' and are therefore close to a ball in the stronger sense of Hausdorff distance. Hence the possibility of mass drifting off to infinity (which required care in lemma \ref{lc}) is precluded. The main tool that is needed in order to establish this fact is the Bonnesen inequality, which we state below after recalling the definition of the outer and inner radius.
\begin{definition}
Consider a simply connected domain $A\subset\mathbb{R}^2$. The outer radius of $A$, denoted $\rho_{out}(A)$, is defined as the infimum of the radii of all the disks in $\mathbb{R}^2$ that contain $A$. Similarly, the inner radius of $A$, denoted $\rho_{in}(A)$, is defined as the supremum of the radii of all the disks in $\mathbb{R}^2$ that are contained in $A$. Lastly, we define the volume radius of $A$, denoted $\rho(A)$, as the radius of a disk in $\mathbb{R}^2$ whose measure is equal to that of $A$.
\end{definition}
\begin{rmrk}
We may use the same definition for the inner and outer radius of a simply connected domain $A\subset\tphil$, provided that there exists a disk in $\tphil$ that contains $A$. Note that in that case there holds $\rho_{out}(A)<(\phi L)/2$ and $\per_{\tphil}(A)=\per_{\mathbb{R}^2}(A)$.

\end{rmrk}
The classical Bonnesen inequality in the plane is as follows.
\begin{thm}[Bonnesen inequality in $\mathbb{R}^2$]\label{Bonnesen.Ineq}
For any simply connected domain $A\subset\mathbb{R}^2$ with smooth boundary, there holds
\begin{equation}\label{Bonnesen.Ineq2}
\per_{\R^2}(A)\geq \sqrt{\pi}\left(4|A|+\Big(\rho_{out}(A)-\rho_{in}(A)\Big)^2 \right)^{1/2}.
\end{equation}
\end{thm}
An application of the Bonnesen inequality to our problem yields the following result.
\begin{prop}\label{Asym.stronger}
Fix $\xi \in (\tilde{\xi_2},\xi_2]$ and consider the critical scaling \eqref{crit}. Fix any $\omega_1>0$ and consider $\phi>0$ sufficiently small. For any volume-constrained minimizer $\uo$ with volume $\omega\in[\omega_1,\xi^3/2]$ and any $\eta\in(-1,1)$, there holds
\begin{equation}\label{str.0}
\rho_{out}(\{\uo >\eta \})=r_\omega+\frac{O(\phi^{1/6})}{(1+\eta)},
\end{equation}
and
\begin{equation}\label{str.0.1}
\rho_{in}(\{\uo >\eta \})=r_\omega+\frac{O(\phi^{1/6})}{(1-\eta)},
\end{equation}
where
\begin{equation}\label{str.0.2}
r_\omega:=\sqrt{\frac{\omega}{\pi}}.
\end{equation}
Consequently, there holds
\begin{equation}\label{str.0.3}
\rho_{out}(\{\uo >\eta \})-\rho_{in}(\{\uo >\eta \})=\frac{O(\phi^{1/6})}{(1-\eta^2)}.
\end{equation}
\end{prop}
\begin{remark}
  In particular \eqref{str.0.3} holds for $\um$ and $\uws$.
\end{remark}
\begin{proof}[Proof of proposition \ref{Asym.stronger}]
We recall from subsection \ref{ss:steinuo} that $\uo$ is smooth and Steiner symmetric. We also recall \eqref{ibd2} and remark that one can deduce in the same way that
\begin{align}
\int_{-1+\phi^{1/3}}^{1-\phi^{1/3}}\sqrt{2\tilde{G}(t)}\,\Big({\rm Per}_{\tphil}(\{\uo>t\})-{\rm P}_E(\{\uo>t\})\Big)\,dt\lesssim \phi^{1/3}.\label{JJ}
\end{align}

Next we claim that we may shift $\uo$ so that $\{\uo>-1+2\phi^{1/3}\}$ is contained within a disk centered at the origin and of radius less than $(\phi L)/2$. Indeed, if this is not the case, it follows by the Steiner symmetry of $\{\uo>-1+2\phi^{1/3}\}$ that

\[
\per_{\tphil}(\{\uo>-1+2\phi^{1/3}\})\geq\frac{1}{2}\phi L,
\]
which, because of the monotonicity of $\{\uo>t\}$ with respect to $t$, implies in turn that
\[
\per_{\tphil}(\{\uo>t\})\geq\frac{1}{2}\phi L \quad \text{for all}\quad t\in[-1+\phi^{1/3},-1+2\phi^{1/3}].
\]
It follows that
\[
\int_{-1+\phi^{1/3}}^{-1+2\phi^{1/3}}\sqrt{2\tilde{G}(t)}\,\Big({\rm Per}_{\tphil}(\{u>t\})-{\rm P}_E(\{u>t\})\Big)\,dt\gtrsim \phi^{1/6},
\]
which contradicts \eqref{JJ}.

We now establish a lower bound on $I(\uo)$. For any $t\in[-1+2\phi^{1/3},1-2\phi^{1/3}]$ we will denote by $\rho(t),\,\rho_{in}(t)$ and $\rho_{out}(t)$ the volume-, inner- and outer radius of $\{\uo>t\}$, respectively. We will also use the notation
\[
\mathit{\Delta} \rho(t):=\rho_{out}(t)-\rho_{in}(t).
\]
Because the superlevel sets of $\uo$ are contained within a disk (as discussed above), we may apply the Bonnesen inequality \eqref{Bonnesen.Ineq2} to $I$ to obtain
\begin{align*}
&I(\uo)\notag\\
&\geq \int_{-1+2\phi^{1/3}}^{1-2\phi^{1/3}}\sqrt{2\tilde{G}(t)}\,\left[(4\pi|\{\uo>t\}|+\pi(\mathit{\mathit{\Delta}} \rho(t))^2)^{1/2}-{\rm P}_E(\{\uo>t\})\right]\,dt \notag \\
&=2\sqrt{2\pi}\int_{-1+2\phi^{1/3}}^{1-2\phi^{1/3}}\sqrt{\tilde{G}(t)}\,|\{\uo>t\}|^{1/2}\left[\left(1+\frac{(\mathit{\Delta} \rho(t))^2}{4|\{\uo>t\}|}\right)^{1/2}-1\right]\,dt\notag\\
&\gtrsim \int_{-1+2\phi^{1/3}}^{1-2\phi^{1/3}}\sqrt{\tilde{G}(t)}\,\frac{(\mathit{\Delta} \rho(t))^2}{|\{\uo>t\}|^{1/2}}\,dt.
\end{align*}
We combine this with the bound \eqref{ibd2},  corollary \ref{cor.1}, and the  bound on $\nu(u)$ to deduce
\begin{align}\label{str.1}
\int_{-1+2\phi^{1/3}}^{1-2\phi^{1/3}}\sqrt{\tilde{G}(t)}\,(\mathit{\Delta} \rho(t))^2\,dt\lesssim \phi^{1/3}.
\end{align}
Next we observe that, due to the monotonicity of the superlevel sets $\{\uo>t\}$ with respect to $t$, we have
\begin{equation}\label{str.2}
\rho_{out}(t)\geq\rho_{out}(\eta)\quad \text{for every}\quad t\in[-1+2\phi^{1/3},\eta],
\end{equation}
and
\begin{equation}\label{str.3}
\rho_{in}(t)\leq\rho_{in}(\eta)\quad \text{for every}\quad t\in[\eta,1-2\phi^{1/3}].
\end{equation}
Moreover, again due to monotonicity, for all $t\in[-1+2\phi^{1/3},1-2\phi^{1/3}]$ there holds
\begin{eqnarray}
\rho_{out}(t)&\geq&\rho_{out}(1-2\phi^{1/3})\geq \left(|\{\uo>1-2\phi^{1/3}\}|/\pi\right)^{1/2}\notag \\
&\overset{\eqref{saddle.superset}}=&r_\omega+O(\phi^{1/6}),\label{str.4}
\end{eqnarray}
and
\begin{eqnarray}
\rho_{in}(t)&\leq&\rho_{in}(-1+2\phi^{1/3})\leq \left(|\{\uo>-1+2\phi^{1/3}\}|/\pi\right)^{1/2}\notag\\
&\overset{\eqref{saddle.superset}}=&r_{\omega}+O(\phi^{1/6}).\label{str.5}
\end{eqnarray}

By \eqref{str.2} and \eqref{str.5} it follows that for all $t\in [-1+2\phi^{1/3},\eta]$ the difference $\mathit{\Delta} \rho(t)$ satisfies
\[
\mathit{\Delta} \rho(t)\geq  \rho_{out}(\eta)-r_\omega+O(\phi^{1/6}).
\]
Substituting into \eqref{str.1} implies \eqref{str.0}.

For  $s\in [\eta,1-2\phi^{1/3}]$ on the other hand, \eqref{str.3} and \eqref{str.4} imply that
\[
\mathit{\Delta} \rho(s)\geq  r_\omega-\rho_{in}(\eta)+O(\phi^{1/6}),
\]
which together with \eqref{str.1} yields \eqref{str.0.1}.

\end{proof}



\section*{Appendix: Isoperimetry on the torus}
\begin{proof}[Proof of corollary \ref{SharpIsoper}]
Our approach is similar to the one used in establishing \cite[theorem 6.2]{CCELM}. Note first that $\epsilon>0$ can be chosen small enough to ensure that the diameter of a ball of volume $|A|$ is less than $1/2$. This is the only restriction on the value of $\epsilon$. With that in mind, let $B$ be a ball that achieves the optimal overlap with $A$ in the definition of the Fraenkel asymmetry. In other words, consider a ball $B$ of volume $|A|$, such that $|A\triangle B|=\lambda(A)|A|$. Since we are working on the torus, we can with no loss of generality assume that $B$ is centered at the origin.

Next, define
\[
a_1(t):=\mathcal{H}^{d-1}(A\cap \{x_1=t\})\quad \text{for}\quad t\in \Big[-\frac{1}{2},\frac{1}{2}\Big],
\]
and observe that, with $I:=[-1/2,1/2]\setminus[-1/4,1/4]$ there holds
\[
|A|=\int_{-1/2}^{1/2}a_1(t)\,dt\geq \int_I a_1(t)\,dt\geq \frac{1}{2}\mathrm{ess}\,\mathrm{inf}_{\substack{I}}a_1,
\]
from which it follows that for arbitrary $\delta>0$  there exists $t_1\in I$ such that
\[
a_1(t_1)\leq 2|A|+\delta.
\]
Shifting in the torus in the $x_1$-direction by $t_1$, if necessary, we may thus assume that
\[
a_1(-1/2)=a_1(1/2)\leq 2|A|+\delta.
 \]
Repeating this process for the other $d-1$ directions, and since the $d$ successive translations are independent of each other, we may assume that, in fact,
\begin{equation}\label{Appendix0}
a_i(-1/2)=a_i(1/2)\leq 2|A|+\delta,\,\,\text{for all}\,\,i\in\{1,\dots,d\}.
\end{equation}
Notice that the set $\tilde{B}$ that is obtained by shifting the ball $B$ in all $d$ directions is still a ball of volume $|A|$ that is contained  in the open set $(-1/2,1/2)^d\subset\R^d$. Moreover, the Fraenkel asymmetry of the set $\tilde{A}\subset\mathbb{R}^d$ that is obtained by shifting $A$ is the same as that of $A$, for if we regard $\tilde{A}$ and $\tilde{B}$ as subsets of $\mathbb{R}^d$, the ball $\tilde{B}$ is an optimal ball for $\tilde{A}$, and $|\tilde{A}\triangle \tilde{B}|=\lambda(A)|\tilde{A}|$.

It thus follows from the quantitative isoperimetric inequality in $\mathbb{R}^d$ that
\begin{align}\label{Appendix1}
{\rm Per}_{\mathbb{R}^d}(\tilde{A})\geq {\rm P}_E(\tilde{A})+C(d)\lambda(\tilde{A})^2 {\rm P}_E(\tilde{B}).
\end{align}
Note that we have
\begin{equation}\label{Appendix2}
{\rm Per}_{\mathbb{T}_1}(\tilde{A})\geq {\rm Per}_{\mathbb{R}^d}(\tilde{A})-2d(2|\tilde{A}|+\delta),
\end{equation}
since, as a consequence of \eqref{Appendix0}, dropping the identification of opposite ends of the torus (and thereby regarding $\tilde{A}$ merely as a subset of $\mathbb{R}^d$), increases the perimeter ${\rm Per}_{\mathbb{T}_1}(\tilde{A})$ by at most $2d(2|\tilde{A}|+\delta)$.
Since $\delta>0$ is arbitrary, it follows from \eqref{Appendix2} that
\begin{equation}\label{Appendix3}
{\rm Per}_{\mathbb{T}_1}(\tilde{A})\geq {\rm Per}_{\mathbb{R}^d}(\tilde{A})-4d|\tilde{A}|,
\end{equation}
Substituting \eqref{Appendix1} into \eqref{Appendix3} and recalling
\[
{\rm Per}_{\mathbb{T}_1}(\tilde{A})={\rm Per}_{\mathbb{T}_1}(A),\,\,\,|\tilde{A}|=|A|,\,\,\,\lambda(\tilde{A})=\lambda(A)\,\,\,\text{and}\,\,\,{\rm P}_E(\tilde{B})={\rm P}_E(A),
\]
we finally obtain
\[
{\rm Per}_{\mathbb{T}_1}(A)\geq {\rm P}_E(A)+C(d)\lambda(A)^2 {\rm P}_E(A)-4d|A|.
\]
\end{proof}




\section*{Appendix: P\'olya-Szeg\"o inequality for smooth functions on the torus.}
In this section we adapt the proofs of theorem 2.2 and proposition 2.3 from \cite{CF} in order to establish lemma \ref{steinereq} and
theorem \ref{t:steineq}.
\medskip

We will
use the notation from section 4.1. Without loss of generality, it suffices to consider $i=d$. Slightly deviating from our notation there, we write
\begin{eqnarray}\notag
\Omega_{t}^{*}:=S_{d}(\Omega_{t})
\end{eqnarray}
for the Steiner symmetrization $S_{d}(\Omega_{t})$ of the set $\Omega_{t}$ with respect to the
hyperplane $\{x_d=0\}$. We will also write $\hat{x}$ instead of $\hat{x}_{d}$ and set
\begin{eqnarray*}
m(\hat{x})&:=&\min\left\{u(\hat{x},y), y\in \left[-\frac{\phi L}{2},\frac{\phi L}{2}\right]\right\},\\
M(\hat{x})&:=&\max\left\{u(\hat{x},y), y\in \left[-\frac{\phi L}{2},\frac{\phi L}{2}\right]\right\}.
\end{eqnarray*}
\begin{remark}\label{regdist}
In the sequel we will use lemma 4.1 from \cite{CF}, which gives the following regularity for the distribution function $\mu_u(\hat{x},t)$:
For almost all $\hat{x}\in\tphild$, there holds
\begin{eqnarray}
\label{regdist1}\partial_{t}\mu_u(\hat{x},t)&=&-\int\limits_{\partial\{y:u(\hat{x},y)>t\}} \frac{1}{ \vert\partial_{y}u\vert } \:d\HM^{0}\\
\label{regdist2}\partial_{i}\mu_u(\hat{x},t)&=&\int\limits_{\partial\{y:u(\hat{x},y)>t\}}  \frac{\partial_{i}u}{ \vert\partial_{y}u\vert }
\:d\HM^{0}\qquad i=1,\ldots,d-1,
\end{eqnarray}
for almost all $t\in (m(\hat{x}),M(\hat{x}))$.
The proof carries over to the case $u\in C^1(\tphil)$ with some simplifications due to the smoothness assumption on $u$.
\end{remark}
\begin{proof}[Proof of lemma \ref{steinereq}]
\underline{Step 1.} For all $\hat{x}\in\tphild$ and for every $t\in (m(\hat{x}),M(\hat{x}))$, we have the decomposition
\begin{eqnarray*}
\mu_{u}(\hat{x},t)&=&\LM^{1}\left(\left\{y:u(\hat{x},y)=M(\hat{x})\right\}\right)\\
&&+
\LM^{1}\left(\left\{y:\partial_{y}u(\hat{x},y)=0,\: t< u(\hat{x},y)<M(\hat{x}) \right\}\right)\\
&&+
\LM^{1}\left(\left\{y:\partial_{y}u(\hat{x},y)\ne 0,\: t< u(\hat{x},y)<M(\hat{x}) \right\}\right).
\end{eqnarray*}
Defining for any $t\in (m(\hat{x}),M(\hat{x}))$ the set
\begin{eqnarray*}
D_{u}:= \left\{y:\partial_{y}u(\hat{x},y)\ne 0,\: t< u(\hat{x},y)<M(\hat{x}) \right\},
\end{eqnarray*}
we apply the coarea formula (see e.g.\ \cite[chapter 3.4]{EG} or \cite[chapter 2.12]{AFP}) to write
\begin{eqnarray}\label{1-d-coarea}
\LM^{1}\left(D_{u}\right)
=
\int\limits_{t}^{M(\hat{x})}\int\limits_{\partial \{u(\hat{x},y)>s\}}\frac{\chi_{D_{u}}}{\vert \partial_{y}u(\hat{x},y)\vert}\:d\HM^{0}\:ds.
\end{eqnarray}
Here $\HM^{0}(E)$ denotes the $0$ - dimensional Hausdorff measure (counting measure) of a set $E$ and $\chi_{D_{u}}$ denotes the characteristic function of $D_{u}$. Formula \eqref{1-d-coarea} holds true for $u^{*}$ as well.

\underline{Step 2.} Next we define the function
\begin{eqnarray*}
h(t):=\LM^{1}\left(\left\{y:\partial_{y}u(\hat{x},y)=0,\: t< u(\hat{x},y)\right\}\right).
\end{eqnarray*}
This function is non increasing and is thus in $BV_{loc}(\R)$ and right continuous, with $h'(t)=0$ for almost all $t\in\R$.  The last fact was proved in \cite[lemma 2.4]{CF1}. For completeness we give the argument in our setting. The right continuity implies
\begin{eqnarray*}
\vert h(t_2)-h(t_1)\vert=\vert Dh\vert((t_1,t_2])\qquad\forall t_1<t_2,
\end{eqnarray*}
where
\begin{eqnarray*}
\vert Dh\vert((t_1,t_2])=\sup\left\{\sum_{i=1}^{n}\vert h(s_{i+1})-h(s_{i})\vert : n\geq 2, t_1<s_1<\ldots<s_{n+1}\leq t_2\right\}.
\end{eqnarray*}
Consequently,
\begin{eqnarray*}
\LM^{1}\left(\left\{y:\partial_{y}u(\hat{x},y)=0,\: t_1< u(\hat{x},y)<t_2\right\}\right)=\vert Dh\vert((t_1,t_2)).
\end{eqnarray*}
For fixed $\hat{x}$, let $C_{u}(\hat{x}):=\left\{y : \partial_{y}u(\hat{x},y)=0\right\}$. Sard's theorem implies that $\LM^{1}\left(u(C_{u}(\hat{x}))\right)=0$ for all $\hat{x}\in\tphild$. Thus $\vert Dh\vert$ is concentrated on a set of measure zero, and this implies that $h'(t)=0$ for almost all $t\in\R$.

\underline{Step 3.} Since the previous step applies equally to $u^{*}$, we obtain
\begin{eqnarray*}
&&\frac{d}{dt}\LM^{1}\left(\left\{y:\partial_{y}u(\hat{x},y)=0,\: t< u(\hat{x},y)\right\}\right)\\
&&\qquad=
\frac{d}{dt}\LM^{1}\left(\left\{y:\partial_{y}u^{*}(\hat{x},y)=0,\: t< u^{*}(\hat{x},y)\right\}\right)=0,
\end{eqnarray*}
for $\LM^{1}$ almost all $t>m(\hat{x})$. Moreover, by the equimeasurability of Steiner symmetrization, we have
\begin{eqnarray*}
\LM^{1}\left(\left\{y:u(\hat{x},y)=M(\hat{x})\right\}\right)
=
\LM^{1}\left(\left\{y:u^{*}(\hat{x},y)=M(\hat{x})\right\}\right).
\end{eqnarray*}
Hence,
\begin{eqnarray*}
\int\limits_{t}^{M(\hat{x})}\int\limits_{\partial \{u(\hat{x},y)>s\}}\frac{1}{\vert \partial_{y}u(\hat{x},y)\vert}\:d\HM^{0}(y)\:ds
=
\int\limits_{t}^{M(\hat{x})}\int\limits_{\partial \{u^{*}(\hat{x},y)>s\}}\frac{1}{\vert \partial_{y}u^{*}(\hat{x},y)\vert}\:d\HM^{0}(y)\:ds,
\end{eqnarray*}
for all $\hat{x}\in\tphild$ and almost all $t\in(m(\hat{x}),M(\hat{x}))$.

\underline{Step 4.} We collect these facts and also use $\mu_{u}(\hat{x},t)=\mu_{u^{*}}(\hat{x},t)$. This proves the claim.
\end{proof}
\begin{definition}
Two functions $u,v\in C^{1}(\tphil)$ are called equivalent if and only if there exists an $a\in\R$ such that $u(\hat{x},y+a)=v(\hat{x},y)$ for all $\hat{x}\in\tphild$ and all $y\in[-\frac{\phi L}{2},\frac{\phi L}{2}]$.

Two sets $A,B\subset\tphil$ are called equivalent if and only if there exists a number $\alpha\in\R$ such that $A+\alpha e_d=B$ up to a set of $\LM^{d}$ - measure zero. Here $e_d=(0,\ldots,0,1)$ denotes $d$-th unit vector in $\R^d$.
\end{definition}
\begin{proof}[Proof of theorem \ref{t:steineq}]
We follow the proof of \cite{CF}. It is sufficient to show that $\Omega_{t}^{*}$ is equivalent to $\Omega_t$ in the sense that there exists a vector $a\in\R^{d}$, which does not depend on $t$ or $\hat{x}$, such that $\Omega_{t}^{*}=\Omega_t+a$. This is done in several steps. In the first three steps we use the one dimensional isoperimetric inequality on the circle $S^1$ and the Cauchy-Schwarz inequality to deduce that
\begin{eqnarray*}
\int\limits_{\tphil}\vert\nabla u\vert^2\:dx\geq \int\limits_{\tphil}\vert\nabla u^{*}\vert^2\:dx.
\end{eqnarray*}
 As in \cite{CF}, we then observe that the condition \eqref{rearr1} implies equality in the isoperimetric Cauchy-Schwarz inequalities. The subsequent steps of the proof exploit this fact to establish the theorem.

\underline{Step 1.} By remark \ref{rekeyass}, we have $u>\inf_{\tphil} u$ up to a set of measure zero. The one dimensional coarea formula gives
\begin{eqnarray}\label{coarea}
\int\limits_{\tphil}\vert\nabla u\vert^2\:dx=\int\limits_{\tphild}\int\limits_{m(\hat{x})}^{M(\hat{x})}\int\limits_{\partial\{y:u(\hat{x},y)>t\}}
\frac{\vert\nabla u\vert^2}{\vert\partial_{y}u\vert} \:d\HM^{0}\:dt\:d\hat{x}.
\end{eqnarray}
The equimeasurability of Steiner symmetrization implies
\begin{eqnarray}
-\int\limits_{\partial\{y:u(\hat{x},y)>t\}}\frac{1}{\vert\partial_{y}u\vert}\:d\HM^{0}&\overset{\eqref{regdist1}}=&\partial_{t}\mu_{u}(\hat{x},t)\notag\\
&=&\partial_{t}\mu_{u^{*}}(\hat{x},t)=-\int\limits_{\partial\{y:u^{*}(\hat{x},y)>t\}}\frac{1}{\vert\partial_{y}u^{*}\vert}\:d\HM^{0}\label{distr1}
\\
\text{and}\quad\int\limits_{\partial\{y:u(\hat{x},y)>t\}}\frac{\partial_{i}u}{\vert\partial_{y}u\vert}\:d\HM^{0}&\overset{\eqref{regdist2}}=&\partial_{i}\mu_{u}(\hat{x},t)\notag\\
&=&\partial_{i}\mu_{u^{*}}(\hat{x},t)=
\int\limits_{\partial\{y:u^{*}(\hat{x},y)>t\}}\frac{\partial_{i}u^{*}}{\vert\partial_{y}u^{*}\vert}\:d\HM^{0}\label{distr2}
\end{eqnarray}
for $i=1,\ldots,d-1$ and for almost all $t\in(m(\hat{x}),M(\hat{x}))$ and $\hat{x}\in\tphild$.
Using that $u^{*}$ is symmetric and satisfies \eqref{keyass} (cf.\ lemma \ref{steinereq}), we simplify the  right-hand sides of \eqref{distr1} and \eqref{distr2}
to
deduce the formulas
\begin{eqnarray}
\label{distr3}\partial_{t}\mu_{u}(\hat{x},t)&=&-\frac{2}{\vert\partial_{y}u^{*}\vert}\Big\vert_{\partial\{y:u^{*}(\hat{x},y)>t\}}\\
\label{distr4}\partial_{i}\mu_{u}(\hat{x},t)&=&\frac{2\partial_{i}u^{*}}{\vert\partial_{y}u^{*}\vert}\Big\vert_{\partial\{y:u^{*}(\hat{x},y)>t\}}.
\end{eqnarray}

\underline{Step 2.} We use formulas \eqref{distr3} - \eqref{distr4} to express
\begin{eqnarray*}
&&\int\limits_{\partial\{y:u^{*}(\hat{x},y)>t\}}\frac{\vert\nabla u^{*}\vert^2}{\vert\partial_{y}u^{*}\vert}\:d\HM^{0}
=
\frac{2}{\vert\partial_{y}u^{*}\vert}\left(\sum_{i=1}^{d-1}\vert\partial_{i} u^{*}\vert^2+\vert\partial_{y}u^{*}\vert^2\right)\Big\vert_{\partial\{y:u^{*}(\hat{x},y)>t\}}\\
&&=
-\partial_{t}\mu_{u}(\hat{x},t)\left(\sum_{i=1}^{d-1}\frac{\vert\partial_{i}\mu_{u}(\hat{x},t)\vert^2}{\vert\partial_{t}\mu_{u}(\hat{x},t)\vert^2}+\frac{4}{\vert\partial_{t}\mu_{u}(\hat{x},t)\vert^2}\right)\Big\vert_{\partial\{y:u^{*}(\hat{x},y)>t\}}\\
&&=
\int\limits_{\partial\{y:u(\hat{x},y)>t\}}\frac{1}{\vert\partial_{y}u\vert}\:d\HM^{0}
\left(\sum_{i=1}^{d-1}\frac{\left(\int\limits_{\partial\{y:u(\hat{x},y)>t\}}\frac{\partial_{i}u}{\vert\partial_{y}u\vert}\:d\HM^{0}\right)^2}{\left(\int\limits_{\partial\{y:u(\hat{x},y)>t\}}\frac{1}{\vert\partial_{y}u\vert}\:d\HM^{0}\right)^2}
+
\frac{4}{\left(\int\limits_{\partial\{y:u(\hat{x},y)>t\}}\frac{1}{\vert\partial_{y}u\vert}\:d\HM^{0}\right)^2}\right).
\end{eqnarray*}
\underline{Step 3.} Next we use the Cauchy-Schwarz inequality
\begin{eqnarray}\label{cau1}
\left(\int\limits_{\partial\{y:u(\hat{x},y)>t\}}\frac{\partial_{i}u}{\vert\partial_{y}u\vert}\:d\HM^{0}\right)^2
\leq
\int\limits_{\partial\{y:u(\hat{x},y)>t\}}\frac{\vert\partial_{i}u\vert^2}{\vert\partial_{y}u\vert}\:d\HM^{0}
\int\limits_{\partial\{y:u(\hat{x},y)>t\}}\frac{1}{\vert\partial_{y}u\vert}\:d\HM^{0}
\end{eqnarray}
and recall that equality holds if and only if $\partial_{i}u=c_{i}(\hat{x},t)$ for some function $c_{i}(\hat{x},t)$ that does not depend on $y$. This implies
\begin{eqnarray}\label{appcf}
\int\limits_{\partial\{y:u^{*}(\hat{x},y)>t\}}\frac{\vert\nabla u^{*}\vert^2}{\vert\partial_{y}u^{*}\vert}\:d\HM^{0}
\leq
\sum_{i=1}^{d-1}\int\limits_{\partial\{y:u(\hat{x},y)>t\}}\frac{\vert\partial_{i}u\vert^2}{\vert\partial_{y}u\vert}\:d\HM^{0}
+
\frac{4}{\int\limits_{\partial\{y:u(\hat{x},y)>t\}}\frac{1}{\vert\partial_{y}u\vert}\:d\HM^{0}}.
\end{eqnarray}
Finally, using that $u$ is $\phi L$-periodic, we deduce from the isoperimetric inequality on $S^1$  that
\begin{eqnarray}\label{appiso}
2\leq\HM^{0}(\partial\{y:u(\hat{x},y)>t\})=\int\limits_{\partial\{y:u(\hat{x},y)>t\}}\:d\HM^{0}.
\end{eqnarray}
Thus we may estimate
\begin{eqnarray}\label{cau2}
4\leq\left(\int\limits_{\partial\{y:u(\hat{x},y)>t\}}\:d\HM^{0}\right)^2
\leq
\int\limits_{\partial\{y:u(\hat{x},y)>t\}}\frac{\vert\partial_{y}u\vert^2}{\vert\partial_{y}u\vert}\:d\HM^{0}
\int\limits_{\partial\{y:u(\hat{x},y)>t\}}\frac{1}{\vert\partial_{y}u\vert}\:d\HM^{0},
\end{eqnarray}
where the last inequality is a consequence of the Cauchy-Schwarz inequality. Again the case of equality implies $\vert\partial_{y}u\vert=c_{y}(\hat{x},t)$ for some non negative function $c_{y}(\hat{x},t)$ that does not depend on $y$. Substituting \eqref{cau2} into \eqref{appcf} yields
\begin{eqnarray}\label{cau3}
\nonumber\int\limits_{\partial\{y:u^{*}(\hat{x},y)>t\}}\frac{\vert\nabla u^{*}\vert^2}{\vert\partial_{y}u^{*}\vert}\:d\HM^{0}
&\leq&
\sum_{i=1}^{d-1}\int\limits_{\partial\{y:u(\hat{x},y)>t\}}\frac{\vert\partial_{i}u\vert^2}{\vert\partial_{y}u\vert}\:d\HM^{0}
+
\int\limits_{\partial\{y:u(\hat{x},y)>t\}}\frac{\vert\partial_{y}u\vert^2}{\vert\partial_{y}u\vert}\:d\HM^{0}\\
&=&
\int\limits_{\partial\{y:u(\hat{x},y)>t\}}\frac{\vert\nabla u\vert^2}{\vert\partial_{y}u\vert}\:d\HM^{0}.
\end{eqnarray}
\underline{Step 4.} Integrating \eqref{cau3} with respect to $t$ and using the one dimensional coarea formula again, we see that the condition \eqref{inteq} implies equality in \eqref{cau3}, and hence in all four inequalities \eqref{cau1}, \eqref{appcf}, \eqref{appiso} and \eqref{cau2} above. We collect the results.
\begin{itemize}
\item[(1)] Because of the equality in \eqref{appiso}, there exist two functions $y_1(\hat{x},t)$ and $y_2(\hat{x},t)$, such that
\begin{eqnarray}\label{cau4}
\nonumber\{y: u(\hat{x},y)>t\}\quad\hbox{is equivalent to}\quad (y_1(\hat{x},t), y_2(\hat{x},t)),
\end{eqnarray}
for all $\hat{x}\in\tphild$ and for almost all $t\in (m(\hat{x}), M(\hat{x}))$.
Moreover, for almost all $t$ the functions $y_i$ are smooth functions of $\hat{x}$ and differentiable in $t$. This is a consequence of \eqref{keyass}, the implicit function theorem, and the monotonicity in $t$.
\item[(2)] Due to the equality case in the Cauchy-Schwarz inequality, there exist functions $c_{i}(\hat{x},t)$ and $c_{y}(\hat{x},t)\geq 0$ (which do not depend on $y$), such that
\begin{eqnarray}
\label{cau5} \partial_{i}u(\hat{x},y_1(\hat{x},t)) &=&\partial_{i}u(\hat{x},y_2(\hat{x},t))=c_{i}(\hat{x},t)\qquad i=1,\ldots,d-1;\\
\label{cau6} \vert\partial_{y}u(\hat{x},y_1(\hat{x},t))\vert&=&\vert\partial_{y}u(\hat{x},y_2(\hat{x},t))\vert=c_{y}(\hat{x},t)
\end{eqnarray}
for all $\hat{x}\in\tphild$ and for almost all $t\in (m(\hat{x}), M(\hat{x}))$. In particular we have
\begin{eqnarray}\label{cau8}
\vert\nabla u(\hat{x},y_1)\vert=\vert\nabla u(\hat{x},y_2)\vert
\end{eqnarray}
\item[(4)] Since the $y_i$ are endpoints of the set $\{y: u(\hat{x},y)>t\}$, we have
\begin{eqnarray}\label{cau7}
\partial_{y}u(\hat{x},y_1(\hat{x},t))=-\partial_{y}u(\hat{x},y_2(\hat{x},t)).
\end{eqnarray}
\end{itemize}
\underline{Step 5.} Let
\begin{eqnarray*}
E:=\{(\hat{x},y,t) : y_1(\hat{x},t)<y<y_2(\hat{x},t)\}.
\end{eqnarray*}
Since the boundary of $E$ is a level set of $u$, we have an explicit representation of the unit normal vector in each point of $\partial E$ (for almost all $t$) as
\begin{eqnarray}\label{appnormal1}
\nu_{E}(\hat{x},y,t)=\left( \frac{\hat{\nabla}u(\hat{x},y)}{\sqrt{1+\vert\nabla u\vert^2}},\frac{\partial_{y}u(\hat{x},y)}{\sqrt{1+\vert\nabla u\vert^2}},\frac{-1}{\sqrt{1+\vert\nabla u\vert^2}}\right)\qquad\forall (\hat{x},y,t)\in\partial E,
\end{eqnarray}
where
\begin{eqnarray*}
\hat{\nabla}u=\left(\partial_{1}u,\ldots,\partial_{d-1}u\right).
\end{eqnarray*}
This vector points into the set $\{u>t\}$ along $\partial E$. Next we write $E=E_1\cap E_2$, where
\begin{eqnarray*}
E_1:=\{(\hat{x},y,t) : y_1(\hat{x},t)<y\},\quad\hbox{and}\quad
E_2:=\{(\hat{x},y,t) : y<y_2(\hat{x},t)\}.
\end{eqnarray*}
We use the representation of $\partial E_{i}$, $i=1,2$ as a graph to derive
\begin{eqnarray}\label{appnormal2}
\nu_{E_{1}}(\hat{x},y_1,t)=-\left( \frac{\hat{\nabla} y_{1}(\hat{x},t)}{\sqrt{1+\vert\nabla y_1\vert^2}},\frac{-1}{\sqrt{1+\vert\nabla y_1\vert^2}},\frac{\partial_{t}y_1(\hat{x},t)}{\sqrt{1+\vert\nabla y_1\vert^2}}\right)
\end{eqnarray}
and
\begin{eqnarray}\label{appnormal3}
\nu_{E_{2}}(\hat{x},y_2,t)=\left( \frac{\hat{\nabla}  y_{2}(\hat{x},t)}{\sqrt{1+\vert\nabla y_2\vert^2}},\frac{-1}{\sqrt{1+\vert\nabla y_2\vert^2}},\frac{\partial_{t}y_2(\hat{x},t)}{\sqrt{1+\vert\nabla y_2\vert^2}}\right).
\end{eqnarray}
Clearly both normal vectors point into the set $\{u>t\}$ and are of unit length.
From \eqref{cau5} - \eqref{appnormal3} we deduce the following three chains of equalities:
\begin{eqnarray*}
-\frac{\hat{\nabla} y_{1}(\hat{x},t)}{\sqrt{1+\vert\nabla y_1\vert^2}}
&\overset{\eqref{appnormal2},\eqref{appnormal1}}=&
\frac{\hat{\nabla}u(\hat{x},y_1)}{\sqrt{1+\vert\nabla u(\hat{x},y_1)\vert^2}}\\
&\overset{\eqref{cau5},\eqref{cau8}}=&
\frac{\hat{\nabla}u(\hat{x},y_2)}{\sqrt{1+\vert\nabla u(\hat{x},y_2)\vert^2}}
\overset{\eqref{appnormal1},\eqref{appnormal3}}=
\frac{\hat{\nabla}  y_{2}(\hat{x},t)}{\sqrt{1+\vert\nabla y_2\vert^2}},
\end{eqnarray*}
and
\begin{eqnarray*}
\frac{1}{\sqrt{1+\vert\nabla y_1\vert^2}}
&\overset{\eqref{appnormal2},\eqref{appnormal1}}=&
\frac{\partial_{y}u(\hat{x},y_1)}{\sqrt{1+\vert\nabla u(\hat{x},y_1)\vert^2}}\\
&\overset{\eqref{cau7},\eqref{cau8}}=&
-\frac{\partial_{y}u(\hat{x},y_2)}{\sqrt{1+\vert\nabla u(\hat{x},y_2)\vert^2}}
\overset{\eqref{appnormal1},\eqref{appnormal3}}=
-\frac{-1}{\sqrt{1+\vert\nabla y_2\vert^2}},
\end{eqnarray*}
and
\begin{eqnarray*}
-\frac{\partial_{t}y_1(\hat{x},t)}{\sqrt{1+\vert\nabla y_1\vert^2}}
&\overset{\eqref{appnormal2},\eqref{appnormal1}}=&
\frac{-1}{\sqrt{1+\vert\nabla u(\hat{x},y_1)\vert^2}}\\
&\overset{\eqref{cau8}}=&
\frac{-1}{\sqrt{1+\vert\nabla u(\hat{x},y_2)\vert^2}}
\overset{\eqref{appnormal1},\eqref{appnormal3}}=
\frac{\partial_{t}y_2(\hat{x},t)}{\sqrt{1+\vert\nabla y_2\vert^2}}.
\end{eqnarray*}
From these equalities we deduce
\begin{eqnarray*}
\sqrt{1+\vert\nabla y_1\vert^2}&=&\sqrt{1+\vert\nabla y_2\vert^2},
\end{eqnarray*}
and hence
\begin{eqnarray}\label{appsymm}
\hat{\nabla}y_1+\hat{\nabla}y_2=0\quad&\hbox{and}&\quad\partial_{t}y_1+\partial_{t}y_2=0.
\end{eqnarray}
From \eqref{appsymm} we conclude that there exists a constant $b\in\R$ such that
\begin{eqnarray*}
\frac{y_1(\hat{x},t)+y_2(\hat{x},t)}{2}=b.
\end{eqnarray*}
This proves the theorem.
\end{proof}
\section*{Acknowledgements}
We would like to thank Stan Alama, Lia Bronsard, Almut Burchard, Nicola Fusco, Bob Kohn, Felix Otto, and Peter Sternberg for helpful and interesting discussions on topics related to this work. We also thank the anonymous referee for useful comments.

\medskip

\end{document}